\documentclass[english,11pt]{smfartVScomptage}
\usepackage[latin1]{inputenc}
\usepackage[T1]{fontenc}
\usepackage{lmodern}
\usepackage{smfenum,smfthm,amssymb,mathrsfs,euscript,color}
\usepackage{stmaryrd,mathabx,upgreek,mathdots,mathtools,amscd}
\input{xypic}

\numberwithin{equation}{section} 
\theoremstyle{plain}

\def\ZZ{\mathbb{Z}} 
\def\FF{\mathbb{F}}

\def\A{{\rm A}}
\def\B{{\rm B}}

\def\D{{\rm D}}
\def\E{{\rm E}}
\def\F{{\rm F}}
\def\G{{\rm G}}
\def\H{{\rm H}}

\def\J{{\rm J}}
\def\K{{\rm K}}
\def\L{{\rm L}}
\def\M{{\rm M}}
\def\N{{\rm N}}
\def\P{{\rm P}}
\def\Q{{\rm Q}}
\def\R{{\rm R}}
\def\SS{{\rm S}}
\def\T{{\rm T}}
\def\U{{\rm U}}
\def\V{{\rm V}}
\def\W{{\rm W}}
\def\X{{\rm X}}

\def\Cc{\EuScript{C}}

\def\Ee{\mathscr{E}}
\def\Ff{\mathscr{F}}

\def\Kk{{\rm K}}
\def\Ll{\mathscr{L}}

\def\Oo{\EuScript{O}}

\def\Vv{{\rm V}}

\def\Ga{\Gamma}

\def\a{\alpha} 
\def\b{\beta}
\def\d{\delta}

\def\g{\gamma}
\def\h{\varphi}
\def\k{\kappa}
\def\l{\lambda}
\def\n{\eta}

\def\s{\sigma}
\def\t{\theta}

\def\>{\geqslant}
\def\<{\leqslant}

\def\Hom{{\rm Hom}}
\def\End{{\rm End}}

\def\Mat{{\rm M}}
\def\GL{{\rm GL}}
\def\Gal{{\rm Gal}}
\def\Ker{{\rm Ker}}

\def\tr{{\rm tr}}

\def\Ind{{\rm Ind}}

\def\mult#1{{#1}^{\times}}

\def\kk{\boldsymbol{k}}
\def\ee{\boldsymbol{l}}

\def\ll{\boldsymbol{t}}

\def\bk{\boldsymbol{\k}}
\def\bl{\boldsymbol{\l}}
\def\bt{\boldsymbol{\rho}}

\def\BJ{\boldsymbol{\J}}

\def\0{\boldsymbol{0}}

\def\aa{\mathfrak{a}}
\def\aam{\mathfrak{a}_{{\rm m}}}
\def\bb{\mathfrak{b}}
\def\bbm{\mathfrak{b}_{{\rm m}}}

\def\pp{\mathfrak{p}}

\def\VV{\mathbb{V}}

\def\GG{\EuScript{G}}
\def\JJ{\mathfrak{j}}
\def\HH{\mathfrak{h}}
\def\XX{\mathbb{W}}
\def\YY{\mathbb{Y}}

\def\wn{\mu} 

\def\qlb{{\overline{\mathbf{Q}}_\ell}}
\def\zlb{{\overline{\mathbf{Z}}_\ell}}
\def\flb{{\overline{\mathbf{F}}_{\ell}}}

\def\FC{\R}
\def\AP{\D}
\def\ap{\delta}
\def\ep{\omega}
\def\zero{0}
\def\xx{\mu}

\def\GB{{\bf G}}
\def\aaa{\varepsilon}

\def\GM{{\G}}
\def\PM{{\P}}
\def\UM{{\U}}

\long\def\MSC#1\EndMSC{\def\arg{#1}\ifx\arg\empty\relax\else
     {\par\narrower\noindent%
     2010 Mathematics Subject Classification: #1\par}\fi}

\long\def\KEY#1\EndKEY{\def\arg{#1}\ifx\arg\empty\relax\else
	{\par\narrower\noindent Keywords and Phrases: #1\par}\fi}

\title{Supercuspidal representations of $\GL_n(\F)$ distinguished by a Galois 
involution}

\author{Vincent Sécherre}
\address{Laboratoire de Mathémati\-ques de Versailles\\
UVSQ\\
CNRS\\
Université Paris-Saclay\\
78035, Versailles, France}
\email{vincent.secherre@math.uvsq.fr}

\begin{abstract}
Let $\F/\F_0$ be a quadratic extension of non-Archimedean locally compact 
fields of resi\-dual cha\-rac\-teristic $p\neq2$,
and let $\s$ denote its non-trivial automorphism.
Let $\FC$ be an algebraically clo\-sed field of characteristic different from 
$p$. 
To any cuspidal representation $\pi$ of $\GL_n(\F)$, 
with coef\-fi\-cients in $\FC$, 
such that $\pi^\s\simeq\pi^\vee$ 
(such a representation is said to be $\s$-selfdual)
we associate~a~quadra\-tic extension $\AP/\AP_0$,
where $\AP$ is a tamely ramified extension of $\F$ and 
$\AP_0$ is a tamely ramified~exten\-sion of $\F_0$,
together with a quadratic character of $\AP_0^\times$.
When $\pi$ is supercuspidal,
we give a necessary and sufficient condition,
in terms of these data,
for $\pi$ to be $\GL_n(\F_0)$-distinguished.
When the charac\-teristic $\ell$ of $\FC$ is not $2$,
denoting by $\ep$ the non-trivial $\FC$-character of $\F_0^\times$
trivial on $\F/\F_0$-norms, 
we prove that any $\s$-selfdual supercuspidal $\FC$-representation 
is ei\-ther distinguished or $\ep$-distinguished,
but not both.
In the modular case, that is when $\ell>0$, 
we give exam\-ples of $\s$-selfdual cuspidal~non-supercuspidal 
representations which are not distinguished nor $\ep$-distinguished. 
In the particular case where $\FC$ is the field of complex numbers, 
in which case all cuspidal representations are supercuspi\-dal,
this gives a complete distinction criterion 
for arbitrary~com\-plex cuspidal representations,
as~well as a purely local proof, 
for cuspidal representations,
of the~di\-chotomy and disjunction theorem due to Kable and 
Anandavardhanan--Kable--Tandon, when $p\neq2$.
\end{abstract}

\begin{document} 

\maketitle

\MSC 
22E50, 11F70, 11F85
\EndMSC
\KEY 
Cuspidal representation,
Distinguished representation,
Galois involution,
Modular representation,
$p$-adic reductive group
\EndKEY

\thispagestyle{empty}

\section{Introduction}

\subsection{}

Let $\F/\F_0$ be a separable quadratic extension of non-Archimedean
locally compact fields of~re\-si\-dual cha\-rac\-teristic $p$,
and let $\s$ denote its non-trivial automorphism.
Let $\GB$ be a connected~reduc\-tive group defined over $\F_0$,
let $\G$ denote the locally profinite group $\GB(\F)$ 
equipped with the~na\-tu\-ral action of $\s$
and $\G^\s=\GB(\F_0)$ be the $\s$-fixed~points subgroup.
The study~of those irreducible (smooth) complex representations of $\G$ which 
are $\G^\s$-distinguished,
that is which carry a non-zero $\G^\s$-in\-variant linear form,
goes back to the 1980's.
We refer to \cite{HLR,JY} for the initial~moti\-vation for~dis\-tin\-guished 
representations in a global context, 
and to~\cite{Hakim,Flicker} in a non-Archimedean context.

\subsection{}
\label{PAR12i}

In this work, we will consider the case where $\GB$ is the general
linear group $\GL_n$ for $n\>1$.~We 
thus have $\G=\GL_n(\F)$ and $\G^\s=\GL_n(\F_0)$.
In this case, it is well-known (see \cite{PrasadCMAT90,Flicker,PrasadDMJ01})
that any~dis\-tinguished irreducible complex representa\-tion $\pi$
of $\G$ is $\s$-\textit{selfdual},
that is, the contragre\-dient~$\pi^\vee$ of $\pi$ is isomorphic to 
$\pi^\s=\pi\circ\s$, 
and the space $\Hom_{\G^{\s}}(\pi,1)$~of all
$\G^\s$-invariant linear forms 
on $\pi$ has dimension $1$.
Also, the cen\-tral character of $\pi$ is trivial on~$\F_0^\times$.
This gives us two necessary~con\-di\-tions for an irreducible complex 
representation of $\G$ to be distinguished, and
it~is natural to ask whe\-ther or not they are sufficient. 

\subsection{}

First, let us consider the case 
where $\F/\F_0$ is replaced by a qua\-dratic extension $\kk/\kk_0$
of finite fields of characteristic $p$. 
In this case, Gow~\cite{Gow} proved that an irreducible complex representation 
of $\GL_n(\kk)$ is $\GL_n(\kk_0)$-distinguished if and only if $\pi$ is 
$\s$-selfdual.
(Note that the latter condition automatically implies that the central 
character is trivial on $\kk_0^\times$.)~Be\-sides,
if $p\neq2$,
the $\s$-self\-dual irreducible representations of $\GL_n(\kk)$ are 
those which arise from some irreducible representation of the 
unitary group $\U_n(\kk/\kk_0)$ by base change
(see~Kawana\-ka~\cite{Kawanaka}). 

\subsection{}
\label{guderianintro}

We now go back to the non-Archimedean setting of \S\ref{PAR12i}
and consider a $\s$-selfdual 
irreducible complex represen\-ta\-tion $\pi$ of $\G$
whose central character is trivial on $\F_0^\times$.
When $\pi$ is cuspidal~and
$\F/\F_0$ is unramified, Prasad \cite{PrasadDMJ01}
proved that,
if $\ep=\ep_{\F/\F_0}$ denotes the non-trivial character of $\F_0^\times$ 
trivial on $\F/\F_0$-norms, 
then $\pi$ is either distinguished or $\ep$-distinguished,
the latter case meaning that the complex vector space 
$\Hom_{\G^{\s}}(\pi,\ep\circ\det)$ is non-zero.

When $p\neq2$ and
$\pi$ is an \textit{essentially tame} cuspidal representation, 
that is, when the number~of unramified cha\-rac\-ters $\chi$ of $\G$ 
such that $\pi\chi\simeq\pi$ is prime to $p$, 
Hakim and Murnaghan~\cite{HMIMRN}~gave {sufficient} con\-ditions 
for $\pi$ to be distinguished.
These conditions are stated in terms of admissible pairs~\cite{Howe}, 
which parametrize essentially tame cuspidal complex representations of 
$\G$ (\cite{Howe,BHETLC1}).
Note that they assume $\F$ has characteristic $0$, 
but their approach also works in characteristic $p$.

When $\pi$ is a discrete series representation and
$\F$ has characteristic $0$,
Kable~\cite{Kable} proved that if $\pi$ is $\s$-selfdual,
then it is either~dis\-tin\-guished or $\ep$-distinguished:
this is the \textit{Dichotomy~Theorem}.
In addition, Anandavardhanan, Kable and Tandon \cite{AKT} 
proved that $\pi$ can't be both dis\-tinguished and 
$\ep$-distinguished:
this is the \textit{Disjunction~Theorem}.
The proofs use global arguments, which~is why 
$\F$ were assu\-med to have characteristic $0$,
and the Asai $\L$-function of $\pi$.
However, these~results still hold 
when $\F$ has characteristic $p\neq2$, 
as is explained in \cite{AKMSS} {Appendix A}.
Note that:
\begin{itemize}
\item
the dis\-junc\-tion theorem 
implies that the sufficient~con\-di\-tions 
given by Hakim and Murna\-ghan in~\cite{HMIMRN}~in 
the essentially tame cuspidal case 
are necessa\-ry con\-di\-tions as well;
\item 
the dichotomy theorem implies that,
when $n$ is odd, any $\s$-selfdual discrete
series representa\-tion of $\G$
with central character trivial on $\F_0^\times$ is automatically 
distinguished.
Indeed, an $\ep$-dis\-tinguished irreducible
representation has a central character
whose restriction to $\F_0^\times$ is $\ep^n$.
\end{itemize}

When $p\neq2$ and
$\pi$ is cuspidal of level zero 
-- in particular $\pi$ is essentially tame -- 
Coniglio~\cite{Coniglio} gave a~neces\-sa\-ry and sufficient condition of
distinction in terms of admissible pairs.
Her proof~is purely local, 
and does not use the disjunction theorem.~(In~fact,
she considers the more general case where $\GB$ is an inner form
of $\GL_{n}$ over $\F_0$,
and~repre\-sen\-ta\-tions of level zero of $\GB(\F)$ 
whose local Jacquet-Langlands transfer to $\GL_n(\F)$ is cus\-pi\-dal.)

If one takes the classification of distinguished cuspidal representations 
of general linear groups for granted,
and assuming $\F$ has characteristic $0$,
Anandavardhanan and Rajan~\cite{AR,AnandIMRN08}
clas\-si\-fied all distinguished discrete series representations of $\G$ in terms 
of~the distinction~of their cus\-pidal support~(see~also \cite{MatringeDS})
and Matringe~\cite{MatringeG,MatringeU} classified distinguished generic, 
as well as distingui\-shed unitary, re\-pre\-sen\-ta\-tions of $\G$ in terms
of the Langlands 
classifica\-tion.
This provides a class of representations 
for which the dicho\-to\-my and disjunction theorems do not hold: 
some $\s$-self\-dual generic irreducible re\-presentations are 
nor~dis\-tinguished nor $\ep$-distinguished, and some are both.

In \cite{Gurevich}
Max~Gurevich extended the dichotomy theorem to the class of
\textit{ladder} representations,
which contains all discrete series representations: 
a $\s$-selfdual ladder representation of $\G$ is~either
distinguished~or $\ep$-distinguished, but it may be both
(see \cite{Gurevich} Theorem 4.6).
Here again, $\F$ is~as\-su\-med to have characteristic $0$.

Finally,
one can deduce from these works 
the connection between distinction for generic~irre\-du\-ci\-ble
re\-pre\-sen\-ta\-tions of $\G$
and base change from a quasi-split unitary group:
see \cite{GanRaghuram} Theorem 6.2,
\cite{AnandPrasad} Theorem 2.3 or \cite{GMM}.

\subsection{}
\label{introQ}

The discussion above leaves us with an open problem 
about cuspidal representations:
find a dis\-tinction criterion for an \textit{arbitrary} 
$\s$-selfdual cuspidal~re\-pre\-sentation $\pi$,
with no assumption~on the characteristic of $\F$, 
on the ramification of $\F/\F_0$, 
on $n$ nor on
the torsion number of $\pi$
(that~is,
the number of unramified cha\-rac\-ters $\chi$ 
such that $\pi\chi\simeq\pi$).

In this paper,
\textit{assuming that $p\neq2$},
we propose an approach which allows us to generalize both 
Hakim-Murnaghan's and Coniglio's distinction criterions to \textit{all} 
cuspidal irreducible complex~re\-pre\-sentations of $\G$ and which works:
\begin{itemize}
\item 
  with no assumption on the characteristic of $\F$
  (apart from the assumption ``$p\neq2$''),
\item
with purely local methods,
\item
not only for complex representations, but more generally 
for representations with coefficients in an algebraically closed 
field of arbitrary characteristic $\ell\neq p$. 
\end{itemize}

We thus give a complete solution to the problem above for 
cuspidal complex representa\-tions when $p\neq2$.
We actually do more:
we solve this problem in the larger context of 
\textit{supercuspidal}~re\-presen\-ta\-tions
with coefficients in an algebraically closed 
field of arbitrary characteristic $\ell\neq p$. 

\subsection{}

First, let us say a word about the third item above.
The theory of smooth~representations~of $\GL_n(\F)$ 
with coefficients in an algebraically closed 
field of characteristic $\ell\neq p$
has been initia\-ted by Vignéras \cite{Vigb,Vigs}
in view to extend the local Langlands programme to representations~with 
coefficients in a field -- or a ring -- as general as possible 
(see for instance \cite{Vigi,HelmMoss}). 
Inner~forms have also been taken into account (\cite{MSc,SSb})
and the congruence properties of the~local Jacquet--Langlands correspondence 
have been studied in \cite{Datj,MSj}.
It is thus natural to ex\-tend the study of distinguished representations to 
this wider context, where the field of complex numbers is~re\-pla\-ced by a more 
general field.
Very little has been done about distinction of 
modular re\-pre\-sen\-ta\-tions so far:
a first study can be found in \cite{SVdml}.

An important phenomenon in the modular~case,
that is when $\ell>0$, is that 
a~cuspi\-dal~repre\-sentation $\pi$ may occur as a 
subquotient of a proper parabo\-lically~in\-du\-ced representation 
(see~\cite{VigCMAT89} Corollaire~5).
When this is not the case, that is when $\pi$ does not occur as a 
subquotient of a~pro\-per parabo\-lically~in\-du\-ced representation, 
$\pi$~is said to be \textit{super\-cus\-pi\-dal}.

\subsection{}

\textit{From now on, we fix an algebraically closed field $\FC$ of arbitrary 
char\-ac\-teristic $\ell\neq p$, and~consi\-der 
irreducible smooth representations of $\G$ with coefficients in $\FC$.}
Note that $\ell$ can be $0$.

We first notice that, as in the complex case,
any distinguished irreducible representation $\pi$~of $\G$ with coefficients 
in $\FC$ is $\s$-selfdual and $\Hom_{\G^\s}(\pi,1)$ is $1$-dimensional
(see Theorem \ref{FP}).

We prove that, if $\ell\neq2\neq p$,
the dichotomy~and disjunction theorems hold for all
\textit{supercuspidal} re\-pre\-sen\-tations with coefficients in $\FC$.
In~particular, when $\FC$ is the field of com\-plex numbers,~in
which case any cuspidal representation is super\-cus\-pi\-dal, 
we get a 
purely local proof of the~dicho\-tomy and~disjunc\-tion theorems
for cuspidal representations in the case where $p\neq2$.

When $\ell=2\neq p$, 
in which case there is no character of order $2$ on $\F_0^\times$, 
the dichotomy~theo\-rem takes a simplified form: 
any $\s$-selfdual super\-cuspidal representation is distinguished.
Let~us~sum\-ma\-ri\-ze this first series of results in the theorem below.

\begin{theo}[Theorem \ref{MAINTH1DD}]
  \label{MAINTH1DDintro}
  Suppose that $p\neq2$,
  and let $\pi$ be a $\s$-selfdual supercuspidal irreducible
  $\FC$-representation of $ \G$.
\begin{enumerate}
\item
If $\ell=2$, then $\pi$ is distinguished.
\item
If $\ell\neq2$, 
then $\pi$ is either distinguished or $\omega$-distinguished, 
but not both.
\end{enumerate}
\end{theo}

In the modular case, for $\ell>2$,
we give exam\-ples of $\s$-selfdual cuspidal, 
non-super\-cus\-pi\-dal~re\-presentations which are not distinguished 
nor $\ep$-distinguished 
(see Remarks \ref{poussinnoir} and \ref{mabille}). 

\begin{center}
$\bullet$
\end{center}

\textit{From now on, until the end of this introduction,
we will assume that $p\neq2$.}

\subsection{}
\label{intro5}

The dichotomy and disjunction theorem
stated in Theorem \ref{MAINTH1DDintro} 
relies on~a distinction criterion,
which we state in Theorem \ref{MAINTH1intro}.
The basic idea~is that we canonical\-ly asso\-cia\-te to any $\s$-selfdual 
supercuspidal representation $\pi$ of $\G$ 
a finite~ex\-ten\-sion $\AP$ of $\F$
equip\-ped with an $\F_0$-involution extending $\s$
and a quadratic character $\ap_0$ of the fixed points 
of $\AP^\times$~; 
it~is these data which govern the distinction of $\pi$.
The character $\ap_0$ refines the information given~by the central character
of $\pi$ in the sense that they coincide on $\F_0^\times$,
the latter one being not enough in general to determine whether $\pi$ is 
distinguished or not.
To state our distinction criterion,
let us write $\AP_0$ for the fixed points subfield of $\AP$.

\begin{theo}[Theorem \ref{MAINTH1}]
  \label{MAINTH1intro}
A $\s$-selfdual supercuspidal $\FC$-representation of $\G$ 
is distinguished if and only if:
\begin{enumerate}
\item
either $\ell=2$,
\item
or $\ell\neq2$ and:
\begin{enumerate}
\item
if $\F/\F_0$ is ramified, $\AP/\AP_0$ is unramified and
$\AP_0/\F_0$ has odd ramification order, 
then the character $\ap_0$ is non-trivial,
\item
otherwise, $\ap_0$ is trivial.
\end{enumerate}
\end{enumerate}
\end{theo}

Even in the complex case, 
this is the first time a necessary and~suf\-ficient distinction 
criterion is exhibited for 
an \textit{arbitrary} cuspidal representation of $\GL_n(\F)$
in odd residual characteristic,
in the spirit of \cite{HMIMRN,HMIMRP,Coniglio}.

\subsection{}

The starting point of our strategy for proving Theorem \ref{MAINTH1intro} is 
to use Bushnell-Kutzko's cons\-truc\-tion of cuspidal representa\-tions of 
$\G$ via compact induction. 
This construction, elaborated~in the complex case by Bushnell and 
Kutzko~\cite{BK}, has been extended to the modular case 
by~Vi\-gné\-ras \cite{Vigb} and Minguez--Sécherre~\cite{MSt}.
There is a family of pairs $(\BJ,\bl)$, made of certain com\-pact 
mod centre open subgroups $\BJ$ of $\G$ and certain irreducible
representations $\bl$ of $\BJ$, such that:
\begin{itemize}
\item
for any such pair $(\BJ,\bl)$, the compact induction of $\bl$ to $\G$
is irreducible and cuspidal;
\item
any irreducible cuspidal representation of $\G$ occurs in this way,
for a pair $(\BJ,\bl)$ uniquely~de\-termined up to $\G$-conjugacy. 
\end{itemize}
Such pairs are called \textit{extended maximal simple types} in \cite{BK}, 
which we will abbreviate to \textit{types}~for simplicity.
We need to give more details about the structure of these types:
\begin{enumerate}
\item 
The group $\BJ$ has a unique maximal compact subgroup $\J$,
and a unique maximal normal pro-$p$-group $\J^1$.
\item 
There is a group isomorphism $\J/\J^1\simeq\GL_{m}(\ee)$ for some 
divisor $m$ of $n$ and finite extension $\ee$ of the residual field 
$\kk$ of $\F$.
\item 
The restriction of $\bl$ to $\J^1$ is made of copies of a single 
irreducible representation $\n$, which extends 
(not uniquely, nor canonically) to $\BJ$.
\item 
Given a representation $\bk$ of $\BJ$ extending $\n$,
there is a unique irreducible representation~$\bt$ of $\BJ$ trivial on $\J^1$ 
such that $\bl$ is isomorphic to $\bk\otimes\bt$,
and $\bt$ restricts irreducibly to $\J$.
\item 
The representation of $\GL_{m}(\ee)$ obtained by
restricting $\bt$ to $\J$ is~cus\-pidal.
\end{enumerate}

The integer~$m$,
called the \textit{relative degree} of $\pi$, 
is uniquely determined by $\pi$.
There~is another type-theoretical invariant called the 
\textit{tame parameter field} 
of $\pi$: 
this is a tamely ramified extension $\T$ of $\F$,
unique\-ly determined up to $\F$-isomorphism,
whose degree divides $n/m$ and whose residual field is 
$\ee$ (see \cite{BHEffectiveLC} for more details). 
Note that $\pi$ if essentially tame if and only if $[\T:\F]=n/m$.

\subsection{}

Now consider a $\s$-selfdual cuspidal $\FC$-representation $\pi$ of $\G$. 
The starting point of all~our~work is \cite{AKMSS} {Theorem 4.1}, 
which asserts that among all the types contained in $\pi$,
there is a type $(\BJ,\bl)$ which is $\s$-selfdual, 
that is $\BJ$ is $\s$-stable and 
$\bl^\vee$ is isomorphic to $\bl^\s$. 
Moreover, 
the tame~para\-meter field $\T$ of $\pi$ is equipped with an 
$\F_0$-involution. 
If $\T_0$ denotes the fixed points subfield of $\T$,
then $\T/\T_0$ is a quadratic extension, 
uniquely determined up to $\F_0$-isomorphism.
The invariants $m$ and $\T/\T_0$ associated with $\pi$
will play a central role in what follows. 

First,
the following result says that the distinction of $\pi$ can be detected by a 
$\s$-selfdual type. 

\begin{theo}[Theorem \ref{DSTT}]
\label{DSTTintro}
Let $\pi$ be a $\s$-selfdual cuspidal representation of $\G$.
Then~$\pi$ is dis\-tin\-guished~if and only if it contains 
a dis\-tin\-guished $\s$-selfdual type,
that is a $\s$-selfdual type $(\BJ,\bl)$ 
such that $\Hom_{\BJ\cap\G^\s}(\bl,1)$ is non-zero.
\end{theo}

The proof of this theorem 
-- which occupies Section~\ref{SEC6} --
is the most technical part of the paper:
starting with a $\s$-selfdual type $(\BJ,\bl)$ contained in $\pi$
and $g\in\G$, one has to prove that, if the~type $(\BJ^g,\bl^g)$ 
is distinguished, then it is $\s$-selfdual, that is $\s(g)g^{-1}\in\BJ$. 
First, one determines the~set of the $g\in\G$ such that $\n^g$ is 
distinguished, as well as the dimension of the space of 
invariant~linear forms (Paragraphs \ref{P51} to \ref{P53}); 
then, one analyzes the dis\-tinc\-tion of $\bk^g$
(Paragraphs \ref{Felicite} and \ref{DCT}); 
one obtains the final statement by using the cuspidality of the representation 
of $\GL_m(\ee)$ induced by $\bt$ (see Theorem \ref{DCT1}).

\subsection{}

When $\T$ is unramified over $\T_0$, 
the $\s$-selfdual types contained in $\pi$ form a single $\G^\s$-conjugacy 
class. 
When $\T$ is ramified over $\T_0$, 
the $\s$-selfdual types contained in $\pi$ form $\lfloor m/2\rfloor+1$ 
different
$\G^\s$-conjugacy classes, characterized by an integer 
$i\in\{0,\dots,\lfloor m/2\rfloor\}$ called the \textit{index} of the~class.
Since the space $\Hom_{\G^\s}(\pi,1)$ has dimension $1$, 
only one of these conjugacy classes can contribute to distinction:
we prove that it is the one with maximal index.
This gives us the following~refine\-ment of Theorem \ref{DSTTintro}.

\begin{prop}[Corollary \ref{zarathoustra} and Proposition \ref{MAINTHM3r}]
\label{classesdetypesstables7intro}
Let $\pi$ be a $\s$-selfdual cuspidal representa\-tion of $\G$.
Let $m$ be its relative degree and $\T/\T_0$ be its associated quadratic 
extension. 
\begin{enumerate}
\item
If $\T$ is unramified over $\T_0$, then $\pi$ is distinguished if and only if 
any of its $\s$-selfdual types is distinguished.
\item
If $\T$ is ramified over $\T_0$, then $\pi$ is distinguished if and only if 
any of its $\s$-selfdual types of index $\lfloor m/2\rfloor$ is distinguished.
\end{enumerate}
\end{prop}

Note that this proposition is proved in \cite{AKMSS} 
in a~diffe\-rent manner, based on a result of Ok~\cite{Ok}.
However, the proof given in the present article is more likely to generalize 
to other situations.

When $\T/\T_0$ is ramified, one can be more precise
(see Proposition \ref{MAINTHM3r}):
if $\pi$ is distinguished,~$m$ is either even or equal to $1$. 
It is not difficult to construct $\s$-selfdual 
cuspidal~re\-pre\-sen\-ta\-tions of $\G$
such that $\T/\T_0$ is ramified and $m>1$ is odd:
such cuspidal representations are not distin\-gui\-shed nor 
$\omega$-distinguished (see Remark \ref{poussinnoir}). 

In the case where $\T/\T_0$ is unramified,
$m$ is odd for any super\-cuspidal $\s$-selfdual representation
(see~Pro\-position \ref{MAINTHM3nr}).
This does not hold
for $\s$-selfdual cuspidal representations (which is easy~to see),
and this does not even hold for distinguished cuspidal representations:
Kurinczuk and~Ma\-trin\-ge recently pro\-ved that,
when $\F/\F_0$ is unramified and $n=\ell=2$,
any~$\s$-selfdual~non-super\-cuspidal cuspidal representation of
$\GL_2(\F)$ of level zero (thus of rela\-tive degree~$2$)
is distingui\-shed.

\subsection{}
\label{P110}

As in the previous paragraph,
$\pi$ is a $\s$-selfdual cuspidal $\FC$-representation of $\G$.
The following definition will be convenient to us 
(see Remark \ref{cariocarema} for the connection with the usual notion~of
a generic representation).

\begin{defi}[Definition \ref{carioca}]
A $\s$-selfdual type $(\BJ,\bl)$ in $\pi$ is 
\textit{generic} if either $\T/\T_0$ is unramified, 
or $\T/\T_0$ is ramified and this type has index $\lfloor m/2\rfloor$. 
\end{defi}

Proposition \ref{classesdetypesstables7intro} 
thus says that, up to $\G^\s$-conjugacy, a $\s$-selfdual 
cuspidal~re\-pre\-sentation $\pi$~con\-tains a unique generic $\s$-selfdual type,
and that $\pi$ is distinguished if and only if such a 
type~is~dis\-tin\-guished 
(see Theorem \ref{genericDSTT}).
This uniqueness property is crucial to the proof of the disjunction part
of Theorem \ref{MAINTH1DDintro}. 

Let us~fix a generic $\s$-selfdual type $(\BJ,\bl)$ in $\pi$.
Recall that, by construction, 
$\bl$ can be~decom\-po\-sed (non ca\-no\-ni\-cal\-ly)
as $\bk\otimes\bt$.
However, not any of these decompositions are suitable for our purpose. 
It is not difficult to prove that $\bk$ can be chosen to be $\s$-selfdual, 
but this is not~enough: 
we need to prove that $\bk$ can be chosen to be both $\s$-selfdual 
and distinguished. 
The stra\-te\-gy of the proof depends on the ramification of $\T$ over $\T_0$.
This is why we treat separately the ramified and the unramified cases,
in Sections \ref{SECDISTTR}
and \ref{SECDIST}, respectively.

The easiest case is when $\T/\T_0$ is ramified. 
Using the fact that $m$ is either even or equal to~$1$,
we prove that $\bk$ can be chosen to be distinguished
by adapting a result of Matringe~\cite{Matringe}
(which we do in Paragraph \ref{anais}).

When $\T/\T_0$ is unramified, 
the existence of a distinguished $\bk$ is more difficult to establish.~Our
proof~re\-qui\-res $\pi$ to be supercuspidal, since in that case $m$ is 
known to be odd,
thus $\GL_m(\ee)$~has $\GL_m(\ee_0)$-distinguished
supercuspidal representations in characteristic $0$,
where $\ee_0$ is the residual field of $\T_0$
(see the proof of Proposition \ref{HoMcKay}).

In~both cases, a distinguished $\bk$ is automatically $\s$-selfdual,
and $\pi$ is distinguished if and~only if $\bt$ is distinguished.
Considering $\bt$ as a ($\s$-selfdual) level zero type,
we are then reduced to the level zero case, which has been treated by 
Coniglio in~the complex case. 
We thus have to~extend her results to the modular case, 
which we know how~to~do when $\pi$ is supercuspidal only.

To sum\-marize, we need the assumption that $\pi$ is supercuspidal 
in Theorems \ref{MAINTH1DDintro} and \ref{MAINTH1intro}
for two reasons: 
for the existence of a distinguished $\bk$ in the case
when $\T/\T_0$ is unramified, 
and for the level zero case. 

\subsection{}
\label{P111}
\label{P12}

To study the distinction of $\bt$ when $\pi$ is supercuspidal, 
we use admissible pairs of level zero as in Coniglio \cite{Coniglio}. 
We attach to $\bt$ a pair $(\AP/\T,\ap)$ 
made of an unramified extension of~degree~$m$ equipped with an involutive 
$\T_0$-algebra homomorphism, non-trivial on $\T$,
denoted by $\s$, together with a cha\-rac\-ter $\ap$ of $\AP^\times$ such that 
$\ap\circ\s=\ap^{-1}$.
(See Paragraphs \ref{pullfroid}, \ref{P45}
although the result is presented in a different way there.)

However, the distinguished representation $\bk$ of Paragraph \ref{P110} is not 
unique in general, thus neither $\bt$ nor $\ap$ are.
Write $\AP_0$ for the $\s$-fixed points of $\AP$,
and $\ap_0$ for the restriction of $\ap$ to $\AP_0^\times$.
This is a quadratic character, trivial on $\AP/\AP_0$-norms. 
We prove in Proposition \ref{rummo} that the~pair $(\AP/\AP_0,\ap_0)$ 
is uniquely determined by $\pi$ up to $\F_0$-isomorphism.
This is the one occuring in our distinction criterion theorem \ref{MAINTH1intro}.

It remains to explain our strategy to prove the distinction criterion for 
$\bt$, in the modular case, in terms of the character $\ap_0$, 
as well as the dimension of the space of invariant linear forms.
This depends on the ramification of $\T/\T_0$.

The easiest case is when $\T/\T_0$ is unramified.
In this case, 
we are reduced to studying the~dis\-tinc\-tion of supercuspidal representations 
of $\GL_m(\ee)$ by $\GL_m(\ee_0)$.
That any distin\-gui\-shed ir\-re\-du\-ci\-ble representation is 
$\s$-selfdual follows from a finite and $\ell$-modular version of
Theorem~\ref{FP}~(see Remark~\ref{bananejaune}).
For the converse statement, we use a lifting argument to characteristic $0$, 
based~on the fact that any $\s$-selfdual supercuspidal
$\flb$-representation has a $\s$-selfdual $\qlb$-lift,
where~$\qlb$~is an algebraic closure of the field of $\ell$-adic
numbers and $\flb$ is its residual field.
This does not~hold for $\s$-selfdual non-supercuspidal representations:
by Remark \ref{mabille},
there are $\s$-selfdual cuspidal~re\-presentations, with $m$ even,
which are not distinguished.

In the case where $\T/\T_0$ is ramified,
we are reduced to studying the distinction of supercuspidal representations 
of $\GL_m(\ee)$ by either $\GL_1(\ee)$ if $m=1$, 
or $\GL_{r}(\ee)\times\GL_{r}(\ee)$ if $m=2r$ is~even.~It is more difficult,
as we do not have an analogue of Theorem \ref{FP}.
Our proof relies on the~structure~of the projective envelope 
of a supercuspidal representation of $\GL_m(\ee)$,
as well as a lifting argument to characteristic $0$.
We prove that a supercuspidal representation is distinguished if and 
only~if~it is selfdual.
Unlike the complex case, 
one can find $\s$-self\-dual cuspidal representations, 
with $m>1$ odd,
which are not distinguished (see Remark \ref{carracci}).

In both cases, 
we prove that a $\s$-selfdual supercuspidal representation of 
$\GL_m(\ee)$ is distingui\-shed if any only if it admits a distinguished lift 
to characteristic $0$.
We conclude by the following theorem.

\begin{theo}[Theorem \ref{MAINTH2}]
Let $\pi$ be a $\s$-selfdual supercuspidal representation of the group $\G$ 
with coefficients in $\flb$.
\begin{enumerate}
\item
The representation $\pi$ admits a $\s$-selfdual supercuspidal lift to $\qlb$.
\item
Let $\widetilde{\pi}$ be a $\s$-selfdual lift of $\pi$,
and suppose that $\ell\neq2$.
Then $\widetilde{\pi}$ is distinguished if and only if 
$\pi$ is distinguished.
\end{enumerate}
\end{theo}

\subsection{}
\label{pis2}

In this paragraph, we discuss the assumption $p\neq2$.
In Section \ref{FFC}, in which we study the~finite field case,
we assume that $p\neq2$ in Paragraphs \ref{anais} and \ref{etrange} only:
see Remark \ref{camille}.
Note~that~nor Gow's results \cite{Gow} nor their modular version
established in \S\ref{level0nr1} require $p$ to be odd.
The~same~is true of the results of Prasad and Flicker -- 
as well as their modular version proved in Section~\ref{COE} -- 
asserting that 
any distinguished irreducible representation $\pi$~of $\G$
is $\s$-selfdual and $\Hom_{\G^\s}(\pi,1)$ has dimension $1$.

From Paragraph \ref{S2} the assumption $p\neq2$ is~cru\-cial
(as in \cite{HMIMRN,HMIMRP} and to a lesser extent~\cite{Coniglio}).
We use at many places,
in particular in Section \ref{SEC6} 
and in the proof of the $\s$-selfdual type~theo\-rem \ref{PIMAIN}, 
the fact that the first cohomology set of 
$\Gal(\F/\F_0)$ in a pro-$p$-group is trivial.
I do not know whether or not Theorem \ref{PIMAIN} still holds when $p=2$.

I also do not now whether the dichotomy and disjunction theorems
hold when $\F$ has characte\-ris\-tic $2$.
The only exception is Prasad's dichotomy theorem~\cite{PrasadDMJ01}
for cuspidal complex representations when $\F/\F_0$ is unramified,
which remains the only known distinction criterion for cuspidal 
representations in arbitrary residual characterictic.
Note that Prasad's approach does not work in the modular case,
for \cite{PrasadDMJ01} Theorem 1 does not hold in characteristic $\ell>0$.

\subsection{}

Finally,
let us mention that the methods developed in this paper are 
expected to generali\-ze to other groups.
The distinction of supercuspidal representations of $\GL_n(\F)$ 
by a~unitary~group~is currently explored by Jiandi~Zou in his ongoing
PhD thesis at Université de Versailles St-Quentin.

\subsection*{Acknowledgements}

I thank D.~Prasad for his invitation to TIFR Mumbai
where part of this work has been~written, 
and U.~K.~Anandavardhanan for his invitation to give a talk about this work 
at IIT Bombay, in January 2018. 
I also thank N.~Matringe for stimulating discussions about this work,
as well as an anonymous referee for valuable comments and suggestions.

This work has been partly supported by the 
Laboratoire International Associé
Centre Franco-Indien pour les Mathémati\-ques, 
and by the Institut Universitaire de France.

\section{The finite field case}
\label{FFC}

The aim of this section is to extend to the modular case some results about 
distinction of~cus\-pidal representations of $\GL_n$ over a finite field
which are know in the complex case only.
These results will be needed in Sections \ref{SECDISTTR0} and \ref{SECDIST},
but this section can also be read independently from the rest of the
article.

In this section,
$\kk$ is a finite field whose characteristic is an arbitrary prime number $p$.
In~Para\-graphs \ref{anais} and \ref{etrange},
we will assume that $p$ is odd.
Let $q$ denote~the cardinality of~$\kk$.

Let $\FC$ be an algebraically closed field of characteristic different 
from $p$, denoted $\ell$. 
(Note that $\ell$ can be $0$.)
We say that we are in the ``modular case'' when we consider the case where 
$\ell>0$. 
By \textit{representation} of a finite group, 
we mean a representation on an $\FC$-vector space.

Given~a re\-pre\-sentation $\pi$ of a finite group $\G$, 
we write $\pi^\vee$ for the contra\-gredient of $\pi$.
Given~a subgroup $\H$ of $\G$,
we say that $\pi$ is $\H$-\textit{distinguished} if the 
space $\Hom_{\H}(\pi,1)$ is non-zero,
where $1$ denotes the trivial character of $\H$. 

Given $n\>1$,
an irreducible representation of $\GL_n(\kk)$ is said to be \textit{cuspidal} 
if it has no non-zero invariant vector under the unipotent radical of any
proper parabolic subgroup or,
equivalently,~if it does not occur as a subrepresentation of any proper 
parabolically~indu\-ced 
representation.~It~is \textit{supercuspidal} 
if it does not occur as a subquotient of~a proper parabolically~induced
represen\-ta\-tion.
When $\FC$ has characteristic $0$,
any cuspidal representation is supercuspidal. 

When $\ell>0$,
we denote by $\qlb$ an algebraic closure of the field 
of $\ell$-adic numbers, 
by $\zlb$ the ring of $\ell$-adic integers in $\qlb$
and by $\flb$ the residue field of $\zlb$.
We refer to \cite{Serre} \S 15 for a definition of the reduction mod 
$\ell$~of a $\qlb$-re\-pre\-sentation of a finite group.

\subsection{Parametrisation of supercuspidal representations}
\label{claudine}

For the results stated in this paragraph, we refer to 
\cite{Green,Dipper,DJ,James}
(see also \cite{Vigb} III.2, and \cite{MSf} \S 2.6).

Let $\ll/\kk$ be an extension of degree $n\>1$ of finite fields of 
characteristic $p$.
Fix a $\kk$-embedding of $\ll$ in the matrix algebra $\Mat_n(\kk)$,
and consider $\ll^\times$ as a maximal torus in $\GL_{n}(\kk)$.
An $\FC$-character $\xi$ of $\ll^\times$ is said to be $\kk$-\textit{regular}
if the characters $\xi$, $\xi^q,\dots,\xi^{q^{n-1}}$ are all distinct. 

Let $\xi$ be a $\kk$-regular $\FC$-character of $\ll^\times$.
By Green \cite{Green} when $\FC$ has characteristic $0$ 
and~James \cite{James} when $\FC$ has positive charac\-te\-ris\-tic
$\ell\neq p$, there is a supercuspidal irreducible~re\-pre\-sentation
$\rho_\xi$ of $\GL_n(\kk)$, uniquely determined up to isomorphism, such that:
\begin{equation}
\label{tracerhoxi}
\tr\ \rho_\xi(g) = (-1)^{n-1}\cdot\sum\limits_{\g} \xi^\g(g)
\end{equation}
for all $g\in\ll^\times$ with irreducible characteristic polynomial,
where $\g$ runs over $\Gal(\ll/\kk)$.
This~indu\-ces a surjective map:
\begin{equation}
\label{GreenJames}
\xi \mapsto \rho_{\xi}
\end{equation}
between $\kk$-regular characters of $\ll^\times$ 
and isomorphism classes of supercuspidal irreducible
represen\-tations of $\GL_n(\kk)$, whose fibers are 
$\Gal(\ll/\kk)$-orbits.

Suppose that $\FC$ is the field $\qlb$.
In the next proposition, we record the main properties of the reduction 
mod $\ell$ of supercuspidal $\qlb$-representations of $\GL_{n}(\kk)$.

\begin{prop}[\cite{Dipper,DJ,James}]
\label{saintfargeau}
Let $\xi$ be a $\kk$-regular $\qlb$-character of $\ll^\times$
and $\rho$ be the super\-cuspidal irredu\-cible
repre\-sentation which~cor\-res\-ponds to it.
\label{valencia}
\begin{enumerate}
\item 
The reduction mod $\ell$ of $\rho$, denoted $\overline{\rho}$, is irreducible
and cuspidal.
\item
The representation $\overline{\rho}$ is supercuspidal if and only if the 
reduction mod $\ell$ of $\xi$, denoted~$\overline{\xi}$, is $\kk$-regular. 
\item
When $\overline{\xi}$ is $\kk$-regular,
the super\-cuspidal irreducible $\flb$-re\-pre\-sentation of $\GL_n(\kk)$ 
which~cor\-res\-ponds to it is $\overline{\rho}$.
\end{enumerate}
\end{prop}

Moreover,
for any cuspidal irreducible $\flb$-re\-pre\-sentation 
$\pi$ of $\GL_n(\kk)$,~there is a supercuspidal irre\-ducible
$\qlb$-re\-pre\-sentation ${\rho}$ of $\GL_n(\kk)$
whose reduction mod $\ell$ is $\pi$.
Such a represen\-ta\-tion~${\rho}$ is said to be a \textit{lift} of 
$\pi$. 

\begin{rema}
\label{claudinerem}
Here are a couple of additional properties which we will use at various places.
\begin{enumerate}
\item 
The identity \eqref{tracerhoxi} shows that,
if $\xi$ is a $\kk$-regu\-lar character of $\ll^\times$
and $\rho$ is the supercuspi\-dal~re\-presentation
corresponding to $\xi$ by \eqref{GreenJames}, 
then its contragredient $\rho^\vee$ corresponds to $\xi^{-1}$.
\item
Let $\iota:\FC\to\FC'$ be an embedding of algebraically closed fields of 
characteristic $\ell\neq~p$.~Then 
any irreducible $\FC'$-representation $\pi'$ of $\GL_n(\kk)$ is
isomorphic to $\pi\otimes\FC'$ for a uniquely determi\-ned 
irreducible~$\FC$-re\-pre\-sentation $\pi$ of $\GL_n(\kk)$, 
which follows from the fact that, since~$\GL_n(\kk)$~is~fini\-te,
the trace of $\pi'$ takes values in $\iota(\FC)$.
Given a subgroup $\H$ of $\GL_n(\kk)$, 
the representation $\pi$ is cuspidal 
(respectively supercuspidal, $\H$-distinguished)
if and only if $\pi'$ is cuspidal 
(respectively supercuspidal, $\H$-distinguished).
Moreover, by \eqref{tracerhoxi},
if $\pi$ is supercuspidal and corresponds to~the $\kk$-re\-gu\-lar 
$\FC$-character~$\xi$, 
then $\pi'$ corresponds to the $\kk$-regular $\FC'$-character
$\xi'=\iota\circ\xi$.
\end{enumerate}
\end{rema}

\subsection{The Galois case}
\label{level0nr1}

Recall that $p$ is an arbitrary prime number.
Let $\kk/\kk_0$ be a quadratic extension of finite fields of characteristic 
$p$.
Write $\s$ for the non-trivial $\kk_0$-automorphism of $\kk$,
and $q_0$ for the cardinality of $\kk_0$.
We thus have $q_0^2=q$.

If $\pi$ is an irreducible representation of $\GL_n(\kk)$,
we write $\pi^\s$ for the representation $\pi\circ\s$,
and we say that $\pi$ is \textit{$\s$-selfdual} if
$\pi^\s$, $\pi^\vee$ are isomorphic.

\begin{lemm}
\label{L71}
Let $n\>1$ be a positive integer.
\begin{enumerate}
\item
If there is a $\s$-selfdual supercuspidal irredu\-ci\-ble representation 
of $\GL_n(\kk)$, then $n$ is odd.
\item
Suppose that $\FC$ has characteristic $0$ and $n$ is odd.
Then there is a $\s$-selfdual supercuspidal irredu\-ci\-ble representation of 
$\GL_n(\kk)$.
\end{enumerate}
\end{lemm}

\begin{proof}
Let $\xi$ be a $\kk$-regular character of $\ll^\times$,
and let $\rho$ denote the supercuspidal
irredu\-ci\-ble~re\-pre\-sentation of $\GL_n(\kk)$ corresponding to it by
\eqref{GreenJames}.
The identity \eqref{tracerhoxi} shows that $\rho^\s$~cor\-res\-ponds
to $\xi^{q_0}$.
Indeed, for all $g\in\ll^\times\subseteq\GL_n(\kk)$
with irreducible characteristic polynomial,
$\s(g)$ and $g^{q_0}$ have the same characteristic 
polynomial, thus they are conjugate in~$\GL_n(\kk)$.
It follows that:
\begin{equation*}
\tr\ \rho^\s(g) = \tr\ \rho(g^{q_0})
= (-1)^{n-1}\cdot\sum\limits_{\g} \xi^\g(g^{q_0})
=  (-1)^{n-1}\cdot\sum\limits_{\g} (\xi^{q_0})^\g(g)
\end{equation*}
for all $g\in\ll^\times$ with irreducible characteristic polynomial.
Thus $\rho$ is $\s$-selfdual if and only if:
\begin{equation}
\label{preramanujanguesthouse}
\xi^{-q_0} = \xi^{q_0^{2i}},
\quad
\text{for some } i\in\{0,\dots,n-1\}.
\end{equation}
Expo\-nen\-tiating to $-q_0$ again gives us 
the equality $\xi^{q} = \xi^{q^{2i}}$.
The $\kk$-regularity assumption on~$\xi$ implies that $n$
divides $2i-1$, thus $n$ is odd.
Besides, since $0\<i\<n-1$, we have $n=2i-1$.
It follows that:
\begin{equation}
\label{ramanujanguesthouse}
\text{$\rho$ is $\s$-selfdual} 
\ \Leftrightarrow\ 
\xi^{-1} = \xi^{q_0^{n}}.
\end{equation}
This is also equivalent to $\xi$ being trivial on $\ll^\times_0$, 
where $\ll_0^{}$ is the subfield of $\ll$ with $q_0^n$ elements.

Conversely, suppose that $\FC$ has characteristic $0$ and $n$ is odd.
Let $\xi$ be an $\FC$-character of~$\ll^\times$~of order $q_0^n+1$, 
which exists since $\ll^\times$ has order $q^n-1=(q_0^{n}-1)(q_0^{n}+1)$.
It is thus trivial on $\ll^\times_0$.
On ther other hand, $q_0$ has order $2n$ mod $q_0^n+1$, 
thus $q$ has order $n$ mod $q_0^n+1$.
It follows that the character $\xi$ is $\kk$-regular.
\end{proof}

\begin{rema}
\label{chaussures}
When $\FC$ has characteristic $\ell>0$,
the group $\GL_n(\kk)$ may have $\s$-selfdual cus\-pi\-dal 
(non supercuspidal) representations for $n$ even,
and it may have no $\s$-selfdual 
supercus\-pi\-dal~represen\-ta\-tion for $n$ odd.
\begin{enumerate}
\item
For an example of $\s$-selfdual cus\-pi\-dal 
non supercuspidal representation for $n$ even,
let~$e$~be the order of $q$ mod~$\ell$,
and suppose that $n=e\ell^u$ for some integer $u\>0$.
Then, by \cite{MSf} Théorème 2.4,
the unique generic subquotient $\pi$ of the representa\-tion in\-duced 
from the trivial character of a Borel subgroup of $\GL_n(\kk)$ 
is cuspidal and $\s$-selfdual.
One may choose $q$, $\ell$ and $u$ such that $n$ is even.
For instance, this is the case when $n=2$ and $\ell\neq2$ divides $q+1$.
\item
The group $\GL_n(\kk)$ may even have no 
supercus\-pi\-dal~represen\-ta\-tion at all:
this is the case, for instance, when $n=3$, $q=2$ and $\ell=7$.
\end{enumerate}
\end{rema}

\begin{lemm}
\label{L72}
Let $\rho$ be a supercuspidal  representation of $\GL_n(\kk)$ 
for some odd integer $n\>1$.
The following assertions are equivalent:
\begin{enumerate}
\item
The representation $\rho$ is $\s$-selfdual.
\item
The representation $\rho$ is $\GL_n(\kk_0)$-distinguished.
\item
The space $\Hom_{\GL_n(\kk_0)}(\rho,1)$ has dimension $1$.
\end{enumerate}
\end{lemm}

\begin{proof}
When $\FC$ has characteristic $0$, this is due to Gow \cite{Gow}.
Suppose that $\FC$ has characte\-ris\-tic $\ell>0$ prime to $q$.
We postpone to Section~\ref{COE}
the proof of the fact that $(2)$ implies $(1)$ and~is
equivalent~to $(3)$,
since the proof of Theorem \ref{FP}
works in both the finite and non-Archimedean cases
(see Remark \ref{bananejaune}).
Here we prove that (1) implies (2).
For this, we use the following general lemma. 

\begin{lemm}
\label{aimee}
Let ${\sf G}$ be a finite group and ${\sf H}$ be a subgroup of ${\sf G}$.
Let $\pi$ be an irreducible~repre\-sentation of ${\sf G}$ on a
$\qlb$-vector space $\V$ such that $\Hom_{{\sf H}}(\pi,1)$ is non-zero. 
Let $\L\subseteq\V$ be a ${\sf G}$-stable $\zlb$-lattice.
Then $\L\otimes\flb$ has at least one irreducible ${\sf G}$-subquotient
$\tau$ such that $\Hom_{{\sf H}}(\tau,1)\neq\{0\}$.
\end{lemm}

\begin{proof}
Let $\h$ be a non-zero ${\sf H}$-invariant linear form on $\V$.
The image of $\L$ by $\h$, denoted $\M$, is a $\zlb$-lattice in $\qlb$. 
Reducing mod the maximal ideal of $\zlb$ gives a non-zero 
${\sf H}$-invariant $\flb$-linear map $\overline{\h}$
from $\L\otimes\flb$ to $\M\otimes\flb\simeq\flb$.
Let $\W$ be the largest subrepresentation of $\L\otimes\flb$ contained
in the kernel of $\overline{\h}$.
Then any irreducible subrepresentation $\tau$ of $(\L\otimes\flb)/\W$
satisfies the required condition. 
\end{proof}

Let $\xi$ be a $\kk$-regular character of $\ll^\times$ parametrizing some 
$\s$-selfdual supercuspidal repre\-senta\-tion $\rho$ of $\GL_n(\kk)$. 
By \eqref{ramanujanguesthouse}, we have $\xi^{-1}=\xi^{q_0^n}$.
Let us fix a field embedding $\iota:\flb\to\FC$.~Since $\xi$ has finite 
image,
there is a $\kk$-regular character 
$\widetilde{\xi}$ of $\ll^\times$ with values~in $\zlb$ such that:
\begin{itemize}
\item 
the character $\widetilde{\xi}$ satifies the identity 
$\widetilde{\xi}^{-1}=\widetilde{\xi}^{q_0^n}$, and
\item 
one has $\xi=\iota\circ\xi_0$
where $\xi_0$ is the reduction mod $\ell$ of $\widetilde{\xi}$.
\end{itemize}
The character $\widetilde{\xi}$ corresponds to a $\s$-selfdual
supercuspidal $\qlb$-re\-pre\-sen\-ta\-tion $\widetilde{\rho}$ of 
$\GL_n(\kk)$.
Let $\V$ denote the vector space of $\widetilde{\rho}$
and fix a $\GL_n(\kk)$-stable $\zlb$-lattice $\L$ in $\V$.
By Proposition \ref{claudine},~the 
representation of $\GL_n(\kk)$~on~the $\flb$-vector 
space $\L\otimes\flb$ is isomorphic to the 
supercuspidal~re\-pre\-sen\-ta\-tion 
$\rho_0$ corresponding to $\xi_0$,
and it is distinguished by Lemma \ref{aimee}.
The result now fol\-lows from Remark \ref{claudinerem},
which tells us that $\rho_0\otimes\FC$ is distinguished and 
isomorphic to $\rho$.
\end{proof}

\begin{rema}
\label{HenryDurie}
If $\FC$ is the field $\flb$, we proved that the representation 
$\rho$ is distin\-guished if and only if it~has a distinguished lift to 
$\qlb$.
\end{rema}

\begin{rema}
\label{mabille}
We give an example of a $\s$-selfdual cuspidal non supercuspidal 
representation of $\GL_n(\kk)$ which is not distinguished. 
With the notation of Remark \ref{chaussures},
assume that $n=e=2$. 
Thus $\pi$ is a $\s$-selfdual cuspidal (non supercuspidal) 
representation of $\GL_2(\kk)$.
Let $\widetilde{\pi}$ be an $\ell$-adic lift of $\pi$
(see Remark \ref{claudinerem}),
and decompose its restriction to $\GL_2(\kk_0)$ 
as a direct sum:
\begin{equation*}
\V_1\oplus\dots\oplus\V_r
\end{equation*} 
of irreducible components. 
Since the order of $\GL_2(\kk_0)$ is prime to $\ell$, 
reduction mod $\ell$ preserves irreducibility, 
and the restriction of $\pi$ to $\GL_2(\kk_0)$ 
is semisimple.
It follows that $\pi$ decomposes as 
$\W_1\oplus\dots\oplus\W_r$,
where $\W_i$ is irreducible and is the reduction mod $\ell$ of $\V_i$ 
for each $i=1,\dots,r$.
Suppose that $\pi$ is distinguished.
Then $\W_i$ is the trivial character for some $i$. 
Thus $\V_i$ is a~cha\-rac\-ter, and it must be trivial since $\GL_2(\kk_0)$ 
has no non-trivial character of order a power of $\ell$,
which implies that $\widetilde{\pi}$ is distinguished.
This is impossible, since $n=2$ is even.
\end{rema}

\begin{rema}
\label{mabille2}
More generally, the argument of Remark~\ref{mabille} shows that,
if $\H$ is a subgroup of $\GL_n(\kk)$ whose order is prime to $\ell$,
then a cuspidal representation $\pi$ of $\GL_n(\kk)$ is $\H$-distinguished 
if and only if any $\ell$-adic lift of $\pi$ is $\H$-distinguished.
\end{rema}

\subsection{A mirabolic interlude}
\label{anais}

This~para\-graph is inspired from Matringe \cite{Matringe}. 
We assume that $p\neq2$.
Let 
$\GM$ denote the~group $\GL_n(\kk)$ for some $n\>2$.
Write $\PM$ for the mirabolic subgroup of $\GM$,
which is 
made of all matrices in $\GM$ whose last row is $(0\ \dots\ 0\ 1)$.
Let $\UM$ be the unipotent radical of $\PM$,
and $\GM'$ be the image of $\GL_{n-1}(\kk)$ in $\GM$ under 
the group homo\-morphism:
\begin{equation*}
g\mapsto
\begin{pmatrix}
g&0\\0&1
\end{pmatrix}.
\end{equation*}
We thus have $\PM=\GM'\UM$,
and we write $\PM'$ for the mirabolic subgroup of $\GM'$.
Let $\N$ be the maximal unipotent subgroup of $\GM$ made of 
all upper trangular unipotent matrices, and $\psi$ be a non-trivial 
character of $\kk$.
We still write $\psi$ for the non-degenerate character 
\begin{equation*}
x \mapsto \psi(x_{1,2}+\dots+x_{n-1,n})
\end{equation*}
of $\N$.
We have a functor:
\begin{equation*}
\pi \mapsto\Ind^\PM_{\PM'\UM}(\pi\otimes\psi)
\end{equation*}
denoted $\Phi^+$,
from $\FC$-representations of $\PM'$ to $\FC$-representations of $\PM$,
where $\pi\otimes\psi$ is the representation of $\PM'\UM$ defined by 
$xu\mapsto\pi(x)\psi(u)$ for all $x\in\PM'$ and $u\in\UM$.

Given integers $r\>s\>0$ such that $r+s=n$,
let $\H_{r,s}$ be the subgroup of $\GM$ defined in~\cite{Matringe}.~It
is the conjugate of the Levi subgroup $\GL_r(\kk)\times\GL_s(\kk)$ 
of $\GM$ under the per\-mu\-ta\-tion matrix $w_{r,s}$ defined 
by the permutation:
\begin{equation*}
\left(\ 
\begin{matrix}
  1 & \cdots & t & \cdots & t+i & \cdots & r & r+1 & \cdots & r+1+j \\
  1 & \cdots & t & \cdots & t+2i & \cdots & n-1 & t+1 & \cdots & t+1+2j 
\end{matrix}
\quad
\begin{matrix}
\cdots & n-1 & n \\
\cdots & n-2 & n
\end{matrix}
\ \right)
\end{equation*}
where $t=r-s+1$.
If $s\>1$, let $\H'_{r,s}$ be the subgroup $\GM'\cap\H_{r,s}$ 
(denoted $\H_{r,s-1}$ in \cite{Matringe}).

\begin{lemm}
\label{LNM1}
Let $\pi$ be a representation of $\PM'$,
and $\chi$ be a character of $\H_{r,s}$.
Suppose that the vector space
$\Hom_{\PM\cap\H_{r,s}}(\Phi^+(\pi),\chi)$ is non-zero.
Then it is isomorphic to $\Hom_{\PM'\cap\H'_{r,s}}(\pi,\chi)$.
\end{lemm}

\begin{proof}
Given $g\in\G$ and a representation $\tau$ of a subgroup $\H$ of $\G$,
we will write $\H^g=g^{-1}\H g$, and $\tau^g$ for the representation
$x\mapsto\tau(gxg^{-1})$ of $\H^g$.
Applying the Mackey formula, and since $\GM'$ normalizes $\UM$, 
the restriction of $\Phi^+(\pi)$ to $\PM\cap\H_{r,s}$ decomposes as
the direct sum:
\begin{equation*}
\bigoplus\limits_{g} 
\Ind^{\PM\cap\H_{r,s}}_{\PM\cap\H_{r,s}\cap\PM'^g\UM}(\pi^g\otimes\psi^g)
\end{equation*}
where $g$ ranges over a set of representatives of 
$(\PM'\UM,\PM\cap\H_{r,s})$-double cosets in $\PM$.
Since $\PM=\GM'\UM$, we may assume $g$ ranges over a set of 
representatives of 
$(\PM',\H'_{r,s})$-double cosets in $\GM'$.
For each~$g$, let us write
the following isomorphism of representations of $\PM\cap\H_{r,s}$:
\begin{equation}
\label{benoui}
\Ind^{\PM\cap\H_{r,s}}_{\UM\cap\H_{r,s}}(\pi^g\otimes\psi^g) \\
\simeq
\Ind^{\PM\cap\H_{r,s}}_{\PM\cap\H_{r,s}\cap\PM'^g\UM}
\left((\pi^g\otimes\psi^g)\otimes
\Ind^{\PM\cap\H_{r,s}\cap\PM'^g\UM}_{\UM\cap\H_{r,s}}(1)\right).
\end{equation}
Since the induced representa\-tion 
$\Ind^{\PM\cap\H_{r,s}\cap\PM'^g\UM}_{\UM\cap\H_{r,s}}(1)$
cano\-nically surjects onto the trivial character of 
$\PM\cap\H_{r,s}\cap\PM'^g\UM$ 
by~Fro\-be\-nius reciprocity, the 
right hand side of \eqref{benoui} surjects onto:
\begin{equation*}
\Ind^{\PM\cap\H_{r,s}}_{\PM\cap\H_{r,s}\cap\PM'^g\UM}(\pi^g\otimes\psi^g).
\end{equation*}
On the other hand, since $\pi^g$ is trivial on $\UM\cap\H_{r,s}$,
the left hand side of \eqref{benoui} is a sum of fini\-tely many
copies of $\Ind^{\PM\cap\H_{r,s}}_{\UM\cap\H_{r,s}}(\psi^g)$.
It follows that,
if $\Hom_{\PM\cap\H_{r,s}}(\Phi^+(\pi),\chi)$ is non-zero,
then there is a $g\in\GM'$ such that:
\begin{equation*}
\Hom_{\PM\cap\H_{r,s}}\left(\Ind^{\PM\cap\H_{r,s}}_{\UM\cap\H_{r,s}}
(\psi^g),\chi\right)\neq\{0\}.
\end{equation*}
By Frobenius reciprocity, this implies that $\psi^g=\chi$ on 
$\UM\cap\H_{r,s}$.
Since we assumed that $p\neq2$, the field $\kk$ has at least three
elements, thus any character of $\H_{r,s}\simeq \GL_r(\kk)\times\GL_s(\kk)$ 
is trivial on unipotent elements.
We thus get $\psi^g=1$ on $\UM\cap\H_{r,s}$. 
By \cite{Matringe} Lemma 3.1, this implies that $g\in\PM'\H'_{r,s}$, that is,
we may assume that $g=1$.
We thus have:
\begin{equation*}
\Hom_{\PM\cap\H_{r,s}}(\Phi^+(\pi),\chi)=
\Hom_{\PM\cap\H_{r,s}}\left(\Ind^{\PM\cap\H_{r,s}}_{\PM'\UM\cap\H_{r,s}}
(\pi\otimes\psi),\chi\right) 
\simeq\Hom_{\PM'\UM\cap\H_{r,s}}(\pi\otimes\psi,\chi).
\end{equation*}
The result now follows from the fact that
$\PM'\UM\cap\H_{r,s}=\PM'\cap\H'_{r,s}$.
(This latter equality has been pointed out to me by N.~Matringe.)
\end{proof}

\begin{rema}
\label{camille}
\begin{enumerate}
\item 
If $\chi$ is assumed to be trivial on unipotent elements,
or if $\kk$ has at least three elements,
then Lemma \ref{LNM1} holds without assuming that $p\neq2$.
\item 
If $\kk$ has cardinality $2$,
the group $\GL_2(\kk)$ has a cuspidal character.
Thus, when $r\>2$, the group $\H_{r,2}$ 
has a character which is non-trivial on $\U\cap\H_{r,2}$.
\end{enumerate}
\end{rema}

Now write $\GM''$ for the copy of $\GL_{n-2}(\kk)$ in the upper left block 
of $\GM'$ and $\PM''$ for the mira\-bo\-lic subgroup of $\GM''$.

\begin{lemm}
\label{LNM2}
Let $\pi'$ be a representation of $\PM''$,
and $\chi'$ be a character of $\H'_{r,s}$.
Suppose~that $s\>1$.
If $\Hom_{\PM'\cap\H'_{r,s}}(\Phi^+(\pi'),\chi')$ is non-zero,
then it is isomorphic to
$\Hom_{\PM''\cap\H_{r-1,s-1}}(\pi',\chi')$.
\end{lemm}

\begin{proof}
It is similar to that of Lemma \ref{LNM1}, 
replacing \cite{Matringe} Lemma 3.1 by 
\cite{Matringe} Lemma 3.2.
\end{proof}

Let $\Ga$ denote the mirabolic representation of $\PM$.
Recall that it is defined as the representation of $\PM$ induced from the 
character $\psi$ of $\N$.

\begin{lemm}
\label{ChouBiDooBah}
Let $n\>2$ and $r\>s\>0$ be such that $r+s=n$.
Let $\chi$ be a character of $\H_{r,s}$.
If the space $\Hom_{\PM\cap\H_{r,s}}(\Ga,\chi)$ is non-zero, then $r=s$.
\end{lemm}

\begin{proof}
If $s=0$, then $\H_{r,s}=\GM$ and the result follows from the fact that
the representation $\Ga$ is irreducible of di\-mension greater than $1$.

Suppose that $s\>1$ and that the space 
$\Hom_{\PM\cap\H_{r,s}}(\Ga,\chi)$ is non-zero.
The mirabolic representa\-tion $\Ga$ is isomorphic to 
$\Phi^+(\Ga')$, where $\Ga'$ denotes the mira\-bolic 
representation of $\PM'$. 
By Lemma \ref{LNM1}, the space 
$\Hom_{\PM'\cap\H'_{r,s}}(\Ga',\chi')$ is non-zero,
where $\chi'$ is the restriction of $\chi$ to $\H'_{r,s}$.
Now iden\-ti\-fy $\Ga'$ with $\Phi^+(\Ga'')$, 
where $\Ga''$ is the mira\-bolic representation of $\PM''$. 
By Lemma \ref{LNM2}, the space
$\Hom_{\PM''\cap\H_{r-1,s-1}}(\Ga'',\chi'')$ is non-zero,
where $\chi''$ is the restriction of $\chi'$ to $\H_{r-1,s-1}$.
By induction on $n$, the fact that 
$\Hom_{\PM''\cap\H_{r-1,s-1}}(\Ga'',\chi'')$ is non-zero
implies that $r-1=s-1$, thus $r=s$.
\end{proof}

\begin{prop}
\label{prs}
Let $n\>2$ and $r,s\>0$ be such that $r+s=n$.
Let $\rho$ be a cuspidal~re\-pre\-sentation of $\GM$,
and $\chi$ be a character of $\M=(\GL_r\times\GL_s)(\kk)$.
Suppose $\Hom_{\M}(\rho,\chi)$ is non-zero. 
Then $r=s$. 
\end{prop}

\begin{proof}
Conjugating $\M$ and $\chi$ if necessary,
we may assume that $r\>s$.
The result then follows from Lemma \ref{ChouBiDooBah} and the fact 
that the restriction of $\rho$ to $\PM$ is isomorphic to $\Ga$.
This latter fact is well-known when $\FC$ has characteristic $0$,
and  is given  by \cite{Vigb}  III.1  when $\FC$  is equal  to $\flb$.~For  an
arbitrary $\FC$ of characteristic $\ell>0$,
fix a field embedding of $\flb$ in $\FC$
and write $\rho$ as $\rho_0\otimes\FC$ for some 
cuspidal~ir\-re\-du\-ci\-ble 
$\flb$-repre\-sentation $\rho_0$ of
$\G$ as in Remark \ref{claudinerem}.
Since the restriction of $\rho_0$ to $\PM$ is isomorphic to $\Ga_0$, 
the mirabolic $\flb$-representation of $\PM$,
the restriction of $\rho$ to $\PM$ is isomorphic to 
$\Ga_0\otimes\FC\simeq\Ga$.
\end{proof}

\begin{rema}
\label{pierrepiau}
Suppose that $r=s\>1$.
Putting Lemmas \ref{LNM1} and \ref{LNM2} 
together, we get:
\begin{equation*}
\Hom_{\PM\cap\H_{r,r}}((\Phi^+)^2(\pi),1) 
\simeq \Hom_{\PM''\cap\H_{r-1,r-1}}(\pi,1)
\end{equation*}
for any representation $\pi$ of $\PM''$.
By induction, we get an isomorphism
$\Hom_{\PM\cap\H_{r,r}}(\Ga,1) \simeq \FC$.
\end{rema}

\begin{coro}
  \label{rescousse}
  Suppose that $n=2r$ for some $r\>1$,
  and let $\rho$ be a cuspidal repre\-sentation of $\GL_n(\kk)$. 
  Then the $\FC$-vector space $\Hom_{(\GL_{r}\times\GL_{r})(\kk)}(\rho,1)$
  has dimension at most $1$. 
\end{coro}

\begin{proof}
  This follows from Remark \ref{pierrepiau},
  together with the fact that the restriction of $\rho$ to $\PM$~is
  isomorphic to $\Ga$.
\end{proof}

\subsection{The Levi case}
\label{etrange}

In this paragraph,
we consider the supercuspidal irredu\-ci\-ble representations of
$\GL_n(\kk)$ distin\-guished by some maximal Levi sub\-group. 
As in Paragraph \ref{anais}, we assume that $p\neq2$.

\begin{lemm}
\label{L71tr}
Let $n\>1$ be a positive integer.
\begin{enumerate}
\item
If there is a selfdual supercuspidal irredu\-ci\-ble representation 
of $\GL_n(\kk)$, then either $n=1$ or $n$ is even. 
\item
Suppose that $\FC$ has characteristic $0$, and that either 
$n=1$ or $n$ is even. 
Then there exists a selfdual supercuspidal irredu\-ci\-ble representation of 
$\GL_n(\kk)$.
\end{enumerate}
\end{lemm}

\begin{proof}
If $n=1$, the trivial character of $\kk^\times$ is selfdual 
and supercuspidal. 
Suppose that $n\>2$.
Let us fix an extension $\ll$ of $\kk$ of degree $n$, 
and identify $\ll^\times$ with a maximal torus in $\GL_n(\kk)$.~We 
consider the Green-James
parametrisation \eqref{GreenJames} of isomorphism classes of 
supercuspidal irre\-du\-ci\-ble represen\-tations of $\GL_n(\kk)$ by 
$\kk$-regular characters of $\ll^\times$.

Given a $\kk$-regular character $\xi$ of $\ll^\times$,
let $\rho$ denote the re\-pre\-sentation corresponding to it.
Recall (see Remark \ref{claudinerem}) that $\rho^\vee$
corresponds to $\xi^{-1}$.
Thus $\rho$ is selfdual if and only if:
\begin{equation}
\label{preramanujanguesthousetri}
\xi^{-1} = \xi^{q^{i}},
\quad
\text{for some } i\in\{0,\dots,n-1\}.
\end{equation}
Taking the contragredient again gives us 
the equality $\xi = \xi^{q^{2i}}$.
The regularity assumption on $\xi$ implies that $n$
divides $2i$ and, since $0\<i\<n-1$, we get $n=2i$.
It follows that:
\begin{eqnarray}
\label{ramanujanguesthousetr}
\text{$\rho_\xi$ is selfdual} 
& \Leftrightarrow &
\xi^{-1} = \xi^{q^{n/2}}.
\end{eqnarray}
This is also equivalent to $\xi$ being trivial on $\ll'^\times$, 
where $\ll'$ is the subfield of $\ll$ with cardinality $q^{n/2}$.

Conversely, suppose that $\FC$ has characteristic $0$ and $n=2r$ for some 
$r\>1$. 
Let us consider an $\FC$-character $\xi$ of~$\ll^\times$~of order $q^r+1$, 
which exists since $\ll^\times$ has order $q^n-1=(q^r-1)(q^r+1)$.
It is trivial on $\ll'^\times$.
On ther other hand, $q$ has order $n=2r$ mod $q^r+1$,
which implies that $\xi$ is $\kk$-regular.
\end{proof}

\begin{rema}
\label{carracci}
When $\FC$ has characteristic $\ell>0$,
the group $\GL_n(\kk)$ may have selfdual~cus\-pidal 
(non supercuspidal) representations even if $n$ is odd and $>1$.
Indeed, let $e$ be the order~of $q$ mod $\ell$,
and suppose that $n=e\ell^u$ for some $u\>0$.
The unique generic irreducible~sub\-quotient of 
the re\-pre\-sentation in\-duced 
from the trivial character of a Borel subgroup of $\GL_n(\kk)$ 
is then both cuspi\-dal (see Remark \ref{chaussures}) and self\-dual.
One then can choose $q$, $\ell$ and $u$ such that 
$n$ is odd and $>1$. 
For instance, this is the case when
$\ell\neq2$ divides $q-1$ and $n=\ell$.

Also, as in Remark \ref{chaussures},
the group $\GL_n(\kk)$ may even have no 
supercus\-pi\-dal~represen\-ta\-tion at all, 
which is the case, for instance, when $n=q=2$ and $\ell=3$.
\end{rema}

\begin{lemm}
\label{L72r}
Suppose that $n=2r$ with $r\>1$,
and let $\rho$ be a supercuspidal  repre\-sentation of $\GL_n(\kk)$. 
The following assertions are equivalent:
\begin{enumerate}
\item
The representation $\rho$ is selfdual.
\item
The representation $\rho$ is $(\GL_{r}\times\GL_{r})(\kk)$-distinguished.
\item
The $\FC$-vector space $\Hom_{(\GL_{r}\times\GL_{r})(\kk)}(\rho,1)$ has 
dimension $1$. 
\end{enumerate}
\end{lemm}

\begin{proof}
When $\FC$ has characteristic $0$, this is 
\cite{HMIMRN} Proposition 6.1 (see 
also \cite{Coniglio} Lemme {3.4.10}).
Suppose now $\FC$ has prime characte\-ristic $\ell$ not dividing $q$.
To prove that (1) implies (2), we~follow the same lifting argument as
in the proof of Lemma \ref{L72}.

We now prove that (2) implies (1).
Let us write $\G=\GL_n(\kk)$ and 
$\H=(\GL_{r}\times\GL_{r})(\kk)$. 
Let~$\rho$ be an $\H$-distinguished supercuspidal repre\-sentation
of $\G$.
If one fixes a field embedding of $\flb$~in $\FC$, then 
Remark \ref{claudinerem} tells us that $\rho$ is isomorphic to 
$\rho_0\otimes\FC$ 
for some distinguished supercuspidal irredu\-ci\-ble 
$\flb$-repre\-sentation $\rho_0$ of $\G$.
Since $\rho$ is selfdual if and only if $\rho_0$ is,
we are thus reduced to proving the result in the case where
$\FC$ is equal to $\flb$, which we assume now.

Since $\rho$ is distinguished,
its contragredient $\rho^\vee$ has a non-zero $\H$-invariant vector.
We thus have a non-zero homomorphism
$i:\zlb[\H\backslash\G]\to\rho^\vee$.
Let~us consider a projective envelope
$f:\P\to\rho^\vee$ of $\rho^\vee$ in the category of 
$\zlb[\G]$-modules. 
Since $\rho^\vee$ is super\-cus\-pi\-dal, it has the following properties 
(see for instance \cite{Vigb} III.2.9): 
\begin{itemize}
\item
the representation $\P\otimes\qlb$ is isomorphic to the direct 
sum of all the $\qlb$-lifts of $\rho^\vee$;
\item
there are $\ell^a$ such lifts, where $a$ is the $\ell$-adic valuation of $q^n-1$;
\item
the representation $\P\otimes\flb$ is indecomposable of length $\ell^a$
and has a unique irreducible~quo\-tient, 
and all its irreducible components are isomorphic to $\rho^\vee$. 
\end{itemize}
By projectivity of $\P$, 
the homomorphism $i$ 
gives rise to a non-zero homo\-mor\-phism:
\begin{equation}
\label{Lenu}
j : \P \to \zlb[\H\backslash\G]
\end{equation}
such that $i\circ j=f$.
Inverting $\ell$, we get a non-zero homomorphism from $\P\otimes\qlb$ to 
$\qlb[\H\backslash\G]$.~It
follows that $\rho^\vee$ has at least one $\H$-distinguished lift.
Thanks to the characteristic $0$ case, such~a lift is self\-dual. 
Reducing mod $\ell$, it follows that $\rho$ is selfdual.

We now go back to the case of a general $\FC$.
The fact that (2) implies (3) is a particular~case of Co\-rol\-lary \ref{rescousse}.
However, we are going to give another proof here,
which works for supercuspidal representations only but
is more likely~to generalize to other situations.

Let $\V$ be the maximal direct summand of 
$\FC[\H\backslash\G]$ in the block of $\rho$. 
This means that $\FC[\H\backslash\G]$ decomposes as a direct sum $\V\oplus\V'$ 
where all irreducible subquotients of $\V$
are isomorphic to $\rho$, and no irreducible subquotient of $\V'$
is iso\-mor\-phic to $\rho$. 
Besides, since $\rho$ is selfdual, we have:
\begin{equation*}
\dim \Hom_{\H}(\rho,1) =
\dim \Hom_{\H}(1,\rho) =
\dim \Hom_{\G}(\FC[\H\backslash\G],\rho) =
\dim \Hom_{\G}(\V,\rho).
\end{equation*}
We thus have to prove that the cosocle of $\V$ is isomorphic to $\rho$. 

\begin{lemm}
\label{cerulloshoes}
The $\FC$-algebra $\A=\End_\G(\V)$ is commutative.
\end{lemm}

\begin{proof}
Since the convolution algebra $\FC[\H\backslash\G/\H]$ decomposes as
$\End_\G(\V)\oplus\End_\G(\V')$, it suffices to prove that 
$\FC[\H\backslash\G/\H]$ is commutative.
For $x\in\G$, let $f_x$ be the characteristic function in 
$\FC[\H\backslash\G/\H]$ of the double coset $\H x\H$.
For $x,y\in\G$, one has:
\begin{equation*}
f_x*f_y = \sum\limits_{z\in\H\backslash\G/\H} a(x,y,z) f_z
\end{equation*}
where $a(x,y,z)\in\FC$ is the image of the cardinality of
$(\H x\H\cap z\H y^{-1}\H)/\H$ in $\FC$.
When $\FC$ has~cha\-rac\-teristic $0$, the algebra 
$\FC[\H\backslash\G/\H]$ is known to be commutative
since $\FC[\H\backslash\G]$ is multiplicity free as an $\FC[\G]$-module, 
thus:
\begin{equation}
\label{egreco0}
{\rm card}\ (\H x\H\cap z\H y^{-1}\H)/\H =
{\rm card}\ (\H y\H\cap z\H x^{-1}\H)/\H
\end{equation}
for all $x,y,z\in\G$.
Now if $\FC$ has characteristic $\ell>0$, 
reducing \eqref{egreco0} mod $\ell$ gives us
a congruence relation
which tells us that the algebra $\FC[\H\backslash\G/\H]$ is commutative.
\end{proof}

It remains to prove that the cosocle of $\V$ is multiplicity free. 
Let $m\>1$ be the multiplicity of $\rho$ in the cosocle of $\V$
and $\Q$ be the projective inde\-com\-po\-sable $\FC[\G]$-module associated 
with~$\rho$.
It has length $\ell^a$, has a unique irreducible quo\-tient, 
and all its irreducible components are isomorphic to $\rho$. 
Write $\V=\V_1\oplus\dots\oplus\V_m$ where 
$\V_1,\dots,\V_m$ are indecomposable $\FC[\G]$-modules with 
cosocle isomorphic to $\rho$. 
There is a nilpotent endomorphism $\N\in\End_\G(\Q)$ such that: 
\begin{equation*}
\End_\G(\Q)=\FC[\N]
\end{equation*}
with $\N^{\ell^a}=0$ and $\N^{\ell^a-1}\neq0$.
Therefore each $\V_i$ is iso\-morphic to the quotient of $\Q$ by the image of 
$\N^{k_i}$ for some $k_i\>0$.
Reordering if necessary, 
we may assume that $\Hom(\V_1,\V_i)$ is non-zero for all $i\>1$.
Suppose that $m\>2$, and define two endomorphisms $u,u'\in\A$ by:
\begin{enumerate}
\item
the endomorphisms $u,u'$ are trivial on $\V_i$ for all $i\>2$,
\item
the restriction of $u$ to $\V_1$ is the identity on $\V_1$,
\item
the restriction of $u'$ to $\V_1$ coincides with some non-zero homomorphism 
in $\Hom(\V_1,\V_2)$. 
\end{enumerate}
Then $uu'=0$ and $u'u\neq0$,
thus $\A$ is not commutative.
Thus $m=1$.
\end{proof}

\begin{rema}
  \label{r72}
If $\FC$ is the field $\flb$, we proved that
$\rho$ is distin\-guished if and only if it~has a distinguished lift to 
$\qlb$ (see Remark \ref{HenryDurie}).
\end{rema}

\begin{rema}
If we only assume $\rho$ to be cuspidal in Lemmas \ref{L71tr} and 
\ref{L72r},
then the lifting ar\-gu\-ment may not work, that is, $\xi$ may not have 
a $\s$-selfdual $\kk$-regular lift $\widetilde{\xi}$. 
Besides, the struc\-tu\-re of the projective envelope of $\rho$ 
is more complicated 
when $\rho$ is cuspidal non-supercuspidal.
\end{rema}

\section{Notation and basic definitions in the non-Archimedean case} 
\label{Notation}

Let $\F/\F_{\zero}$ be a separable quadratic extension of locally compact 
non-archimedean local fields~of residual characteristic $p$.
Apart from Section \ref{COE}, we will assume that $p\neq2$.

Write $\s$ for the non-trivial $\F_{\zero}$-automorphism of $\F$.
Write $\EuScript{O}$ for the ring of integers of $\F$ 
and $\EuScript{O}_{\zero}$ for that of $\F_{\zero}$.
Write $\kk$ for the residue field of $\F$ and $\kk_{\zero}$ for that of $\F_{\zero}$.
The involution~$\s$~in\-duces a $\kk_{\zero}$-auto\-mor\-phism of $\kk$, 
still denoted $\s$, which generates $\Gal(\kk/\kk_{\zero})$. 

As in Section \ref{FFC},
let $\FC$ be an algebraically closed field of characteristic $\ell$ different 
from $p$.
(Note that $\ell$ can be $0$.)
We say we are in the ``modular case'' when we consider the case where 
$\ell>0$. 

We fix once and for all a character:
\begin{equation}
\label{eq1}
\psi_{\zero}:\F_{\zero}\to\FC^\times
\end{equation}
trivial on the maximal ideal of $\EuScript{O}_{\zero}$ but not on $\EuScript{O}_{\zero}$,
and define $\psi=\psi_{\zero}\circ\tr_{\F/\F_{\zero}}$.

When $\ell\neq2$, we write:
\begin{equation}
\label{eq2}
\ep=\ep_{\F/\F_{\zero}} : \F_{\zero}^\times \to \FC^\times
\end{equation}
for the character of $\F_{\zero}^\times$ 
whose kernel is the subgroup of $\F/\F_{\zero}$-norms.

Let $\G$ be the locally profinite group $\G=\GL_n(\F)$, with $n\>1$,
equipped with the involution $\s$ acting componentwise. 
Its $\s$-fixed points is the closed subgroup $\G^\s=\GL_n(\F_{\zero})$. 
We will identify the centre of $\G$ with $\mult\F$,
and that of $\G^\s$ with $\F_{\zero}^\times$. 

By \textit{representation} of a locally profinite group, 
we always mean a smooth representation on an $\FC$-modu\-le. 
Given a representation $\pi$ of a closed subgroup $\H$ of $\G$, 
we write $\pi^\vee$ for the smooth contra\-gredient of $\pi$, 
and $\pi^\s$ for the representation $\pi\circ\s$ of the subgroup $\s(\H)$.
We say that $\pi$ is \textit{$\s$-selfdual} if $\H$ is $\s$-stable and 
$\pi^\s$, $\pi^\vee$ are isomorphic.
If $g\in\G$, we write $\H^g=\{g^{-1}hg\ |\ h\in\H\}$ 
and $\pi^g$ for the representation $x\mapsto\pi(gxg^{-1})$ of $\H^g$.
If $\chi$ is a character of $\H$, we write $\pi\chi$ 
for the representation $g\mapsto\chi(g)\pi(g)$.

If $\mu$ is a character of $\H\cap\G^\s$, we say that $\pi$ is
$\mu$-\textit{distinguished} if the  
space $\Hom_{\H\cap\G^\s}(\pi,\mu)$ is non-zero.
If $\mu$ is the trivial character,
we will simply say that $\pi$ is $\H\cap\G^\s$-\textit{distinguished}, 
or just \textit{distinguished}. 
If $\H=\G$ and $\phi$ is a character of $\F_0^\times$, 
we will abbreviate $\phi\circ\det$-\textit{distinguished} to 
$\phi$-\textit{dis\-tin\-guished}.

An irreducible representation of $\G$ is said to be \textit{cuspidal} 
if all its proper Jacquet modules~are trivial or, equivalently, 
if it does not occur as a subrepresentation of a proper 
parabolically~indu\-ced 
representation.
It is said to be \textit{supercuspidal} 
if it does not occur as a subquotient of~a proper parabolically 
induced representation
(by \cite{DatDMJ12} Corollaire B.1.3,
this is equivalent to not occuring as 
a subquotient of the parabolic induction of any 
\textit{irreducible} representation of a pro\-per
Levi sub\-group of $\G$). 
When $\FC$ has characteristic $0$,
any cuspidal representation is supercuspidal. 

\section{A modular version of theorems of Prasad and Flicker}
\label{COE}

In this section, the residue characteristic $p$ is arbitrary.
We prove the following theo\-rem, which is well-known in the complex case. 
Note that, in the modular case, any irreducible representation of $\G$
has a central character by \cite{Vigb} II.2.8.

\begin{theo}
\label{FP}
Let $\pi$ be a distinguished irreducible representation of $\G$. 
Then:
\begin{enumerate}
\item
The central character of $\pi$ is trivial on $\F_0^\times$.
\item
The contragredient representation $\pi^\vee$ is distinguished. 
\item
The space $\Hom_{\G^\s}(\pi,1)$ has dimension $1$.
\item
The representations $\pi^\s$ and $\pi^\vee$ are isomorphic,
that is, $\pi$ is $\s$-selfdual. 
\end{enumerate}
\end{theo}

\begin{rema}
In the complex case, 
this theorem was first proved under the assumption that 
the~cha\-rac\-teristic of $\F$ is not $2$, 
which was required in the proof of \cite{Flicker} Proposition 10.
Later,~in \cite{PrasadDMJ01} Section 4, 
Prasad gave an argument which only requires $\F/\F_0$ to be 
separable quadra\-tic.
\end{rema}

\begin{proof}
Property 1 is straightforward. 
Property 2 follows from an argument of Gelfand-Kazh\-dan 
(see \cite{SVdml} Proposition {8.4} in the modular case).
For Properties 3 and 4,
we follow the proofs of Prasad \cite{PrasadCMAT90} 
and Flicker \cite{Flicker}.
The reference for the basic results in the
theory of modular~re\-presentations of $p$-adic reductive
groups which we use in the proof is \cite{Vigb}.

Write $\Cc_{{\rm c}}^\infty(\G)$ for the space of locally constant, 
compactly supported $\FC$-valued func\-tions on $\G$,
and fix an $\FC$-valued Haar measure on $\G$,
that is, a non-zero $\FC$-linear form on $\Cc_{{\rm c}}^\infty(\G)$ 
in\-va\-riant un\-der left trans\-la\-tion by $\G$. 

Let $\W$ denote the vector space of $\pi$, 
and $l:\W\to\FC$ be a non-zero $\G^{\s}$-invariant linear form. 
For any $f\in\Cc_{{\rm c}}^\infty(\G)$, define a linear form on $\W$ by:
\begin{equation*}
\pi(f)l : w \mapsto \int_\G f(x) l(\pi(x)w)\ dx.
\end{equation*}
Since $f$ is smooth, the linear form $\pi(f)l$ on $\W$ is smooth. 
This defines a non-zero 
homomorphism $\L:\Cc_{{\rm c}}^\infty(\G)\to\W^{\vee}$.
It is $\G$-equivariant under right translation
and $\G^{\s}$-invariant under left~trans\-la\-tion.
Since $\W$ is irreducible, it is surjective.
Similarly, given a non-zero $\G^{\s}$-invariant linear form 
$m:\W^\vee\to\FC$, we obtain a surjective right $\G$-equivariant 
and left $\G^{\s}$-invariant homo\-mor\-phism 
$\M$ from $\Cc_{{\rm c}}^\infty(\G)$ to $\W^{\vee\vee}\simeq\W$
(see \cite{Vigb} Proposition I.4.18 for the latter isomorphism).
We now define:
\begin{equation*}
\B(f,g) = \langle \M(f),\L(g)\rangle \in \FC
\end{equation*}
for all $f,g\in\Cc_{{\rm c}}^\infty(\G)$.
This defines a right $\G$-invariant and left $\G^{\s}\times\G^{\s}$-invariant 
linear form $\B$ on the space 
$\Cc_{{\rm c}}^\infty(\G)\otimes\Cc_{{\rm c}}^\infty(\G)\simeq
\Cc_{{\rm c}}^\infty(\G\times\G)$.
As in \cite{PrasadCMAT90} Lemma 4.2 
(and with \cite{PrasadDMJ01} Lemma~4.1,
which extends the result of \cite{Flicker} Proposition 10
to the case where $\F$ has arbitrary characteristic)
we have:
\begin{equation}
\label{AlFi}
\B(f,g) = \B(g\circ\s,f\circ\s)
\end{equation}
for all $f,g\in\Cc_{{\rm c}}^\infty(\G)$.
It follows that the kernel of $\L$ is equal to $\{f\circ\s\ |\ f\in\Ker(\M)\}$.
Thus, if $l'$ is any non-zero $\G^{\s}$-invariant linear form on $\W$,
with associated homomorphism $\L'$, then $\L$,~$\L'$~have the same kernel.
Since $\pi^\vee$ is admissible (by \cite{Vigb} II.2.8),
Schur's Lemma applies (see \cite{Vigb} I.6.9)
thus one has $l'=cl$ for some $c\in\FC^\times$.
Thus $\Hom_{\G^{\s}}(\W,1)$ has dimension $1$.

As in \cite{PrasadCMAT90} Lemma 4.2, the bilinear form
$\B$ corresponds to 
the $\G^{\s}$-bi-invariant linear form $\D$ on $\Cc_{{\rm c}}^\infty(\G)$ 
defined by: 
\begin{equation*}
\D(h) = m(\pi(h)l)
\end{equation*}
for all $h\in\Cc_{{\rm c}}^\infty(\G)$.
The correspondence between $\B$ and $\D$ is given by:
\begin{equation*}
\D(h) = \B(k),
\quad
\text{with $k:(x,y) \mapsto h(xy^{-1})$}.
\end{equation*}
Note that \eqref{AlFi} gives $\D(h)=\D(h\circ\s\circ\iota)$ 
with $\iota:x\mapsto x^{-1}$ on $\G$.
Replacing $\pi$ by $\pi^*=\pi^{\vee\s}$
and exchanging the roles played by $l,m$ we get a linear form:
\begin{equation*}
\D^* : h \mapsto l(\pi^*(h)m).
\end{equation*}
Since we have $l(\pi^*(h)m)=m(\pi(h\circ\s\circ\iota)l)$,
it follows that $\D^*=\D$.
In order to deduce Property 4, it remains to prove that $\D$ determines 
$\pi$ entirely.
For any $\xi\in\W^\vee$ we define the function:
\begin{equation*}
c_\xi : x \mapsto m(\pi^\vee(x)\xi) = m(\xi\circ\pi(x^{-1}))
\end{equation*}
on $\G$.
Then $\xi\mapsto c_\xi$ is an embedding of $\W^\vee$ in the space 
$\Cc^\infty(\G^{\s}\backslash\G)$ of smooth $\FC$-valued functions on 
$\G^{\s}\backslash\G$.
For $y\in\G$ and $h\in\Cc_{{\rm c}}^\infty(\G)$, let ${}^yh$ denote the function 
$x\mapsto h(xy)$.
Since $\L$ and $\M$ are surjec\-tive, there is a function $h$ such that 
$\pi(h)l$ is non-zero. 
Then $y \mapsto \D({}^yh)$ is a non-zero function in the space 
$\Cc^\infty(\G^{\s}\backslash\G)$, 
generating a subrepresentation isomorphic to $\W^\vee$.
Indeed, it is equal to $c_{\pi(h)l}$.
It thus follows from the equality $\D^*=\D$ that we have 
$\pi^\s\simeq\pi^\vee$, 
as expected. 
\end{proof}

\begin{rema}
\label{bananejaune}
The same results hold --
and the same argument works -- 
when $\F/\F_0$ is replaced by a qua\-dra\-tic extension of finite fields
of arbitrary characteristic.
It suffices to replace \cite{PrasadDMJ01}~Lem\-ma 4.1
by \cite{Gow} Lemma~3.5.
\end{rema}

\section{Preliminaries on simple types}
\label{PrelimST}

\textit{From now on, until the end of this article,
we will assume that $p\neq2$.}
This assumption is not needed in Paragraphs
\ref{genitrix41} to \ref{pullfroid},
but we assume it from now on for simplicity.

We assume the reader is familiar with the language of simple types. 
We recall the main~results on simple strata, characters and types 
\cite{BK,BHLTL1,BHEffectiveLC,MSt} that we will need.
Part of these preliminaries can also be found in \cite{AKMSS}.

\subsection{Simple strata and characters}
\label{genitrix41}

Let $[\aa,\b]$ be a simple stratum in 
the $\F$-algebra $\Mat_{n}(\F)$ of $n\times n$ matrices with entries 
in $\F$ for some $n\>1$. 
Recall that $\aa$ is a hereditary order in $\Mat_n(\F)$ and 
$\b$ is a matrix in $\Mat_{n}(\F)$ such that: 
\begin{enumerate}
\item
the $\F$-algebra $\E=\F[\b]$ is a field, whose degree over $\F$ is denoted 
$d$,
\item
the multiplicative group $\E^\times$ normalizes $\aa$. 
\end{enumerate}
The cen\-trali\-zer of $\E$ in $\Mat_{n}(\F)$, denoted $\B$,
is an $\E$-algebra isomorphic to $\Mat_{m}(\E)$, with $n=md$.
The intersection $\bb=\aa\cap\B$ is a hereditary order in $\B$.

Write $\pp_\aa$ for the Jacobson radical of $\aa$,
and $\U^1(\aa)$ for the compact open 
pro-$p$-subgroup $1+\pp_\aa$ of $\G=\GL_n(\F)$. 
We recall the follow\-ing useful \textit{simple intersection property}
(\cite{BK} Theorem 1.6.1):
for all $x\in\B^\times$, we have:
\begin{equation}
\label{SIP}
\U^1(\aa) x\U^1(\aa)\cap\B^\times=\U^1(\bb)x\U^1(\bb).
\end{equation} 

Associated with $[\aa,\b]$,
there are compact open subgroups:
\begin{equation*}
\H^1(\aa,\b)\subseteq\J^1(\aa,\b)\subseteq\J(\aa,\b)
\end{equation*}
of $\aa^\times$ and a finite set $\Cc(\aa,\b)$
of characters of $\H^1(\aa,\b)$ called \textit{simple characters}, 
depending on the choice of the 
character $\psi$ fixed in Section \ref{Notation}.
Write $\BJ(\aa,\b)$ for the compact mod centre subgroup generated 
by $\J(\aa,\b)$ and the normalizer of $\bb$ in $\B^\times$. 

\begin{prop}[\cite{BHEffectiveLC} {\rm 2.1}]
\label{patel}
We have the following properties:
\begin{enumerate}
\item
The group $\J(\aa,\b)$ is the unique maximal compact subgroup of 
$\BJ(\aa,\b)$.
\item
The group $\J^1(\aa,\b)$ is the unique maximal normal 
pro-$p$-subgroup of $\J(\aa,\b)$. 
\item
The group $\J(\aa,\b)$ is generated by $\J^1(\aa,\b)$ and $\bb^\times$, 
and we have:
\begin{equation}
\label{lapremiereegalite}
\J(\aa,\b)\cap\B^\times=\bb^\times,
\quad
\J^1(\aa,\b)\cap\B^\times=\U^1(\bb).
\end{equation}
\item
The normalizer of any simple character $\t\in\Cc(\aa,\b)$ in $\G$
is equal to $\BJ(\aa,\b)$. 
\item
\label{bourrin6}
The intertwining set of any $\t\in\Cc(\aa,\b)$ in $\G$ is equal to 
$\J^1(\aa,\b)\B^\times\J^1(\aa,\b)$. 
\end{enumerate}
\end{prop}

By \cite{BK} Theorem 3.4.1, 
the quotient $\J^1(\aa,\b)/\H^1(\aa,\b)$ is a finite 
$\kk$-vector space, and the map:
\begin{equation}
\label{SYMPLEC}
(x,y)\mapsto\langle x,y \rangle = \t(xyx^{-1}y^{-1})
\end{equation}
makes it into a non-degenerate symplectic space.
More precisely, if $\HH^1(\aa,\b)$ and $\JJ^1(\aa,\b)$ are the 
sub-$\Oo$-lattices of $\aa$ such that
$\H^1(\aa,\b)=1+\HH^1(\aa,\b)$ and $\J^1(\aa,\b)=1+\JJ^1(\aa,\b)$, 
then we have:
\begin{equation}
\label{SYMPLECGOTHIC}
\langle 1+u,1+v \rangle = \psi\circ\tr(\b(uv-vu))
\end{equation}
for all $u,v\in\JJ^1(\aa,\b)$ (\cite{BHLTL3} Proposition 6.1), 
where $\tr$ denotes the trace map of $\Mat_n(\F)$.

Let $[\aa',\b']$ be another simple stratum in $\Mat_{n'}(\F)$ for some 
$n'\>1$, and suppose that there is an $\F$-algebra isomorphism
$\h:\F[\b]\to\F[\b']$ such that $\h(\b)=\b'$. 
Then there is a canonical~bi\-jec\-tive map:
\begin{equation}
\label{transfermap}
\Cc(\aa,\b) \to \Cc(\aa',\b')
\end{equation}
called the \textit{transfer map} (\cite{BK} Theorem 3.6.14).

When the hereditary order $\bb=\aa\cap\B$ is a maximal order in $\B$, 
we say that the 
simple stratum $[\aa,\b]$ and the simple characters in $\Cc(\aa,\b)$
are \textit{maximal}. 
When this is the case, 
then, given a homo\-mor\-phism of $\E$-algebras $\B\simeq\Mat_m(\E)$ 
identifying $\bb$ with the standard maximal order, 
there are group iso\-mor\-phisms: 
\begin{equation}
\label{JJ1UU1GLmax}
\J(\aa,\b)/\J^1(\aa,\b) \simeq \bb^\times/\U^1(\bb) \simeq \GL_{m}(\ee)
\end{equation}
where $\ee$ is the residue field of $\E$.

\subsection{Types and cuspidal representations} 
\label{pifoulechien}
\label{par44}

Let us write $\G=\GL_n(\F)$ for some $n\>1$.
A family of pairs $(\BJ,\bl)$ called \textit{extended maximal sim\-ple types}, 
made of a compact mod centre, open subgroup $\BJ$ of $\G$
and an ir\-reducible representation $\bl$ of $\BJ$,
has been constructed in \cite{BK} (see also \cite{MSt} 
in the modular case).

Given an extended maximal simple type $(\BJ,\bl)$ in $\G$, 
there are a maximal simple stra\-tum $[\aa,\b]$ in $\Mat_n(\F)$ 
and a maximal simple char\-acter 
$\t\in\Cc(\aa,\b)$ such that $\BJ(\aa,\b)=\BJ$ and $\t$ is 
contained in the res\-tric\-tion of $\bl$ to $\H^1(\aa,\b)$.
Such a simple character is said to be \textit{attached to} $\bl$.
By \cite{BK} Propo\-si\-tion 5.1.1 (or \cite{MSt} Proposition 2.1 in 
the modular case), 
the group $\J^1(\aa,\b)$ carries, up to isomor\-phism, 
a unique irreductible re\-pre\-sen\-ta\-tion $\n$ 
whose restriction to $\H^1(\aa,\b)$ contains $\t$.
It is called the \textit{Heisenberg representation}
associated to $\t$ and has the following properties: 
\begin{enumerate}
\item
the restriction of $\n$ to $\H^1(\aa,\b)$ is made of 
$(\J^1(\aa,\b):\H^1(\aa,\b))^{1/2}$ copies of $\t$,
\item
the representation $\n$ extends to $\BJ$. 
\end{enumerate}
For any representation $\bk$ of $\BJ$ extending $\n$, 
there is, up to isomorphism, a unique irreducible~re\-pre\-sentation $\bt$ of 
$\BJ$ trivial on $\J^1(\aa,\b)$ such that 
$\bl \simeq \bk\otimes\bt$.
Through \eqref{JJ1UU1GLmax}, 
the restriction of $\bt$ to the maximal compact subgroup $\J=\J(\aa,\b)$
identifies with a cuspidal re\-presentation of $\GL_m(\ee)$.

\begin{rema}
\label{handluggage}
The reader familiar with the theory of simple types will have noticed that 
we did not introduce the notion of beta-extension.
Since $\GL_{m}(\ee)$ is not isomorphic to $\GL_2(\FF_2)$ (as $p$ is not $2$),
any character of $\GL_{m}(\ee)$ factors through the determinant.
It follows that, if $[\aa,\b]$ is a maximal simple stratum, 
any representation of $\J$ extending 
$\n$ is a beta-extension.
\end{rema}

We have the following additional property,
which follows from \cite{MSt} Lemme 2.6.

\begin{prop}
\label{suzanne}
Let $\bk$ be a representation of $\BJ$ extending $\n$,
and write $\J^1$ for the maximal normal pro-$p$-subgroup of $\BJ$. 
The map:
\begin{equation*}
\boldsymbol{\xi}\mapsto\bk\otimes\boldsymbol{\xi}
\end{equation*} 
induces a bijection between isomorphism classes of 
irreducible representations $\boldsymbol{\xi}$ of $\BJ$ trivial on~$\J^1$ 
and isomorphism classes of irreducible representations of $\BJ$ whose 
restriction to $\J^1$ contains $\n$.
\end{prop}

We now give the classification of cuspidal irreducible representations of $\G$ 
in terms of extended maximal simple types
(see \cite{BK} 6.2, 8.4 and \cite{MSt} Section 3 in the modular case).

\begin{prop}[\cite{BK,MSt}]
\label{GeorgesSaval}
Let $\pi$ be a cuspidal representation of $\G$. 
\begin{enumerate}
\item
There is an extended maximal simple type $(\BJ,\bl)$ such that $\bl$ occurs as 
a subrepresentation of the res\-triction of $\pi$ to $\BJ$.
It is uniquely deter\-mined up to $\G$-conjugacy.
\item
Compact induction defines a bijection between the $\G$-conjugacy 
classes of extended maximal sim\-ple types 
and the isomorphism classes of cuspidal representations of $\G$.
\end{enumerate}
\end{prop}

From now on, we will abbreviate 
\textit{extended maximal simple type} to 
\textit{type}.

\subsection{Supercuspidal representations}
\label{pullfroid}

Let $\pi$ be a cuspidal representation of $\G$. 
By Proposition \ref{GeorgesSaval}, it contains a type $(\BJ,\bl)$.
Fix an irreducible representation $\bk$ as in Proposition \ref{suzanne}
and let $\bt$ be the corresponding re\-presentation of $\BJ$ 
trivial on its maximal normal pro-$p$-subgroup $\J^1$.

Fix a maximal simple stratum $[\aa,\b]$ such that $\BJ=\BJ(\aa,\b)$.
Write $\E=\F[\b]$ and let $\rho$ be the cus\-pi\-dal represen\-tation of 
$\J/\J^1\simeq\GL_m(\ee)$ induced by $\bt$. 
We record the following fact.

\begin{enonce}{Fact}[\cite{MSc} Proposition 6.10]
\label{cordoba}
The representation $\pi$ is supercuspidal if and only if $\rho$ is 
supercuspidal.
\end{enonce}

Now suppose that $\pi$ is supercuspidal,
thus $\rho$ is also supercuspidal.
We show how to parame\-tri\-ze $\bt$ by an 
``admissible pair of level zero''.
This will be needed in Sections \ref{SECDISTTR0} and \ref{S7}.

First,
let $\ll$ be an extension of degree $m$ of $\ee$,
and identify $\ll^\times$ with a maximal torus in $\GL_{m}(\ee)$.
We have the correspondence \eqref{GreenJames} 
between $\ee$-regular characters of $\ll^\times$ 
and isomorphism classes~of supercuspidal irreducible
represen\-tations of $\GL_m(\ee)$.

\begin{defi}[\cite{Howe,BHETLC1}]
\label{defiL0AP}
An \textit{admissible pair of level zero} over $\E$ is a pair $(\K/\E,\xi)$ 
made of a 
fi\-ni\-te unramified extension $\K$ of $\E$ and a tamely ramified character 
$\xi:\K^\times\to\FC^\times$ which does not 
fac\-tor through $\N_{\K/\L}$ for 
any field $\L$ such that $\E\subseteq\L\subsetneq\K$.
Its \textit{degree} is $[\K:\E]$. 
\end{defi}

If $(\K'/\E,\xi')$ is another admissible pair of level zero over $\E$, 
it is said to be \textit{isomorphic} to $(\K/\E,\xi)$ 
if there is an isomorphism of $\E$-algebras $\h:\K'\to\K$ such that 
$\xi'=\xi\circ\h$.

Recall (see \S\ref{genitrix41}) that, 
if we write $\B^\times$ for the centralizer of $\E$ in $\G$, 
then $\BJ=(\BJ\cap\B^\times)\J^1$,
thus the group $\BJ/\J^1$ is isomorphic to $(\BJ\cap\B^\times)/(\J^1\cap\B^\times)$.
In particular, the image of $\E^\times$ in $\BJ/\J^1$ is central. 
Since $\bt$ is trivial on $\J^1$, the automorphism $\bt(x)$ is thus a scalar 
for all $x\in\E^\times$. 

\begin{defi}
\label{defiL0AP1}
An admissible pair $(\K/\E,\xi)$ of level zero and degree $m$ 
is \textit{attached to} $\bt$ if:
\begin{enumerate}
\item 
writing $\ll$ for the residue field of $\K$,
the $\ee$-regular 
character of $\ll^\times$ induced by the restriction of 
$\xi$ to the units of the ring of inte\-gers of $\K$ cor\-responds to $\rho$ 
via \eqref{GreenJames}, 
\item
one has $\bt(x)=\xi(x)\cdot{\rm id}$ for all $x\in\E^\times$,
where ${\rm id}$ is the identity on the space of $\bt$.
\end{enumerate}
\end{defi}

The following proposition 
is a refinement of the property of the map \eqref{GreenJames}.

\begin{prop}
The attachment relation defines a bijection:
\begin{equation}
\label{aspirantdunand}
(\K/\E,\xi) \mapsto \bt(\K/\E,\xi)
\end{equation}
between isomorphism classes of admissible pairs of level zero over $\E$
and isomorphism classes~of irre\-ducible representations of $\BJ$ trivial 
on $\J^1$ whose restriction to $\J$ defines a supercuspidal~repre\-senta\-tion 
of $\GL_m(\ee)$ through \eqref{JJ1UU1GLmax}.
\end{prop}

\begin{rema}
\label{claudineremadm}
As in Remark \ref{claudinerem},
let us fix an embedding $\iota:\FC\to\FC'$ 
of algebraically~closed fields of characteristic $\ell$,
and let $(\K/\E,\xi)$ be an admissible pair of level zero over $\E$
such that $\xi$ takes values in $\FC$.
Then $(\K/\E,\iota\circ\xi)$ is an admissible pair of level zero over $\E$, 
and:
\begin{equation*}
\bt(\K/\E,\iota\circ\xi)=\bt(\K/\E,\xi)\otimes\FC'.
\end{equation*}
This refines the last assertion of Remark \ref{claudinerem}.
\end{rema}

\subsection{The $\s$-selfdual type theorem}
\label{S2}

Let us fix an integer $n\>1$ and write $\G=\GL_n(\F)$.
We recall the first main result of \cite{AKMSS}.

\begin{theo}[\cite{AKMSS} {Theorem 4.1}]
\label{PIMAIN}
Let $\pi$ be a cuspidal representation of $\G$. 
It is $\s$-selfdual if and only if it contains a type $(\BJ,\bl)$ 
such that $\BJ$ is $\s$-stable and $\bl^\s\simeq\bl^\vee$.
\end{theo}

\begin{rema}
\label{StandardStableType}
More precisely (see \cite{AKMSS} {Corollary 4.21}),
any $\s$-selfdual cuspidal representation contain a $\s$-selfdual type 
$(\BJ,\bl)$ with the additional property that 
$\BJ=\BJ(\aa,\b)$ for some maximal simple 
stratum $[\aa,\b]$ in $\Mat_n(\F)$ such that:
\begin{enumerate}
\item
the hereditary order $\aa$ is $\s$-stable and $\s(\b)=-\b$;
\item
the element $\b$ has the block diagonal form:
\begin{equation*}
\b=
\begin{pmatrix}
\b_0&&\\
&\ddots&\\
&&\b_0\\
\end{pmatrix}
=\b_0\otimes1\in\Mat_d(\F)\otimes_\F\Mat_m(\F)=\Mat_n(\F)
\end{equation*}
for some $\b_0\in\Mat_d(\F)$, where $d$ is the degree of $\b$ over $\F$
and $n=md$;
the centralizer $\B$ of $\E=\F[\b]$ in 
$\Mat_n(\F)$ thus identifies with $\Mat_m(\E)$,
equipped with the invo\-lu\-tion $\s$ acting componentwise;
\item
the order $\bb=\aa\cap\B$ is the standard maximal order of $\Mat_m(\E)$.
\end{enumerate}
Such a type will be useful in the discussion following Proposition 
\ref{classesdetypesstables7}.
\end{rema}

\begin{rema}
\label{StandardStableType2}
If $(\BJ,\bl)$ is any $\s$-selfdual type, 
then there is a maximal simple stratum $[\aa,\b]$ in $\Mat_n(\F)$ such that 
$\BJ=\BJ(\aa,\b)$, the order $\aa$ is $\s$-stable and $\s(\b)=-\b$
(see \cite{AKMSS} {Corollary 4.24}).
Such a maximal simple stratum will be said to be \textit{$\s$-selfdual}.
\end{rema}

\begin{rema}
Let $\pi$ be a $\s$-selfdual cuspidal representation of $\G$.
Let $(\BJ,\bl)$ be a $\s$-selfdual type in $\pi$, 
let $[\aa,\b]$ be a simple stratum such that $\BJ=\BJ(\aa,\b)$
and let $\t\in\Cc(\aa,\b)$ be the maximal sim\-ple character 
attached to $\bl$.
Then $\H^1(\aa,\b)$ is $\s$-stable and $\t\circ\s=\t^{-1}$.
\end{rema}

Let $\pi$ be a $\s$-selfdual cuspidal representation of $\G$.
Let $(\BJ,\bl)$ be a $\s$-selfdual type in $\pi$
and~fix a $\s$-selfdual simple stratum $[\aa,\b]$ as in Remark 
\ref{StandardStableType2}.
Then $\E=\F[\b]$ is $\s$-stable.
We deno\-te by $\E_0$ the field~of $\s$-fixed
points in $\E$, 
by $\T$ the maximal tamely ramified sub-ex\-ten\-sion of $\E$ over $\F$
and by $\T_0$ the intersection $\T\cap\E_0$.
Also write $d=[\E:\F]$ and $n=md$.

\begin{prop}[\cite{AKMSS} {Proposition 4.30}]
\label{TT0canonique}
The integer: 
\begin{equation}
\label{poolee}
m(\pi)=m=n/d
\end{equation} 
and the $\F_0$-isomorphism class of the quadratic extension $\T/\T_0$ only 
depend on $\pi$, and not on the choice of the $\s$-selfdual 
simple stratum $[\aa,\b]$ as in 
Remark \ref{StandardStableType2}. 
\end{prop}

The integer $m$ defined by \eqref{poolee} it called the 
\textit{relative degree} of $\pi$.
We record a list of pro\-per\-ties of the field extension $\T/\F$.

\begin{lemm}
\label{MaraDesBois}
\begin{enumerate}
\item
The canonical homomorphism of $\T_0\otimes_{\F_0}\F$-modules:
\begin{equation*}
\T_0\otimes_{\F_0}\F\to\T
\end{equation*}
is an isomorphism.
\item
If $\F/\F_0$ is unramified, then $\T/\T_0$ is unramified 
and $\T/\F$ has odd residual degree.
\item
The extension $\T/\T_0$ is ramified if and only if $\F/\F_0$ 
is ramified and $\T_0/\F_0$ has odd ramifi\-cation order.
\end{enumerate}
\end{lemm}

\begin{proof}
Assertion (1) is \cite{AKMSS} {Lemma 4.10}.
We now prove (2) and (3).

First, suppose that $\F/\F_0$ is unramified.
We remark that:
\begin{equation*}
f(\T/\F)\cdot f(\F/\F_0)=f(\T/\T_0)\cdot f(\T_0/\F_0)
\end{equation*}
is even.
Since $\F$ does not embed in $\T_0$ as an $\F_0$-algebra, 
$\T_0$ has odd residue degree over $\F_0$. It follows that $f(\T/\T_0)=2$ and that 
$\T$ has odd residue degree over $\F$.

Suppose $\F/\F_0$ is ramified, and let $\varpi$ be a uniformizer of $\F$ 
such that $\varpi_0=\varpi^2$ is a uni\-formi\-zer of $\F_0$.
Let $e_0$ be the ramification order of $\T_0/\F_0$,
and let $t_0$ be a uni\-for\-mi\-zer of $\T_0$ such that:
\begin{equation*}
\varpi_0^{}=t_0^{e_0}\zeta_0^{}
\end{equation*}
for some root of unity $\zeta_0^{}\in\F_0^\times$ of order prime to $p$.
Let $a$ be the greatest integer smaller than or equal to $e_0/2$,
and write $x=\varpi t_0^{-a}$.
We have $\s(x)=-x$, thus $x\notin\T_0$ and 
$x^2\in\T_0$.

If $e_0$ is odd, 
then $x^2=\zeta_0 t_0$ is a uniformi\-zer of $\T_0$, 
whereas $x$ is a uniformizer of $\T$, 
thus $\T$ is ramified over $\T_0$. 

If $e_0$ is even, then $x^2=\zeta_0$.
It follows that $x$ is a root of unity of order prime to $p$ which is 
in $\T$ but not in $\T_0$, thus $\T$ is unramified over $\T_0$.
\end{proof}

\begin{rema}
\label{Turron}
\begin{enumerate}
\item
The extensions $\E/\E_0$ and $\T/\T_0$ have the same ramification order. 
\item
The extension $\E/\E_0$ is ramified if and only if $\F/\F_0$ 
is ramified and $\E_0/\F_0$ has odd ramifica\-tion order.
\end{enumerate}
The first property comes from \cite{AKMSS} {Remark 4.22}, 
and the second one follows from the first one together with Lemma \ref{MaraDesBois}.
\end{rema}

We now recall the classification of all $\s$-selfdual types contained in a 
given $\s$-selfdual cuspidal representation of $\G$
(see \cite{AKMSS} {Proposition 4.31}).

\begin{prop}
\label{classesdetypesstables7}
Let $\pi$ be a $\s$-selfdual cuspidal representation of $\G$,
and let $\T/\T_0$ denote the quadra\-tic extension associated to it.
\begin{enumerate}
\item
If $\T$ is unramified over $\T_0$, 
the $\s$-selfdual types contained in $\pi$ 
form a single $\G^\s$-conjugacy class.
\item
If $\T$ is ramified over $\T_0$, 
the $\s$-selfdual types contained in $\pi$ 
form exactly $\lfloor m/2\rfloor+1$ different $\G^\s$-conjugacy classes.
\end{enumerate}
\end{prop}

One can give a more precise description in the ramified case. 
Suppose that $\T$ is ramified over $\T_0$, 
and let $(\BJ_0,\bl_0)$ be a $\s$-selfdual type in $\pi$ satisfying 
the conditions of Remark 
\ref{StandardStableType}. 
Let us~fix a uniformizer $t$ of $\E$.
For $i=0,\dots,\lfloor m/2\rfloor$, 
let $t_i$ denote the diagonal matrix:
\begin{equation*}
{\rm diag}(t,\dots,t,1,\dots,1)\in\B^\times=\GL_m(\E)
\end{equation*}
where $t$ occurs $i$ times.
Then the pairs $(\BJ_i,\bl_i)=(\BJ_0^{t_i},\bl_0^{t_i})$, for
$i=0,\dots,\lfloor m/2\rfloor$,
form a set~of~re\-pre\-senta\-ti\-ves of the $\G^\s$-conjugacy classes 
of $\s$-self\-dual types in $\pi$.

\begin{defi}
\label{defindex}
The integer $i$ is called the \textit{index} of the $\G^\s$-conjugacy class 
of $(\BJ_i,\bl_i)$.
It does not depend on the choice of $(\BJ_0,\bl_0)$, nor on that of $t$.
\end{defi}

Let $[\aa,\b]$ be a simple stratum as in Remark 
\ref{StandardStableType} such that $\BJ_0=\BJ(\aa,\b)$.
If one identifies the quotient $\J(\aa,\b)^{t_i}/\J^1(\aa,\b)^{t_i}$
with $\GL_m(\ee)$ via:
\begin{equation*}
\J(\aa,\b)^{t_i}/\J^1(\aa,\b)^{t_i} 
\simeq
\J(\aa,\b)/\J^1(\aa,\b)
\simeq
\U(\bb)/\U^1(\bb)
\simeq
\GL_m(\ee)
\end{equation*}
then $\s$ acts on $\GL_m(\ee)$ by conjugacy by the diagonal element:
\begin{equation*}
\d_i = {\rm diag}(-1,\dots,-1,1,\dots,1)\in\GL_m(\ee)
\end{equation*}
where $-1$ occurs $i$ times,
and $(\J(\aa,\b)^{t_i}\cap\G^\s)/(\J^1(\aa,\b)^{t_i}\cap\G^\s)$ 
identifies with the Levi subgroup $(\GL_i\times\GL_{m-i})(\ee)$ of 
$\GL_m(\ee)$.

\subsection{Admissible pairs and $\s$-selfduality}
\label{P45}

Let $(\BJ,\bl)$ be a $\s$-selfdual type in $\G$.
Fix a $\s$-selfdual 
maximal simple stratum $[\aa,\b]$ such that $\BJ=\BJ(\aa,\b)$ 
as in Remark \ref{StandardStableType2}, 
and a decomposition of $\bl$ of the form $\bk\otimes\bt$ as 
in Paragraph \ref{pifoulechien}.
Write $\E=\F[\b]$ as usual.

\begin{prop}
\label{charaghdinS4}
Suppose that the representation $\bt$ is $\s$-selfdual, 
and let $(\K/\E,\xi)$ be an~ad\-mis\-si\-ble pair of level zero attached to it. 
There is a unique involutive $\E_0$-automorphism $\widehat{\s}$ of $\K$ 
such that $\xi\circ\widehat{\s}=\xi^{-1}$ and $\widehat{\s}$ coincides with
$\s$ on $\E$. 
\end{prop}

\begin{proof}
Let $\K'$ denote the extension of $\E$ given by the field $\K$ equipped 
with the map $x\mapsto\s(x)$ from $\E$ to $\K$.
Then the pair $(\K'/\E,\xi)$ is admissible of level zero, 
and it is attached to $\bt^{\s}$.
On the other hand, $(\K/\E,\xi^{-1})$ is admissible of level zero, 
attached to $\bt^{\vee}$.
Since $\bt$ is $\s$-selfdual, there~is an $\E$-algebra 
isomorphism $\widehat{\s}:\K\to\K'$ 
such that $\xi\circ\widehat{\s}=\xi^{-1}$.
We thus have:
\begin{equation*}
\xi\circ\widehat{\s}^2=\xi^{-1}\circ\widehat{\s}=\xi
\end{equation*}
and $\widehat{\s}^2$ is an $\E$-algebra automorphism of $\K$.
By admissibility of $(\K/\E,\xi)$, 
the latter automorphism is trivial,
thus $\widehat{\s}$ satisfies the required conditions.
Uniqueness follows by admissibility again.
\end{proof}

For simplicity, we will write $\s$ for the involutive automorphism given 
by Proposition \ref{charaghdinS4}.
Let $\K_0$ be the field of $\s$-fixed points of $\K$.
The following lemma will be useful in Section~\ref{S7}.

\begin{lemm}
\label{knr}
If $\E/\E_0$ is ramified and $m$ is even, 
then $\K/\K_0$ is unramified.
\end{lemm}

\begin{proof}
Write $m=2r$ for some $r\>1$. 
Let $t$ be a uniformizer of $\E$ such that $\s(t)=-t$ and~let $\zeta\in\K$ 
be a root of unity of order $c^m-1$, where $c$ is~the car\-di\-na\-lity of 
$\ee$. 
We thus have $\E=\E_0[t]$ and $\K=\E[\zeta]$.
Since $\s$ is involutive, it induces an involutive $\ee$-automorphism 
of $\ll$, the re\-si\-dual field of $\K$.
If the latter were trivial, the relation $\xi\circ\s=\xi^{-1}$ would 
imply that the character $\overline{\xi}$ of $\ll^\times$ induced by $\xi$ 
is quadratic, contradicting the fact that it is $\ee$-regular.
The auto\-mor\-phism of $\ll$ induced by $\s$ is thus 
the $r$th power of the Frobenius automorphism.
Now consider the element:
\begin{equation*}
\a=\zeta^{(c^{r}+1)/2}. 
\end{equation*} 
It has order $2(c^r-1)$, thus $\s(\a)=-\a$.
Since $\a^2$ has order $c^r-1$,
the ex\-ten\-sion of $\E_0$ it~ge\-ne\-ra\-tes is un\-ra\-mi\-fied and has 
degree $r$. 
We thus have $\E_0[\a^2,t\a]\subseteq\K_0$ and their degrees are equal.
Now we deduce that $\K=\K_0[\a]=\K_0[\zeta]$ 
is un\-ra\-mi\-fied over $\K_0$.
\end{proof}

\subsection{}

The following 
lemma will be useful in Sections \ref{SECDISTTR} and 
\ref{SECDIST}, when we investigate decompositions~of $\s$-self\-dual 
types of the form $\bk\otimes\bt$ 
which behave well under $\s$.

Let $\t\in\Cc(\aa,\b)$ be a maximal simple character 
such that $\H^1(\aa,\b)$ is $\s$-stable and $\t\circ\s=\t^{-1}$.
Let $\BJ$ be its normalizer in $\G$,
let $\J^1$ be the maximal normal pro-$p$-subgroup~of~$\BJ$
and $\n$ be the ir\-reducible representation of $\J^1$ 
containing $\t$.

\begin{lemm}
\label{chouxfarci}
Let $\bk$ be a representation of $\BJ$~ex\-tend\-ing $\n$. 
There is a unique character $\xx$ of $\BJ$ trivial on $\J^1$ 
such that $\bk^{\s\vee}\simeq\bk\xx$.
It satisfies the identity $\xx\circ\s=\xx$.
\end{lemm}

\begin{proof}
Let $\bk$ be an irreducible representation of $\BJ$ extending $\n$.
By uniqueness of the Heisenberg re\-pre\-sen\-tation, 
the fact that $\t\circ\s=\t^{-1}$ implies that $\n^{\s\vee}$ is isomorphic to $\n$.
Thus $\bk$ and $\bk^{\s\vee}$ are re\-pre\-sentations of $\BJ$ 
extending $\n$.
There is a character $\xx$ of $\BJ$ trivial on 
$\J^1$ such that we have $\bk^{\s\vee}\simeq\bk\xx$.
It satisfies $\xx\circ\s=\xx$.
It is unique by Proposition \ref{suzanne}.
\end{proof}

\section{The distinguished type theorem}
\label{SEC6}

In this section we prove the following result,
which is our first main theorem.
It will be~re\-fi\-ned by Theorem \ref{genericDSTT} in Section \ref{S7}.
Recall that $p\neq2$ until the end of this article.

\begin{theo}
\label{DSTT}
Let $\pi$ be a $\s$-selfdual cuspidal representation of $\G$.
Then $\pi$ is dis\-tin\-guished~if and only if it contains 
a $\s$-selfdual type $(\BJ,\bl)$ 
such that $\Hom_{\BJ\cap\G^\s}(\bl,1)$ is non-zero.
\end{theo}

\begin{rema}
If $\pi$ is distinguished, it follows easily from the multiplicity $1$ 
property in 
Theorem \ref{FP} that the distinguished $\s$-selfdual types $(\BJ,\bl)$ 
occurring in $\pi$ form a single $\G^\s$-con\-ju\-gacy class
(see Remark \ref{chienaphasique}).
\end{rema}

\begin{rema}
\label{minutes}
Theorem \ref{DSTT} is proved in \cite{AKMSS} in a different manner
than the one we give~here,
although both proofs use the $\s$-selfdual~type theorem \ref{PIMAIN}.
The proof given in \cite{AKMSS}
is based on~a~re\-sult of Ok \cite{Ok},
proved by Ok for complex representations 
and extended to the modular~case~in~\cite{AKMSS} Appendix B.
However, 
the proof we give here is more likely to generalize 
to other~si\-tua\-tions.
\end{rema}

Let $\pi$ be a $\s$-selfdual cuspidal  representation.
Theorem \ref{PIMAIN} tells us that it contains a $\s$-selfdual type 
$(\BJ,\bl)$, 
and Proposition \ref{GeorgesSaval} tells us that
$\pi$ is compactly induced from $\bl$. 
A simple~appli\-ca\-tion of the Mackey formula gives us:
\begin{equation}
\label{MACKEYFONDA}
\Hom_{\G^\s}(\pi,1) \simeq \prod\limits_{g} 
\Hom_{\BJ^g\cap\G^\s}(\bl^g,1)
\end{equation}
where $g$ ranges over a set of representatives of $(\BJ,\G^\s)$-double cosets 
in $\G$.

\begin{rema}
It follows from Theorem \ref{FP} that there is at most one double coset $\BJ g\G^\s$ 
such that the space $\Hom_{\BJ^g\cap\G^\s}(\bl^g,1)$ is non-zero, 
and that this space has dimension at most $1$.
Thus the product in \eqref{MACKEYFONDA} is actually a direct sum. 
\end{rema}

In this section, our main task (achieved in Paragraph \ref{DCT})
is to prove that, if $\Hom_{\BJ^g\cap\G^\s}(\bl^g,1)$ is non-zero, 
then $\s(g)g^{-1}\in\BJ$.
Theorem \ref{DSTT} will follow easily from there
(see Paragraph \ref{viteuntitre}). 

We may assume that $\BJ=\BJ(\aa,\b)$ for a maximal simple stratum 
$[\aa,\b]$ satisfying the conditions of Remark \ref{StandardStableType}. 
The extension $\E=\F[\b]$, its centralizer $\B$ 
and the maximal order $\bb=\aa\cap\B$~are thus stable by $\s$. 
We write $d=[\E:\F]$ and $n=md$.
We identify $\B$ with the $\E$-algebra $\Mat_m(\E)$
equipped with the involution $\s$ acting componentwise, 
and $\bb$ with its stan\-dard maximal order. 

We write $\E_0=\E^\s$, the field of $\s$-invariant elements of $\E$, 
and fix once and for all a uni\-for\-mi\-zer $t$ of $\E$ such that:
\begin{equation}
\s(t)=
\left\{
\begin{array}{rl}
t & \text{if $\E$ is unramified over $\E_0$,} \\
-t & \text{if $\E$ is ramified over $\E_0$.} \\
\end{array}
\right.
\end{equation}
We also write 
$\J=\J(\aa,\b)$, $\J^1=\J^1(\aa,\b)$ and $\H^1=\H^1(\aa,\b)$. 
Recall that $\BJ=\E^\times\J$.

We denote by $\T$ the maximal tamely ramified
sub-extension of $\E$ over $\F$, and set $\T_0=\T\cap\E_0$.

\textit{We insist on the fact that, throughout this section, 
we assume that the stratum $[\aa,\b]$ satisfies the conditions of 
Remark \ref{StandardStableType}}.

\subsection{Double cosets contributing to the distinction 
of $\t$}
\label{P51}

Let $\t\in\Cc(\aa,\b)$ be the maximal simple cha\-rac\-ter occurring 
in the restriction of $\bl$ to $\H^1$. Sup\-po\-se that 
$\Hom_{\BJ^g\cap\G^\s}(\bl^g,1)$ is non-zero for some 
double coset $\BJ g\G^\s$.
Restricting to $\H^{1g}\cap\G^\s$,~we 
deduce that the character $\t^g$ is trivial on $\H^{1g}\cap\G^\s$.

In this paragraph, we look for the double cosets $\BJ g\G^\s\subseteq\G$ such
that the character $\t^g$ is trivial on $\H^{1g}\cap\G^\s$.
For this, let us introduce the following general lemma. 

\begin{lemm}
\label{bathmologieabstraite}
Let $\tau$ be an involution of $\G$,
let $\H$ be a $\tau$-stable open pro-$p$-subgroup of $\G$
and let $\chi$ be a character of $\H$ 
such that $\chi\circ\tau=\chi^{-1}$.
For any $g\in\G$,
the character $\chi^g$ is trivial on $\H^{g}\cap\G^\tau$ if and only 
if $\tau(g)g^{-1}$ intertwines $\chi$.
\end{lemm}

\begin{proof}
Write $\K$ for the $\tau$-stable subgroup 
$\H^{g}\cap\tau(\H^{g})$, which contains $\H^{g}\cap\G^\tau$. 
Let $\A$ be the quotient of $\K$ by $\overline{[\K,\K]}$, the closure of the 
derived subgroup of $\K$.
This is a $\tau$-stable commutative pro-$p$-group.
Given $x\in\K$, write ${x'}$ for its image in $\A$.
For any $b\in\A$, we have:
\begin{equation*}
b = \sqrt{b\tau(b)}\cdot
\sqrt{b\tau(b)^{-1}}
\end{equation*}
where $b\mapsto\sqrt{b}$ is the inverse of the automorphism 
$b\mapsto b^2$ of $\A$.
Thus, for any $x\in\K$, there are $y,z\in\K$ such that $x=yz$ 
and $\tau({y'})={y'}$ 
and $\tau({z'})=z'^{-1}$.

Since $\tau(z)=z^{-1}h$ for some $h\in\overline{[\K,\K]}$, we have:
\begin{equation}
\label{Z1}
\chi^g(\tau(z)) = \chi^g(z^{-1}h) = \chi^g(z)^{-1}.
\end{equation}
On the other hand, since $\tau(y)=yk$ for some $k\in\overline{[\K,\K]}$, 
the element $y^{-1}\tau(y)$ defines a $1$-cocy\-cle in 
the $\tau$-stable pro-$p$-group $\overline{[\K,\K]}$.
Since $p\neq2$, this cocycle is a coboun\-dary, which im\-plies:
\begin{equation}
\label{Z2}
y\in(\H^{g}\cap\G^\tau)\overline{[\K,\K]}.
\end{equation}
Now suppose that $\chi^g$ is trivial on $\H^{g}\cap\G^\tau$. 
Then \eqref{Z1} and \eqref{Z2} imply that:
\begin{equation}
\label{W1}
\chi^g(\tau(x)) = \chi^g(\tau(z)) 
= \chi^g(z)^{-1} = \chi^g(x)^{-1},
\quad
\text{for all $x\in\K$}.
\end{equation}
Besides, \eqref{W1} is \textit{equivalent} to 
$\chi^g$ being trivial on $\H^{g}\cap\G^\tau$.
On the other hand, we have:
\begin{equation}
\label{W2} 
\chi^g\circ\tau=(\chi\circ\tau)^{\tau(g)}
=(\chi^{-1})^{\tau(g)}=(\chi^{\tau(g)})^{-1}
\end{equation}
on $\K$ by assumption on $\chi$.
If we set $\g=\tau(g)g^{-1}$, then \eqref{W1} 
is equivalent to:
\begin{equation*}
\chi(h)=\chi^{\g}(h)
\quad
\text{for all $h\in\H\cap \g^{-1}\H\g$}.
\end{equation*}
This amounts to saying that $\g$ intertwines $\chi$.
\end{proof}

\begin{prop}
\label{bathmologie}
Let $g\in\G$.
Then the character $\t^g$ is trivial on $\H^{1g}\cap\G^\s$ if and only 
if we have $\s(g)g^{-1}\in\J\B^\times\J$.
\end{prop}

\begin{proof}
This follows from Lemma \ref{bathmologieabstraite} 
applied to the simple character $\t$ of $\H^1$ and the involution $\s$, 
together with the fact that the intertwining set of $\t$ is 
$\J\B^\times\J$ by Proposition \ref{patel}\eqref{bourrin6}. 
\end{proof}

\subsection{The double coset lemma}
\label{DCL}

We now prove the following fundamental lemma.

\begin{lemm}
\label{FL}
Let $g\in\G$.
Then $\s(g)g^{-1}\in\J\B^\times\J$ if and only if $g\in\J\B^\times\G^\s$.
\end{lemm}

\begin{proof}
Write $\g=\s(g)g^{-1}$. 
If $g\in\J\B^\times\G^\s$, one verifies immediately that 
$\g\in\J\B^\times\J$. Con\-verse\-ly, suppose that $\g\in\J c\J$ 
for some $c\in\B^\times$.
We will first show that the double coset representative $c$ can be chosen nicely. 

\begin{lemm}
\label{rouge}
There is a $b\in\B^\times$ such that $\g\in\J b\J$ and $b\s(b)=1$.
\end{lemm}

\begin{proof}
Recall that $\B^\times$ has been identified with $\GL_m(\E)$ 
and $\U=\J\cap\B^\times=\bb^\times$ is its stan\-dard maximal compact 
subgroup.
By the Cartan decomposition, $\B^\times$ decom\-poses as 
the disjoint union of the double cosets:
\begin{equation*}
\U\cdot
\text{diag}(t^{a_1},\dots,t^{a_m})
\cdot\U
\end{equation*}
where $a_1\>\dots\>a_m$ ranges over non-increasing sequences of $m$ integers, 
and $\text{diag}(\l_1,\dots,\l_m)$ denotes the diagonal matrix of 
$\B^\times$ with eigenvalues $\l_1,\dots,\l_m\in\E^\times$.
We thus may assume that $c=\text{diag}(t^{a_1},\dots,t^{a_m})$ for a 
uniquely determined sequence of integers $a_1\>\dots\>a_m$. 

The fact that $\s(\g)=\g^{-1}$ implies that we have 
$c\in\J c^{-1}\J\cap\B^\times$.
Using the simple intersection proper\-ty \eqref{SIP} 
together with the fact that $\J=\U\J^1$ and $\J^1\subseteq\U^1(\aa)$, 
we have $\J c^{-1}\J\cap\B^\times=\U c^{-1}\U$. 
The uniqueness of the Cartan decomposition of $\B^\times$ 
thus implies that the sequences 
$a_1\>\dots\>a_m$ and $-a_m\>\dots\>-a_1$ are equal.
We thus have $a_i+a_{m+1-i}=0$ for all $i\in\{1,\dots,m\}$.
Now write $\k=\s(t)t^{-1}\in\{-1,1\}$ and choose signs 
$\k_1,\dots,\k_m\in\{-1,1\}$ such that 
$\k_i\k_{m+1-i}=\k^{a_i}$ for all $i$.
This is always possible since $a_{(m+1)/2}=0$ when $m$ is odd. 
Then the anti\-diagonal element:
\begin{equation}
\label{bantidiagonal}
b=
\begin{pmatrix}
&&\k_1t^{a_1} \\
&\iddots&\\
\k_mt^{a_m}&&
\end{pmatrix}\in\B^\times
\end{equation}
satisfies the required conditions $b\s(b)=1$ and $\g\in \J b\J$. 
\end{proof}

Now write $\g=x'bx$ with $x,x'\in\J$ and $b\in\B^\times$. 
Replacing $g$ by $\s(x')^{-1}g$ does not chan\-ge the 
double coset $\J g\G^\s$ but changes $\g$ into $bx\s(x')$. 
From now on, we will thus assume that:
\begin{equation}
\label{VGE}
\g=bx,
\quad
b\s(b)=1,
\quad
x\in\J,
\quad
\text{$b$ is of the form \eqref{bantidiagonal}}.
\end{equation}
Write $\Kk$ for the group $\J\cap b^{-1}\J b$.
Since $\s(b)=b^{-1}$ and $\J$ is $\s$-stable, we have $x\in\Kk$.

\begin{lemm}
\label{lemmadelta}
The map $\delta : k \mapsto b^{-1}\s(k)b$
is an involutive group automorphism of $\Kk$.
\end{lemm}

\begin{proof}
This follows from an easy calculation using the fact that $b\s(b)=1$.
\end{proof}

Let $b_1>\dots>b_r$ be the unique decreasing 
sequence of integers such that:
\begin{equation*}
\{a_1,\dots,a_m\}=\{b_1,\dots,b_r\}
\end{equation*}
and $m_j$ denote the multiplicity of $b_j$ in $(a_1,\dots,a_m)$, for 
$j\in\{1,\dots,r\}$. 
We have $m_{j}=m_{r+1-j}$ for all $j$, and $m_1+\dots+m_r=m$. 
These integers define a standard Levi subgroup:
\begin{equation}
\label{Detroit}
\M=\GL_{m_1 d}(\F)\times\dots\times\GL_{m_r d}(\F)\subseteq\G.
\end{equation}
Write $\P$ for the standard parabolic subgroup of $\G$ 
generated by $\M$ and upper trian\-gular matri\-ces. 
Let $\N$ and $\N^-$ denote the uni\-po\-tent ra\-di\-cals of $\P$ and 
its opposite parabolic subgroup with respect to $\M$, 
respectively. 
Since $b$ has the form \eqref{bantidiagonal}, 
it nor\-ma\-lizes $\M$ and we have:
\begin{align*}
\Kk &= (\Kk\cap\N^-)\cdot(\Kk\cap\M)\cdot(\Kk\cap\N), \\
\Kk\cap\P&= \J\cap\P, \\
\Kk\cap\N^- &\subseteq \J^1\cap\N^-.
\end{align*}
We have similar properties for the subgroup
$\Vv=\Kk\cap\B^\times=\U\cap b^{-1}\U b$ of $\B^\times$, 
that is:
\begin{align*}
\Vv &= (\Vv\cap\N^-)\cdot(\Vv\cap\M)\cdot(\Vv\cap\N), \\
\Vv\cap\P&= \U\cap\P, \\
\Vv\cap\N^- &\subseteq \U^1\cap\N^-,
\end{align*}
where $\U^1=\J^1\cap\B^\times=\U^1(\bb)$.
Note that this subgroup $\Vv$ is stable by $\delta$.

\begin{lemm}
The subset:
\begin{equation*}
\Kk^1 
= (\Kk\cap\N^-)\cdot(\J^1 \cap \M)\cdot(\K\cap\N)
\end{equation*}
is a $\delta$-stable normal pro-$p$-subgroup of $\Kk$,
and we have $\Kk=\Vv\Kk^1$. 
\end{lemm}

\begin{proof}
To prove that $\Kk^1$ is a subgroup of $\Kk$, it is enough to prove that 
one has the containment 
$(\K\cap\N)\cdot(\Kk\cap\N^-)\subseteq\Kk^1$.
Let $\JJ^1$ be the sub-$\Oo$-lattice of $\aa$ 
such that $\J^1=1+\JJ^1$ and let $\JJ=\bb+\JJ^1$, 
thus $\J=\JJ^\times$.
A simple computation shows that
$\Kk\cap\N^-\subseteq (1+t\JJ)\cap\N^-$
and: 
\begin{equation*}
(\K\cap\N)\cdot(\Kk\cap\N^-)
\subseteq(\Kk\cap\N^-)\cdot((1+t\JJ)\cap\M)\cdot(\K\cap\N).
\end{equation*}
The expected result thus follows from the fact that 
$t\JJ\subseteq\JJ^1$.
Besides, $\K^1$ is a $\delta$-stable pro-$p$-group.

Since $\Vv\cap\M$ normalizes $\Kk\cap\N^-$, 
$\Kk\cap\N$ and $\Kk^1\cap\M$,
we have $(\Vv\cap\M)\Kk^1=\Kk$
whence $\Kk^1$ is normal in $\Kk$
and $\Kk=\Vv\Kk^1$, as expected.
\end{proof}

The subgroup $\Kk^1$ is useful in the following lemma.
Note that we have $x\delta(x)=1$.

\begin{lemm}
\label{boxon}
Let $y\in\K$ be such that $y\delta(y)=1$.
There are $k\in\Kk$ and $v\in\Vv$ such that:
\begin{enumerate}
\item
the element $v$ is diagonal in $\B^\times$
with eigenvalues in $\{-1,1\}$ and it satisfies $v\delta(v)=1$, 
\item
one has $\delta(k)yk^{-1}\in v\Kk^1$.
\end{enumerate}
\end{lemm}

\begin{proof}
Let $\Vv^1=\Vv\cap\Kk^1=\K^1\cap\B^\times$.
We have:
\begin{equation*}
\Vv^1
=(\Vv\cap\N^-)\cdot(\U^1 \cap \M)\cdot(\U\cap\N).
\end{equation*}
We thus have canonical $\delta$-equivariant group isomorphisms:
\begin{equation}
\label{Mitterand}
\Kk/\Kk^1\simeq\Vv/\Vv^1 \simeq (\U\cap\M)/(\U^1\cap\M).
\end{equation}
By \eqref{Detroit}, we have 
$\M\cap\B^\times=\GL_{m_1}(\E)\times\dots\times\GL_{m_r}(\E)$,
thus the right hand side of \eqref{Mitterand}~iden\-ti\-fies with 
$\EuScript{M}=\GL_{m_1}(\ee)\times\dots\times\GL_{m_r}(\ee)$, 
where $\ee$ denotes the residue field of $\E$.
Besides, since~$b$ is given by \eqref{bantidiagonal}, 
the involution $\delta$ acts on $\EuScript{M}$ as:
\begin{equation*}
(g_1,\dots,g_r)\mapsto(\s(g_r),\dots,\s(g_1)).
\end{equation*}
Write $y=vy'$ for some $v\in\Vv$ and $y'\in\Kk^1$.
The simple intersection property \eqref{SIP} gives us:
\begin{equation*}
\delta(v)^{-1}=\delta(y')vy'\in\Vv\cap\Kk^1v\Kk^1=\Vv^1v\Vv^1.
\end{equation*}
Thus there is $u\in\Vv^1$ such that $vu\delta(vu)\in\Vv^1$.
Replacing $(v,y')$ by $(vu,u^{-1}y')$, we may and will assume that $y=vy'$ 
with $v\delta(v)\in\Vv^1$.

We now compute the first cohomology set of $\delta$ 
in $\EuScript{M}$.
Let $w=(w_1,\dots,w_r)$ denote the image of $y$ in $\EuScript{M}$. 
We have $w\delta(w)=1$, that is:
\begin{equation*}
\s(w_j^{})=w_{r+1-j}^{-1},
\quad
\text{ for all $j\in\{1,\dots,r\}$}.
\end{equation*}
If $r$ is even, one can find 
an element $a\in\EuScript{M}$ such that 
$w=\delta(a)a^{-1}$.
If $r$ is odd, say $r=2s-1$,~one can find an element
$a\in\EuScript{M}$ such that:
\begin{equation*}
\delta(a)wa^{-1}=(1,\dots,1,w_s,1,\dots,1)
\end{equation*}
and we have $w_s\s(w_s)=1$.
If $\E/\E_0$ is unramified, then $\ee$ is quadratic over the 
residue field of $\E_0$, and it follows from the triviality of the 
first cohomology set of $\s$ in $\GL_{m_s}(\ee)$
that $w=\s(c)c^{-1}$ for some $c\in\EuScript{M}$.
In these two cases, we thus may find 
$k\in\Kk$ such that $\delta(k)xk^{-1}\in\Kk^1$.

It remains to treat the case where $r$ is odd and $\E/\E_0$ is ramified. 
In this case, we have $w_s^2=1$, thus $w_s$ is conjugate in $\GL_{m_s}(\ee)$
to a diagonal element with eigenvalues $1$ and $-1$.
Let $i$ denote the multiplicity of $-1$. 
Let:
\begin{equation*}
v\in\U\cap\M = 
\GL_{m_1}(\Oo_\E)\times\dots\times\GL_{m_r}(\Oo_\E)
\end{equation*}
(here $\Oo_\E$ is the ring of integers of $\E$) 
be a diagonal matrix with eigenvalues $1$ and $-1$, 
such that $-1$ occurs with multiplicity 
$i$ and only in the $s$th block.
Then $v\delta(v)=1$
and there is $k\in\Kk$ such that $\delta(k)yk^{-1}\in v\Kk^1$.
\end{proof}

Applying Lemma \ref{boxon} to $x$ gives us $k\in\K$,
$v\in\V$ such that $bv\s(bv)=1$ and $\delta(k)xk^{-1}\in v\Kk^1$.
Besides, $bv$ is antidiagonal of the form \eqref{bantidiagonal}
and $\s(k)\g k^{-1}\in bv\Kk^1$.
Therefore, replacing $g$ by $kg$, which does not chan\-ge the 
double coset $\J g\G^\s$, we will assume that $\g$ can be 
written:
\begin{equation}
\label{VGE1}
\g=bx,
\quad
b\s(b)=1,
\quad
x\in\J^1,
\quad
\text{$b$ is of the form \eqref{bantidiagonal}}.
\end{equation}
Comparing with \eqref{VGE}, we now have a stronger condition on $x$,
that is $x\delta(x)=1$ and $x\in\Kk^1$. 

Since $\Kk^1$ is a $\d$-stable pro-$p$-group and $p$ is odd,
the first cohomology set of $\d$ in $\Kk^1$ is trivial.
Thus $x=\delta(y)y^{-1}$ for some $y\in\Kk^1$, 
hence $\g=\s(y)by^{-1}$.
Since $b\s(b)=1$ and the first cohomo\-logy set of $\s$ in $\B^\times$ is 
trivial, 
one has $b=\s(h)h^{-1}$ for some $h\in\B^\times$.
Thus $g\in yh\G^\s\subseteq\J \B^\times\G^\s$, and Lem\-ma \ref{FL} is proved.
\end{proof}

\subsection{Contribution of the Heisenberg representation} 
\label{P53}

Let $\n$ be the Heisenberg representation of $\J^1$ associated to $\t$
(see \S\ref{pifoulechien}).
In this para\-graph, we prove the following result.

\begin{prop}
\label{MultOneEta}
Given $g\in\G$, we have:
\begin{equation*}
\dim \Hom_{\J^{1g}\cap\G^\s}(\n^g,1) = 
\left\{
\begin{array}{cl}
1 & \text{if $g\in\J\B^\times\G^\s$}, \\
0 & \text{otherwise}.
\end{array}
\right.
\end{equation*}
\end{prop}

\begin{proof}
Suppose that $\Hom_{\J^{1g}\cap\G^\s}(\n^g,1)$ is non-zero. 
Restricting to $\H^{1g}\cap\G^\s$, 
the character~$\t^g$ is trivial on $\H^{1g}\cap\G^\s$,
and Proposition \ref{bathmologie} together with Lemma \ref{FL}
give us $g\in\J\B^\times\G^\s$. 
Con\-ver\-sely, assume that $g\in\J\B^\times\G^\s$.
Since   the   dimension    of   $\Hom_{\J^{1g}\cap\G^\s}(\n^g,1)$   does   not
chan\-ge~when~$g$ varies in a given $(\BJ,\G^\s)$-double coset, we may and 
will assume 
that we have $g\in\B^\times$.
Thus we have $\g=\s(g)g^{-1}\in\B^\times$ as well. 

\begin{lemm}
\label{LouElias1}
The map $\tau:x\mapsto\g^{-1}\s(x)\g$ 
is an involutive automorphism of $\G$ and,
for any subgroup $\H\subseteq\G$,
we have $\H^g\cap\G^\s=(\H\cap\G^\tau)^g$.
\end{lemm}

\begin{proof}
This follows from an easy calculation using the fact that $\s(\g)=\g^{-1}$.
\end{proof}

Our goal is thus to prove that the space $\Hom_{\J^{1}\cap\G^{\tau}}(\n,1)$
has dimension $1$.
By Paragraph \ref{pifoulechien}, the representation of $\J^1$ induced from 
$\t$ decomposes as the direct sum of $(\J^1:\H^1)^{1/2}$ copies of the 
representation $\n$.
The space:
\begin{equation}
\label{labello}
\Hom_{\J^{1}\cap\G^{\tau}}\left(\Ind^{\J^{1}}_{\H^{1}}(\t),1\right) 
\end{equation} 
thus decomposes as the direct sum of $(\J^1:\H^1)^{1/2}$ copies of 
$\Hom_{\J^{1}\cap\G^\tau}(\n,1)$.
Applying the~Fro\-be\-nius reci\-pro\-city and the Mackey formula, the space 
\eqref{labello} is isomorphic to:
\begin{equation*}
\Hom_{\J^{1}}
\left(\Ind^{\J^{1}}_{\H^{1}}(\t),\Ind^{\J^{1}}_{\J^{1}\cap\G^\tau}(1)\right) 
\simeq
\bigoplus\limits_{x\in\X} \Hom_{\H^{1}}
\left(\t,\Ind^{\H^{1}}_{\H^{1}\cap(\J^{1}\cap\G^\tau)^x}(1)\right)
\end{equation*}
where $\X$ is equal to $\J^{1}/(\J^{1}\cap\G^\tau)\H^{1}$
(recall that $\H^1$ is normal in $\J^1$ and $\J^1/\H^1$ is abelian).
Since $\J^{1}$ normalizes $\t$, this is isomorphic to:
\begin{equation}
\label{passear}
\bigoplus\limits_{x\in\X} 
\Hom_{\H^{1}}\left(\t,\Ind^{\H^{1}}_{\H^{1}\cap\G^\tau}(1)\right) 
\simeq
\bigoplus\limits_{x\in\X} \Hom_{\H^{1}\cap\G^\tau}(\t,1). 
\end{equation}
Since $\Hom_{\H^{1}\cap\G^\tau}(\t,1)$ has dimension $1$, 
the right hand side of \eqref{passear} has dimension the cardina\-lity of $\X$. 
It thus remains to prove that $\X$ has cardinality 
$(\J^1:\H^1)^{1/2}$, or equivalently:
\begin{equation}
\label{CharlesDeGaulle}
(\J^{1}\cap\G^\tau:\H^{1}\cap\G^\tau) = (\J^1:\H^1)^{\frac{1}{2}}.
\end{equation}
Now consider the groups $\J^1\cap\J^{1\g}$ and $\H^1\cap\H^{1\g}$,
which are both stable by $\tau$.

\begin{lemm}
\label{L2now}
We have $\t(\tau(x))=\t(x)^{-1}$ for all $x\in\H^1\cap\H^{1\g}$.
\end{lemm}

\begin{proof}
Given $x\in\H^1\cap\H^{1\g}$, 
and using the fact that $\t\circ\s=\t^{-1}$ on $\H^1$, we have:
\begin{equation*}
\t(\tau(x))^{-1} = \t\circ\s(\tau(x)) = \t^{\g}(x) = \t(x)
\end{equation*}
since $\g\in\B^\times$ intertwines $\t$.
\end{proof}

Let us write $\VV$ for the $\kk$-vector space
$(\J^1\cap\J^{1\g})/(\H^1\cap\H^{1\g})$ equipped with both 
the involution $\tau$ and the sym\-plec\-tic form 
$(x,y)\mapsto\langle x,y\rangle$ induced by 
\eqref{SYMPLEC}. 
We write $\VV^+=\{v\in\VV\ |\ \tau(v)=v\}$ and 
$\VV^-=\{v\in\VV\ |\ \tau(v)=-v\}$.
We have the decomposition $\VV=\VV^+\oplus\VV^-$ since $p\neq2$. 

\begin{lemm}
\label{willycolette}
There is a group isomorphism
$\VV^+\simeq (\J^1\cap\G^\tau)/(\H^1\cap\G^\tau)$.
\end{lemm}

\begin{proof}
  First note that we have the containment
  $(\J^1\cap\G^\tau)(\H^1\cap\H^{1\g})/(\H^1\cap\H^{1\g})\subseteq\VV^+$.
  The lemma will follow if we prove that this containment is an equality.
  Let $x\in \J^1\cap\J^{1\g}$ be~such~that
  $x (\H^1\cap\H^{1\g})\in\VV^+$.
  One thus has $x^{-1}\tau(x)\in\H^1\cap\H^{1\g}$.
  Since $\H^1\cap\H^{1\g}$ is a $\tau$-stable pro-$p$-group and $p\neq2$,
  there is an $h\in \H^1\cap\H^{1\g}$ such that $x^{-1}\tau(x)=h^{-1}\tau(h)$.
  The expected result follows.
\end{proof}

We are now going to prove that $\VV^+$ and $\VV^-$ have the same dimension. 

\begin{lemm}
The subspaces $\VV^+$ and $\VV^-$ are totally isotropic.
\end{lemm}

\begin{proof}
Indeed, thanks to Lemma \ref{L2now}, first note that:
\begin{equation}
\label{Chirac}
\langle \tau(x),y \rangle = \langle \tau(y),x \rangle,
\quad
x,y\in\VV.
\end{equation}
If $x,y\in\VV^+$, then we get 
$\langle x,y \rangle = \langle x,y \rangle^{-1}$,
thus $\langle x,y \rangle = 1$ since $p\neq2$.
If $x,y\in\VV^-$, then we get 
$\langle x^{-1},y \rangle = \langle x,y^{-1} \rangle^{-1}$.
But $\langle x^{-1},y \rangle =\langle x,y \rangle^{-1}
=\langle x,y^{-1}\rangle$.
It follows again that $\langle x,y \rangle = 1$.
\end{proof}

Let $\XX$ denote the kernel 
of the symplectic form $(x,y)\mapsto\langle x,y\rangle$ on $\VV$,
that is:
\begin{equation*}
\XX = \{w\in\VV\ |\ \langle w,v\rangle=1 \text{ for all } v\in\VV\}.
\end{equation*}
Let $\YY$ and $\YY'$ denote the images 
of $\H^1\cap\J^{1\g}$ and $\J^1\cap\H^{1\g}$ in $\VV$, 
respectively. 

\begin{lemm}
\label{gerardmer}
The subspaces $\YY$ and $\YY'$ are both contained 
in $\XX$, and we have 
$\XX=\YY\oplus\YY'$.
\end{lemm}

\begin{proof}
One easily verifies that $\tau$ stabilizes $\XX$ and 
exchanges $\YY$ and $\YY'$. 
First note that $\YY\subseteq\XX$, 
since $\langle x,y \rangle=1$ for any $x\in\H^1$ and $y\in\J^1$.
By applying $\tau$, and thanks to \eqref{Chirac}, we deduce that 
$\YY'$ is also contained in $\XX$. 
Now, thanks to \eqref{SYMPLECGOTHIC}, we have:
\begin{equation*}
\langle 1+x,1+y\rangle=\psi\circ\tr(\b(xy-yx))
\end{equation*}
for all $x,y\in\mathfrak{j}^1\cap\mathfrak{j}^{1\g}$,
where $\JJ^1$ is the sub-$\Oo$-lattice of $\aa$ 
such that $\J^1=1+\JJ^1$.
Let $a_\b$ denote the endomorphism of $\F$-algebras
$x\mapsto \b x-x\b$ of $\Mat_n(\F)$.
Given a subset $\SS\subseteq\Mat_n(\F)$, write $\SS^*$ for the set
of $a\in\Mat_n(\F)$ such that $\psi(\tr(as))=1$ for all $s\in\SS$. 
Then the set of $x\in\mathfrak{j}^1\cap\mathfrak{j}^{1\g}$ 
such that $\langle 1+x,1+y\rangle=1$ for all 
$y\in\mathfrak{j}^1\cap\mathfrak{j}^{1\g}$ is equal to:
\begin{align*}
\mathfrak{j}^1\cap\mathfrak{j}^{1\g}\cap 
a_\b(\mathfrak{j}^1\cap\mathfrak{j}^{1\g})^* 
&= \mathfrak{j}^1\cap\mathfrak{j}^{1\g}\cap 
(a_\b(\mathfrak{j}^1)\cap a_\b(\mathfrak{j}^1)^\g)^* \\ 
&= \mathfrak{j}^1\cap\mathfrak{j}^{1\g}\cap 
(a_\b(\mathfrak{j}^1)^*+a_\b(\mathfrak{j}^1)^{*\g}) \\
&= \mathfrak{j}^{1\g}\cap(\mathfrak{j}^1\cap a_\b(\mathfrak{j}^1)^*)
+ \mathfrak{j}^1\cap(\mathfrak{j}^{1}\cap a_\b(\mathfrak{j}^1)^*)^\g.
\end{align*}

We now claim that:
\begin{equation}
\label{sandrinecollette}
\mathfrak{j}^{1}\cap a_\b(\mathfrak{j}^1)^*=\mathfrak{h}^1.
\end{equation}
To see this, look at the case where $g=1$.
On the one hand, for $x\in\JJ^1$, we have
$\langle 1+x,1+y\rangle=1$ for all $y\in\JJ^1$ if and only if 
$x\in\mathfrak{j}^1\cap a_\b(\mathfrak{j}^1)^*$.
On the other hand, the symplectic form \eqref{SYMPLEC} on the space 
$\J^1/\H^1$ is non-degenerate.
We thus have 
$\mathfrak{j}^{1}\cap a_\b(\mathfrak{j}^1)^*\subseteq\mathfrak{h}^1$
and the other containment follows from the fact that $\psi$ is trivial on 
the maximal ideal of $\Oo$.

We now go back to our general situation with $g\in\B^\times$.
Applying \eqref{sandrinecollette} to $\JJ^1$ and $\JJ^{1\g}$,
we get:
\begin{equation*}
\mathfrak{j}^1\cap\mathfrak{j}^{1\g}\cap 
a_\b(\mathfrak{j}^1\cap\mathfrak{j}^{1\g})^* 
=\JJ^{1\g}\cap\HH^{1}+\HH^{1\g}\cap\JJ^{1}.
\end{equation*}
The result follows. 
\end{proof}

\begin{coro}
\label{Pleven}
The subspaces $\XX^+=\XX\cap\VV^+$ and 
$\XX^-=\XX\cap\VV^-$ have the same dimension
and we have $\XX=\XX^+\oplus\XX^-$.
\end{coro}

\begin{proof}
The map $x\mapsto x+\tau(x)$ is an isomorphism from $\YY$ 
to $\XX^+$,
and the map $x\mapsto x-\tau(x)$ is an isomorphism from $\YY$ 
to $\XX^-$. 
Thanks to Lemma \ref{gerardmer} and the fact that $\YY$, $\YY'$ 
have the same dimension, we thus have:
\begin{equation*}
\dim \XX^++\dim \XX^- = 2\cdot\dim\YY = \dim \XX,
\end{equation*}
which ends the proof of the corollary.
\end{proof}

Now consider the non-degenerate symplectic space $\VV/\XX$. 
It decomposes into the direct sum~of two totally isotropic 
subspaces $(\VV^++\XX)/\XX$ and $(\VV^-+\XX)/\XX$. 
We thus have:
\begin{align*}
\max(\dim((\VV^++\XX)/\XX), \dim((\VV^-+\XX)/\XX))
&\< \frac{1}{2}\cdot\dim(\VV/\XX), \\
\dim((\VV^++\XX)/\XX)+\dim((\VV^-+\XX)/\XX)
&=\dim(\VV/\XX). 
\end{align*}
These spaces thus have the same dimension and are 
maximal totally isotropic.
Corollary~\ref{Pleven}~now implies that $\VV^+$ and $\VV^-$ have the same 
dimension.

In order to deduce the equality \eqref{CharlesDeGaulle},
and thanks to Lemma \ref{willycolette},
it remains to prove\footnote{I thank Jiandi Zou for pointing out this argument
to me, which was missing in a former version of the proof.}
that:
\begin{equation*}
(\J^1\cap\J^{1\g}:\H^1\cap\H^{1\g}) = (\J^1:\H^1),
\end{equation*}
which follows from \cite{BK} Lemma 5.1.10.
This ends the proof of Proposition \ref{MultOneEta}. 
\end{proof}

\subsection{Contribution of the mixed Heisenberg representation}
\label{Felicite}

Let $g\in\J\B^\times\G^\s$.
We have seen in Paragraph \ref{DCL} (see \eqref{VGE1})
that, changing $g$ without changing the double coset $\J g\G^{\s}$, 
we may assume that $g\in\B^\times$ and that $\g=\s(g)g^{-1}$ can be 
written $\g=bu$ with $b$ of the form \eqref{bantidiagonal} 
and $u\in\U^1=\J^1\cap\B^\times$.
We write $\tau$ for the involution defined 
by Lemma \ref{LouElias1} and $\U=\J\cap\B^\times$.

We have a standard Levi subgroup $\M$ of $\G$ defined by \eqref{Detroit}
and parabolic subgroups $\P$,~$\P^-$ of $\G$ with Levi component $\M$, 
opposite to each other and with unipotent radicals $\N$, $\N^-$ 
res\-pec\-ti\-ve\-ly. 
There is a unique standard hereditary order $\bbm\subseteq\bb$ 
such that:
\begin{equation*}
\bbm^\times = (\U^1\cap\N^-)\cdot(\U\cap\P).
\end{equation*}
Since $u\in\U^1\subseteq\U^1(\bbm)$ and thanks to the specific form of $b$, 
one verifies that:
\begin{equation}
\label{PhillipedeMaubert}
\U^1(\bbm) = (\U^1\cap\P^-)\cdot(\U\cap\N)= (\U\cap\U^{1\g})\U^1.
\end{equation}
Let $\aam\subseteq\aa$ be the unique hereditary order 
of $\Mat_n(\F)$
normalized by $\E^\times$ such that $\aam\cap\B=\bbm$.
This gives us a simple stratum $[\aam,\b]$.
Let $\t_{{\rm m}}\in\Cc(\aam,\b)$ be the transfer of $\t$
(see \eqref{transfermap})
and~$\n_{{\rm m}}$ be the Hei\-sen\-berg repre\-sen\-tation
on $\J^1_{{\rm m}}=\J^1(\aam,\b)$ associated with $\t_{{\rm m}}$~(by 
\cite{BK} Propo\-si\-tion 5.1.1 and \cite{MSt} Proposition 2.1).

Let $\SS^1$ be the pro-$p$-subgroup $\U^1(\bbm)\J^1\subseteq\J$. 
By \cite{BK} Proposition 5.1.15, there is an ir\-re\-du\-cible representation 
$\wn$ of the group $\SS^1$, 
unique up to isomorphism, 
extending $\n$ and such that:
\begin{equation}
\label{etacoherence}
\Ind^{\U^1(\aam)}_{\SS^1} (\wn) 
\simeq
\Ind^{\U^1(\aam)}_{\J^1_{{\rm m}}} (\n_{{\rm m}}).
\end{equation}
In this paragraph, we prove the following result.

\begin{prop}
\label{dimwn}
We have $\dim \Hom_{\SS^{1g}\cap\G^\s}(\wn^g,1) = 1$.
\end{prop}

\begin{proof}
Since $\wn$ extends $\n$, 
the space $\Hom_{\SS^{1g}\cap\G^\s}(\wn^g,1)$ 
is contained in the $1$-dimensional space 
$\Hom_{\J^{1g}\cap\G^\s}(\n^g,1)$.
It is thus enough to prove that  
$\Hom_{\SS^{1g}\cap\G^\s}(\wn^g,1)$ 
is non-zero. 
Equivalently, by Lem\-ma \ref{LouElias1}, it is enough to prove that 
$\Hom_{\SS^{1}\cap\G^\tau}(\wn,1)$ is non-zero.

First note that, since $\bbm$ is $\s$-stable, 
$\aam$ is $\s$-stable as well.
We have:
\begin{equation*}
\s(\H^1(\aam,\b)) 
= \H^1(\s(\aam),\s(\b))
= \H^1(\aam,-\b)
= \H^1(\aam,\b)
\end{equation*}
thus $\H_{{\rm m}}^{1}=\H^1(\aam,\b)$ is $\s$-stable. 
By an argument similar to the one used in \cite{AKMSS} {Paragraph 4.6}, 
it then follows that $\t_{{\rm m}}\circ\s=(\t_{{\rm m}})^{-1}$.

Since $\g$ intertwines $\t_{{\rm m}}$ by 
Proposition~\ref{patel}\eqref{bourrin6}, 
it follows from Proposition \ref{bathmologie} 
that the character $\t_{{\rm m}}^g$ is trivial on 
$\H_{{\rm m}}^{1g}\cap\G^\s$, thus  
$\Hom_{\J^{1g}_{{\rm m}}\cap\G^\s}(\n^g_{{\rm m}},1)=
\Hom_{\J^{1}_{{\rm m}}\cap\G^\tau}(\n_{{\rm m}},1)$ 
is non-zero. 
Inducing to $\U^1(\aam)$, we get:
\begin{equation*}
\Hom_{\U^1(\aam)\cap\G^\tau}\left(\Ind^{\U^1(\aam)}_{\J^{1}_{{\rm m}}} 
(\n_{{\rm m}}),1\right)\neq\{0\}.
\end{equation*}
Applying the Frobenius reciprocity and the Mackey formula, 
it follows that there is a $x\in\U^1(\aam)$ such that:
\begin{equation}
\label{inter1}
\Hom_{\SS^{1x}\cap\G^\tau}(\wn^{x},1) \neq \{0\}.
\end{equation}
We claim that $x\in\SS^{1}(\U^1(\aam)\cap\G^{\tau})$. 
Restricting \eqref{inter1} to the subgroup $\H^{1x}\cap\G^\tau$ 
and applying Proposition \ref{bathmologie}, we get:
\begin{equation*}
\s(xg)g^{-1}x^{-1}=\s(x)\g x^{-1}\in\J^1\B^\times\J^1 \cap \U^1(\aam)\g\U^1(\aam).
\end{equation*}
Write $\s(x)\g x^{-1}=jcj'$ for some $j,j'\in\J^1$ and $c\in\B^\times$.
Since $\g\in\B^\times$ and $\J^1\subseteq\U^1(\aam)$, 
the simple intersection property \eqref{SIP} implies that 
$c\in\U^1(\aam)\g\U^1(\aam)\cap\B^\times=\U^1(\bbm)\g\U^1(\bbm)$.
Therefore we have $\s(x)\g x^{-1}=\s(s)\g s'$ for some $s,s'\in\SS^1$. 
If we let $y=s^{-1}x$,
then we have $\s(y)\g y^{-1}=\g l$ for some $l\in\SS^1$, 
that is $\tau(y)y^{-1}=l$. 
Since the first cohomology set of $\tau$ in 
$\SS^1\cap\SS^{1\g}$ is trivial, 
we get $l=\tau(h)h^{-1}$ for some $h\in\SS^1$. 
This gives us:
\begin{equation*}
x\in\U^1(\aam)\cap\SS^1(\G^\s)^{g^{-1}}
\end{equation*}
and the claim follows from the fact that $\SS^1\subseteq\U^1(\aam)$.

Putting the claim and \eqref{inter1} together, we deduce that 
$\Hom_{\SS^{1}\cap\G^\tau}(\wn,1)$ is non-zero. 
\end{proof}

\subsection{The double coset theorem}
\label{DCT}

Let $\bk$ be an irreducible representation of $\BJ$ extending $\n$
as in Paragraph \ref{pifoulechien}.
There is an~irre\-du\-cible representa\-tion $\bt$ of $\BJ$,
unique up to isomorphism, which is trivial on the subgroup 
$\J^1$~and satisfies $\bl\simeq\bk\otimes\bt$. 
We have the following lemma. 

\begin{lemm}
\label{passypassion}
Let $g\in\J\B^\times\G^\s$.
\begin{enumerate}
\item
There is a unique charac\-ter $\chi$ of $\BJ^g\cap\G^\s$ trivial on
$\J^{1g}\cap\G^\s$ such that:
\begin{equation*}
\Hom_{\J^{1g}\cap\G^\s}(\n^g,1) = 
\Hom_{\BJ^{g}\cap\G^\s}(\bk^g,\chi^{-1}).
\end{equation*}
\item
The canonical linear map:
\begin{equation*}
\Hom_{\J^{1g}\cap\G^\s}(\n^g,1)
\otimes\Hom_{\BJ^g\cap\G^\s}(\bt^g,\chi) \to \Hom_{\BJ^g\cap\G^\s}(\bl^g,1) 
\end{equation*}
is an isomorphism.
\end{enumerate}
\end{lemm}

\begin{proof}
Let us fix a non-zero linear form $\Ee\in\Hom_{\J^{1g}\cap\G^\s}(\n^g,1)$. 
The choice of $\bk$ defines an~ac\-tion of $\BJ^g\cap\G^\s$ on the space 
$\Hom_{\J^{1g}\cap\G^\s}(\n^g,1)$, which has dimension $1$ by 
Pro\-po\-sition \ref{MultOneEta}.
This determines a unique charac\-ter $\chi$ of $\BJ^g\cap\G^\s$ trivial on
$\J^{1g}\cap\G^\s$ such that:
\begin{equation*}
\Ee\circ\bk^{g}(x)=\chi(x)^{-1}\cdot\Ee
\end{equation*}
for all $x\in\BJ^g\cap\G^\s$.
This gives us the first part of the lemma.

Given $\Ll\in\Hom_{\BJ^g\cap\G^\s}(\bl^g,1)$ and $w$ in the space of $\bt$, 
the linear form $v\mapsto\Ll(v\otimes w)$ defined on the space of $\n$ is in 
$\Hom_{\J^{1g}\cap\G^\s}(\n^g,1)$. By 
Proposition \ref{MultOneEta} it is thus of the form $\Ff(w)\Ee$ for a 
unique $\Ff(w)\in\FC$.
We have $\Ll=\Ee\otimes\Ff$ and
$\Ff\in\Hom_{\BJ^g\cap\G^\s}(\bt^g,\chi)$.
\end{proof}

\begin{theo}
\label{DCT1}
Let $g\in\G$ and suppose $\Hom_{\BJ^g\cap\G^\s}(\bl^g,1)$ is non-zero. 
Then $\s(g)g^{-1}\in\BJ$.
\end{theo}

\begin{proof}
We know from Proposition \ref{bathmologie} and Lemma \ref{FL} 
that $g\in\J\B^\times\G^\s$.
We thus may assume that $g\in\B^\times$ and $\g=\s(g)g^{-1}$ 
is as in Paragraph \ref{Felicite}. 
In particular, we have a standard hereditary order $\bbm\subseteq\bb$
and an involution $\tau$.

Let us fix an irreducible representation $\bk$ of $\BJ$ extending $\n$, 
and let $\chi$ be the character given by Lemma \ref{passypassion}.
The restriction of $\bk$ to $\J$, denoted $\k$, 
is an irreducible representation of $\J$ extending $\n$.
It follows from Remark \ref{handluggage} 
that $\k$ is a beta-extension of $\n$, 
and from \cite{BK} Theorem 5.2.3 that $\k$ extends $\wn$. 
Proposition \ref{dimwn} thus implies:
\begin{equation*}
\Hom_{\SS^{1g}\cap\G^\s}(\wn^g,1) = 
\Hom_{\BJ^{g}\cap\G^\s}(\bk^g,\chi^{-1})
\end{equation*}
and $\chi$ is trivial on $\U^1(\bbm)^{g}\cap\G^{\s}$.
By Lemma \ref{passypassion}, the space 
$\Hom_{\BJ^{g}\cap\G^\s}(\bt^g,\chi)$ is non-zero. 

Write $\U=\J\cap\B^\times$ and $\U^1=\J^1\cap\B^\times$.
Since $g\in\B^\times$, we have 
$\J^{g}\cap\G^\s=(\U^{g}\cap\G^{\s})(\J^{1g}\cap\G^\s)$.
Let $\rho$ be the restriction of $\bt$ to $\J$.
Then $\Hom_{\U^{g}\cap\G^{\s}}(\rho^g,\chi)$ is non-zero. 
Lemma \ref{LouElias1} implies:
\begin{equation}
\label{portemaillot}
\Hom_{\U^1(\bbm)\cap\G^{\tau}}(\rho,1)
= \Hom_{\U^1(\bbm)^{g}\cap\G^{\s}}(\rho^g,1)
\neq\{0\}.
\end{equation}
We now describe more carefuly the subgroup $\U^1(\bbm)$.

\begin{lemm}
\label{choupinew}
We have $\U^1(\bbm)=(\U^1(\bbm)\cap\G^{\tau})\U^{1}$.
\end{lemm}

\begin{proof}
We follow the proof of \cite{HMIMRP} Proposition 5.20. 
According to \eqref{PhillipedeMaubert} it is enough to prove 
that $\U\cap\U^{1\g}$ is contained in 
$(\U^1(\bbm)\cap\G^{\tau})\U^{1}$. 
Let $x\in\U\cap\U^{1\g}$ and define $y=x^{-1}\tau(x)^{-1}x\tau(x)$. 
Then $y\in\U^1\cap\U^{1\g}$ and $y\tau(y)=1$.
Since the first cohomology set of $\tau$ in $\U^1\cap\U^{1\g}$ is 
trivial, we get $y=z\tau(z)^{-1}$ for some $z\in\U^1\cap\U^{1\g}$.
Define $x'=x\tau(x)\tau(z)$.
Then $x'\in\U^1(\bbm)\cap\G^\tau$ and we have $x\in x'\U^1$.
\end{proof}

Since $\rho$ is trivial on $\U^{1}$, Lemma \ref{choupinew} and 
\eqref{portemaillot} together imply that 
$\Hom_{\U^1(\bbm)}(\rho,1)$ is non-zero. 
Since $\U^1(\bbm)/\U^{1}$ is a unipotent subgroup of 
$\U/\U^{1}\simeq\GL_m(\ee)$, 
the fact that the representation $\rho$ is cuspidal 
(see Paragraph \ref{pifoulechien}) 
implies that $\bbm=\bb$, that is $\g\in\U\subseteq\BJ$.
\end{proof}

{Lemma 4.25} of \cite{AKMSS} 
gives a detailed account of the elements $g\in\G$ 
such that $\s(g)g^{-1}\in\BJ$.

\subsection{Proof of Theorem \ref{DSTT}}
\label{viteuntitre}

Let $\pi$ be a $\s$-selfdual cuspidal  representation of $\G$, 
and $(\BJ,\bl)$ be a $\s$-selfdual type in $\pi$ given by 
Theorem \ref{PIMAIN}.
If the space 
$\Hom_{\BJ\cap\G^\s}(\bl,1)$ is non-zero, 
then \eqref{MACKEYFONDA}
implies that $\pi$ is distinguished. 

Conversely, suppose that $\pi$ is distinguished
and that $(\BJ,\bl)$ has been chosen as in 
Remark~\ref{StandardStableType} as it may be. 
Then the space $\Hom_{\BJ^g\cap\G^\s}(\bl^g,1)$ is non-zero for some $g\in\G$. 
By Theorem \ref{DCT1}, one has $\s(g)g^{-1}\in\BJ$.
Thus $\BJ^g$ is $\s$-stable, and:
\begin{equation*}
(\bl^g)^\s=(\bl^{\s})^{\s(g)}\simeq(\bl^\vee)^g=(\bl^g)^\vee
\end{equation*}
thus the type $(\BJ^g,\bl^g)$ is $\s$-selfdual.

\begin{rema}
\label{chienaphasique}
Let $(\BJ,\bl)$ and $(\BJ',\bl')$ be two distinguished $\s$-selfdual types in 
$\pi$. 
Since they both occur in $\pi$, there is a $g\in\G$ such that 
$\BJ'=\BJ^g$ and $\bl'\simeq\bl^g$.
Thanks to the multiplicity $1$ property of Theorem \ref{FP},
the formula \eqref{MACKEYFONDA} tells us that the double cosets $\BJ\G^\s$ 
and $\BJ g\G^\s$ are equal, which implies that $g\in\BJ\G^\s$.
Thus a distinguished cuspidal 
representation $\pi$ contains, up to $\G^\s$-conjugacy, 
a unique distinguished $\s$-selfdual type.
\end{rema}

Recall that Proposition \ref{TT0canonique} associates to 
any $\s$-selfdual cuspidal  repre\-sen\-tation of $\G$ 
a quadra\-tic extension $\T/\T_0$.

\begin{coro}
\label{zarathoustra}
Let $\pi$ be a $\s$-selfdual cuspidal  representation of $\G$, 
and suppose that $\T/\T_0$ is unramified. 
Then $\pi$ is distinguished if and only if any $\s$-selfdual type 
in $\pi$ is distinguished.
\end{coro}

\begin{proof}
This follows from Theorem \ref{DSTT} together with 
Proposi\-tion \ref{classesdetypesstables7}, 
which says that 
the re\-presentation $\pi$ contains, up to $\G^\s$-conjugacy, a uni\-que 
$\s$-selfdual type. 
\end{proof}

When $\T/\T_0$ is ramified, 
Proposi\-tion \ref{classesdetypesstables7} tells us that $\pi$ contains 
more than one $\G^\s$-conjugacy class of $\s$-selfdual types as soon 
as its relative degree $m$ is at least $2$. 
In the next section,
we~will see that the $\G^\s$-conjugacy class of 
index $\lfloor m/2\rfloor$ (see Definition \ref{defindex})
is the only one which may con\-tri\-bute to the 
distinc\-tion of $\pi$.

\section{The cuspidal ramified case}
\label{SECDISTTR}

As usual, write $\G=\GL_n(\F)$ for some $n\>1$.
To any $\s$-selfdual cuspidal represen\-ta\-tion of $\G$, 
one can associate 
a quadra\-tic extension $\T/\T_0$ and its relative degree $m$
(see Proposition \ref{TT0canonique}).
In this section, we will consider the case where $\T/\T_0$ is 
ramified.

\subsection{}

The first main result of this section is the following proposition,
which we will prove in Para\-graph \ref{mission2}.

\begin{prop}
\label{MAINTHM3r}
Let $\pi$ be a $\s$-selfdual {cuspidal} 
representation of $\G$ with quadratic extension $\T/\T_0$
and relative degree $m$.
Suppose $\T/\T_0$ is ramified.
Then $\pi$ is dis\-tinguished if and only if:
\begin{enumerate}
\item
either $m=1$ or $m$ is even, and
\item
any $\s$-selfdual type 
of index $\lfloor m/2\rfloor$ contained in $\pi$ is distinguished. 
\end{enumerate}
\end{prop}

\begin{rema}
Proposition \ref{MAINTHM3r} refines Theorem \ref{DSTT} by saying that, 
if $\T/\T_0$ is ra\-mi\-fied, then the $\G^\s$-conjugacy class of 
$\s$-selfdual types of index $\lfloor m/2\rfloor$ contained in $\pi$
is the only one which may con\-tri\-bute to the distinction of $\pi$. 
See \cite{AKMSS} {Proposition 5.5} for a characterization of this class 
in terms of Whittaker data. 
See also Definition \ref{carioca} and Remark \ref{cariocarema} below.
\end{rema}

\begin{rema}
\label{canarduck}
Proposition \ref{MAINTHM3r} is proved in \cite{AKMSS}
in a different manner from the one we give~here 
(see~Re\-mark \ref{minutes} above and
\cite{AKMSS} {Corollary 6.6}, {Remark 6.7}).
\end{rema}

\begin{rema}
\label{herisson}
If we assume $\pi$ to be \textit{supercuspidal} in Proposition \ref{MAINTHM3r}, 
then $m$ is automatically either even or equal to $1$,
even if $\pi$ is not distinguished (see Proposition \ref{flipflaplagirafe}).
\end{rema}

\begin{rema}
\label{poussinnoir}
However, if $\pi$ is non-supercuspidal in Proposition 
\ref{MAINTHM3r}, 
then its relative degree $m$ need not be either even nor equal to $1$.
Let $k$ be a divisor of $n$, and $\tau$ be a $\s$-selfdual super\-cus\-pidal 
repre\-sen\-ta\-tion of $\GL_{n/k}(\F)$. 
Assume $\FC$ has characteristic $\ell>0$,
let $\nu$ be the unramified character ``absolute value of the determinant'' 
and let $e(\tau)$ be the smallest integer $i\>1$ such that 
$\tau\nu^i\simeq\tau$.
Suppose that $k=e(\tau)\ell^u$ for some $u\>0$.
Then \cite{MSc} Théorème 6.14 tells us that 
the unique generic irreducible subquotient $\pi$ of the normalized 
parabolically induced representation: 
\begin{equation*}
\tau\times\tau\nu\times\dots\times\tau\nu^{k-1}
\end{equation*}
is cuspidal, and that it is $\s$-selfdual since $\tau$ is.
If $k>1$ and $m(\pi)=km(\tau)$ is odd, then $\pi$ is 
a $\s$-self\-dual cuspidal representation which is not 
distinguished nor $\ep$-distinguished. 
(For instance, this is the case when $\tau$ is the trivial character of
$\F^\times$
and $k=n=\ell$ where $\ell\neq2$ divides $q-1$,
which gives $m(\tau)=e(\tau)=1$).
\end{rema}

\subsection{Existence of $\s$-selfdual extensions of the Heisenberg 
  representation} 
\label{mission12}

We now go back to our usual notation.
Let $[\aa,\b]$ be a maximal simple stratum in $\Mat_n(\F)$~such 
that $\aa$ is $\s$-stable and $\s(\b)=-\b$.
Write $\E$ for the exten\-sion $\F[\b]$,
and suppose that it is rami\-fied over the field $\E_0$ 
of $\s$-fixed points in $\E$.
Let $d$ be the degree $[\E:\F]$ and write $n=md$.

Let $\ee$ denote the residue field of $\E$. 
Let us notice once and for all that, since $p\neq2$,
any~char\-acter of $\GL_m(\ee)$ is of the form $\a\circ\det$, 
for some character $\a$ of $\ee^\times$.

The following lemma generalizes \cite{Coniglio} Lemme 3.4.6 
(which is con\-cer\-ned with complex repre\-sen\-tations and $\chi$ trivial 
only). 

\begin{lemm}
\label{L71r}
Let $\chi$ be a character of $(\GL_{i}\times\GL_{m-i})(\ee)$ for some 
$i\in\{0,\cdots,\lfloor m/2\rfloor\}$.
Suppose there is a $\chi$-distin\-guish\-ed cuspidal 
representation of $\GL_m(\ee)$.
Then either $m=1$ or $m=2i$. 
\end{lemm}

\begin{proof}
If $m\>2$, the result follows from Proposition \ref{prs}.
Note that, if $m=1$, then $\chi$ is the unique $\chi$-distin\-guish\-ed 
irreducible representation of $\GL_1(\ee)$.
\end{proof}

Let $\t\in\Cc(\aa,\b)$ be a maximal simple character such that 
$\H^1(\aa,\b)$ is $\s$-stable and $\t\circ\s=\t^{-1}$,
and let $\BJ=\BJ(\aa,\b)$ be its normalizer in $\G$. 
Let $\n$ be the Heisenberg representation of $\J^1=\J^1(\aa,\b)$ 
containing $\t$, and write $\J=\J(\aa,\b)$.

\begin{lemm}
\label{GuyLodrant5}
There is a $\s$-selfdual representation $\bk$ of $\BJ$ extending $\n$. 
\end{lemm}

\begin{proof}
Conjugating by a suitable element in $\G$,
we may assume that the stratum $[\aa,\b]$ 
satisfies the conditions of Remark \ref{StandardStableType}. 
Indeed, if it doesn't, there is a $g\in\G$ 
such that $\t^g$ is $\s$-selfdual and $[\aa^g,\b^g]$ satisfies these 
conditions. 
This implies that $\g=\s(g)g^{-1}$ normalizes $\t$, that is 
$\g\in\BJ$. Now, as\-sum\-ing the lemma to be true 
for $[\aa^g,\b^g]$,
there exists a $\s$-selfdual representation $\bk'$ of $\BJ^g$ ex\-tend\-ing 
$\n^g$. 
Define a representation $\bk$ of $\BJ$ by $\bk^g=\bk'$.
Then $\bk$ extends $\n$, and it is $\s$-selfdual since $\g\in\BJ$. 
From now on, we will assume that $[\aa,\b]$ 
satisfies the conditions of Remark \ref{StandardStableType}.~We 
will iden\-ti\-fy $\J/\J^1$ with $\GL_m(\ee)$, on 
which $\s$ acts trivially.

Suppose first that $\FC$ has characteristic $0$.
Let $\bk$ be a representation of~$\BJ$ extending 
$\n$,
let $\xx$~be the character of $\BJ$ trivial on $\J^1$ 
such that $\bk^{\s\vee}\simeq\bk\xx$ given by Lemma \ref{chouxfarci}
and $\chi$ be the~charac\-ter of $\BJ\cap\G^\s$ associated with $\bk$
by Lemma \ref{passypassion}.
We claim that there is a character $\nu$ of $\BJ$ trivial on 
$\J^1$ such that $(\nu\circ\s)\nu=\xx$. 
Indeed, $\bk\nu$ will then extend $\n$ and be $\s$-selfdual.
We have:
\begin{eqnarray*}
\Hom_{\J^{1}\cap\G^\s}(\n,1)
&=&\Hom_{\BJ\cap\G^\s}(\bk,\chi^{-1}) \\
&\simeq&\Hom_{\BJ\cap\G^\s}(\chi,\bk^{\s\vee}) \\
&\simeq&\Hom_{\BJ\cap\G^\s}(\chi\xx^{-1},\bk)
\end{eqnarray*}
where the isomorphism in the middle follows from 
the fact that $\s$ acts trivially on $\BJ\cap\G^\s$
and by duality.
Since $\FC$ has characteristic $0$ and $\BJ\cap\G^\s$ is compact, 
the latter space is isomorphic~to $\Hom_{\BJ\cap\G^\s}(\bk,\chi\xx^{-1})$. 
By uniqueness of $\chi$, it follows that 
the restriction of $\xx$ to $\BJ\cap\G^\s$ is $\chi^{2}$.~Res\-tric\-ting 
to $\J\cap\G^\s$ and writing $\xx=\varphi\circ\det$ and 
$\chi=\a\circ\det$ as characters of $\GL_m(\ee)$
for suitable characters $\varphi$, $\a$ of $\ee^{\times}$, 
we get $\varphi=\a^2$.
Let $\nu$ be the unique character of $\J$ which is trivial on $\J^1$
and equal to $\a\circ\det$ as a character of $\GL_m(\ee)$.
Since $\BJ$ is generated by $t$ and $\J$, 
it remains to ex\-tend $\nu$ to $\BJ$ by fixing a scalar
$\nu(t)\in\FC^\times$ such that $\nu(t)^2=\nu(-1)\xx(t)$.

Suppose now that $\FC$ is equal to $\flb$. 
As in the proof of Lemma \ref{L72}, we use a lifting and reduction argument. 
Note that reducing finite-dimensional smooth $\qlb$-representations of 
profinite groups is the same as for finite groups (for which we referred to
\cite{Serre} \S 15).
The simple character $\t$~lifts to~a simple character
$\widetilde{\t}$ with~va\-lues in $\zlb$, 
defined with respect to the same simple stratum as~$\t$,
and such that $\widetilde{\t}\circ\s=\widetilde{\t}^{-1}$.
By the characteristic $0$ case, 
there is a $\s$-selfdual $\qlb$-representa\-tion $\widetilde{\bk}$
of $\BJ$ extending the irredu\-ci\-ble $\qlb$-representation $\widetilde{\n}$
of $\J^1$ associa\-ted with $\widetilde{\t}$.
The reduction mod $\ell$ of $\widetilde{\n}$
is a representation of $\J^1$ containing $\t$, of the same dimen\-sion
as $\n$: it is thus isomorphic to $\n$ itself.
Let $\widetilde{\k}$ denote the restriction of $\widetilde{\bk}$ to $\J$. 
Its reduction mod $\ell$, denoted $\k$, 
is a $\s$-selfdual representa\-tion of $\J$ extending~$\n$,
and which extends to some representation $\bk$ of $\BJ$.
Since~$\k$~is $\s$-selfdual, the representation $\bk^{\s\vee}$ is
isomorphic to $\bk\mu$ for some character $\mu$ of $\BJ$ trivial on $\J$.
Since $\BJ$ is generated by $\J$ and $t$,
there is a character $\nu$ of $\BJ$ trivial on $\J$ such that
$(\nu\circ\s)\nu=\mu$,
thus $\bk\nu$ is $\s$-selfdual.

Finally, suppose that $\FC$ has characteristic $\ell>0$,
and fix an embedding $\iota:\flb\to\FC$.
Since~$\t$ has finite image,
there is a simple $\flb$-character $\t_0$
defined with respect to the same simple stratum~as $\t$
such that $\t_0^{}\circ\s=\t_0^{-1}$ and $\t=\iota\circ\t_0$.
Let $\bk_0$ be a $\s$-selfdual $\flb$-representation of $\BJ$ 
extending the irreducible $\flb$-representation $\n_0$
of $\J^1$ associa\-ted with $\t_0$.
The irreducible representations~$\n$ and $\n_0\otimes\FC$ 
both contain $\t$.
By uniqueness of the Heisenberg representation,
they are~iso\-mor\-phic.
It follows that $\bk=\bk_0\otimes\FC$
is a $\s$-selfdual $\FC$-representa\-tion of $\J$ extending~$\n$.
\end{proof}

\subsection{Proof of Proposition \ref{MAINTHM3r}}
\label{mission2}

Let $(\BJ,\bl)$ be a $\s$-selfdual type,
with associated simple character the character $\t$ 
of \S\ref{mission12}.

\begin{lemm}
\label{topchef5}
If $(\BJ,\bl)$ is distinguished, then:
\begin{enumerate}
\item
either $m=1$, 
\item
or $m=2r$ for some $r\>1$,
and $(\BJ,\bl)$ has index $r$.
\end{enumerate}
\end{lemm}

\begin{proof}
Let $\bk$ be a $\s$-selfdual representation of $\BJ$ extending $\n$ 
provided by Lemma \ref{GuyLodrant5}.
Let~$\bt$ be the unique irreducible representation of $\BJ$ trivial on $\J^1$ 
such that $\bl\simeq\bk\otimes\bt$ and $i$ be the index of $(\BJ,\bl)$.
Lemma \ref{passypassion} tells us that $\bt$ is $\chi$-distinguished for some 
character $\chi$ of $\BJ\cap\G^\s$ trivial on $\J^1\cap\G^\s$. 
Restricting $\bt$ to $\J$ and identifying $\J/\J^1$ with $\GL_{m}(\ee)$,
we get a cuspidal  representation $\rho$ 
of $\GL_{m}(\ee)$ and a character $\chi$ of $(\GL_i\times\GL_{m-i})(\ee)$ 
such that $\rho$ is $\chi$-distinguished. 
The result follows from Lemma \ref{L71r}.
\end{proof}

Let $\pi$ be a $\s$-selfdual cuspidal representation of 
$\G$, and suppose that the quadratic extension $\T/\T_0$ associated with it 
by Proposi\-tion \ref{TT0canonique} is ra\-mi\-fied.
Let $(\BJ,\bl)$
be a $\s$-selfdual type contained in $\pi$.
By Remark \ref{StandardStableType2}, we may assume that it is
defined with respect to a $\s$-selfdual simple~stra\-tum. 
By Re\-mark \ref{Turron}, $\E$ is ramified over $\E_0$.
We can thus apply the results of Paragraph \ref{mission12}~and Lemma 
\ref{topchef5}. 
Proposition \ref{MAINTHM3r} now follows from Theo\-rem \ref{DSTT} 
together with Lemma \ref{topchef5}.

\subsection{Existence of distinguished extensions of the Heisenberg 
  representation}
\label{TR4}

The second main result of this section is the following proposition. 

\begin{prop}
\label{michelesolara}
Let $\pi$ be a $\s$-selfdual cuspidal representation of 
$\G$ with ramified quadratic extension $\T/\T_0$.
Assume that $m=1$ or $m$ is even, 
and let $(\BJ,\bl)$ be a $\s$-selfdual type in $\pi$ of index $\lfloor m/2\rfloor$. 
Let $\J^1$ be the maxi\-mal nor\-mal pro-$p$-subgroup of $\BJ$
and $\n$ be an irredu\-ci\-ble com\-po\-nent of the restriction of $\bl$ to $\J^1$.
\begin{enumerate}
\item
\label{michelesolara.p1}
There is a distinguished representation of $\BJ$ extending $\n$,
and any such representation of $\BJ$ is $\s$-self\-dual.
\item
\label{michelesolara.p2}
Let $\bk$ be a distinguished representation of $\BJ$ ex\-ten\-ding $\n$,
and let $\bt$ be the unique represen\-ta\-tion of $\BJ$ trivial on $\J^1$ 
such that $\bl\simeq\bk\otimes\bt$.
Then $\pi$ is distinguished if and only if $\bt$ is~dis\-tin\-guished.
\end{enumerate}
\end{prop}

We start with the following lemma, which slightly refines part (1) of the 
proposition.

\begin{lemm}
\label{HoMcKayTR}
\label{uniquekappa}
Let $(\BJ,\bl)$ be as in {\rm Proposition \ref{michelesolara}}. 
\begin{enumerate}
\item
There is a distinguished representation $\bk$ of $\BJ$ extending $\n$. 
\item
If $\ell=2$ or if $m$ is even, such a distinguished representation $\bk$ is unique.
\item
Any distinguished representation $\bk$ of $\BJ$ extending $\n$ is 
$\s$-selfdual. 
\end{enumerate}
\end{lemm}

\begin{rema}
If $\ell\neq2$ and $m=1$, 
there are exactly two distinguished representations of $\BJ$ extending $\n$,
twisted of each other by the unique non-trivial character of $\BJ$ trivial on 
$(\BJ\cap\G^\s)\J^1$.
(See the proof below, which shows that $(\BJ\cap\G^\s)\J^1$ has index $2$ in 
$\BJ$.)
\end{rema}

\begin{proof}
Let $\J$ be the maximal compact subgroup of $\BJ$,
and $\J^1$ be its maximal normal pro-$p$-sub\-group. 
As usual, we fix a maximal simple stratum $[\aa,\b]$ defining $(\BJ,\bl)$ 
such that $\aa$ is $\s$-stable and $\s(\b)=-\b$,
and write $\E=\F[\b]$ and $\ee$ for its residue field. 
We will identify $\J/\J^1$ with $\GL_{m}(\ee)$ equipped with an
involution whose fixed points is $(\GL_i\times\GL_{m-i})(\ee)$ where 
$i=\lfloor m/2\rfloor$.

Let $\bk$ be an irreducible representation of $\BJ$ extending $\n$.
By Lemma \ref{passypassion}, there is a cha\-racter $\chi$ of 
$\BJ\cap\G^\s$ trivial on $\J^1\cap\G^\s$ associated to $\bk$.
We claim that $\chi$ extends to a character $\phi$ of $\BJ$ 
trivial on $\J^1$.
It will then follow that $\bk\phi$ is distinguished and extends $\n$.

Suppose first that $m=1$.
We then have canonical group isomorphisms:
\begin{equation}
\label{casuistique}
(\J\cap\G^\s)/(\J^1\cap\G^\s)\simeq\J/\J^1\simeq\ee^\times.
\end{equation}
Thus there is a unique character $\phi$ of $\J$ 
trivial on $\J^1$ which coincides with $\chi$ on $\J\cap\G^\s$. 
Since $\BJ$ is generated by $t$ and $\J$, 
and since $t$ normalizes $\phi$, 
this character extends to a character of $\BJ$ trivial on $\J^1$.

\begin{lemm}
\label{CASm1}
Suppose that $m=1$.
Then $\BJ\cap\G^\s$ is generated by $\J\cap\G^\s$ and $t^2$. 
\end{lemm}

\begin{proof}
Since we have $\BJ=\E^\times\J^1$ when $m=1$,
we may consider the exact sequence of $\s$-groups:
\begin{equation*}
1 \to \U^1_\E \to \E^\times\times\J^1 \to \BJ \to 1.
\end{equation*}
Taking $\s$-invariants
and since the first cohomology group $\textsf{H}^1(\s,\U^1_\E)$ is trivial, 
$\BJ\cap\G^\s$ is genera\-ted by $\E_0^\times$ and $\J^1\cap\G^\s$.
The result follows from \eqref{casuistique} and the fact that $t^2$ is a 
uniformizer of $\E_0$. 
\end{proof}

It follows from Lemma \ref{CASm1}
that $\chi$ can be extended to a character $\phi$ of $\J$ 
trivial on $\J^1$.
Since we must have $\phi(t)^2=\chi(t^2)$ in the field $\FC$ of characteristic 
$\ell$, 
there are at most two such characters,
with uniqueness if and only if $\ell=2$.

Suppose now that $m=2r$ for some $r\>1$, and consider the element:
\begin{equation*}
w=
\begin{pmatrix}
&{\rm id}_{r}\\
{\rm id}_{r}&
\end{pmatrix}
\in\bb^\times\subseteq\GL_{m}(\E)
\end{equation*}
where ${\rm id}_r$ is the identity matrix in $\GL_r(\E)$.
 
\begin{lemm}
\label{tprime}
Suppose that $m=2r$.
The group $\BJ\cap\G^\s$ is generated by $\J\cap\G^\s$ and $tw$. 
\end{lemm}

\begin{proof}
First, notice that $t'=tw$ is $\s$-invariant. 
Any $x\in\BJ$ can be written $x=t'^ky$ for unique $k\in\ZZ$ 
and $y\in\J$. 
We thus have $x\in\BJ\cap\G^\s$ if and only if $y\in\J\cap\G^\s$.
\end{proof}

Since $\bk$ and $\BJ\cap\G^\s$ are normalized by $w$, 
we have
$\Hom_{\BJ\cap\G^\s}(\bk,\chi^{-1})=\Hom_{\BJ\cap\G^\s}(\bk,(\chi^w)^{-1})$,
and the uniqueness of $\chi$ implies that $\chi^w=\chi$.
First, consider the character of:
\begin{equation*}
(\J\cap\G^\s)/(\J^1\cap\G^\s)\simeq(\GL_r\times\GL_{r})(\ee)
\end{equation*}
defined by $\chi$ and write it 
$(\a_1\circ\det)\otimes(\a_2\circ\det)$
for some characters $\a_1$, $\a_2$ of $\ee^\times$.
The identity $\chi^w=\chi$ implies that $\a_1=\a_2$,
thus there is a unique character $\phi$ of $\J$ 
trivial on $\J^1$ which coincides with $\chi$ on $\J\cap\G^\s$. 
By Lemma \ref{tprime},
there is a unique character $\phi$ of $\BJ$ trivial on $\J^1$ extending $\chi$. 
This proves (1) and (2).

Now let $\bk$ be a dis\-tin\-guish\-ed representation of $\BJ$ 
extending $\n$.
It satisfies $\bk^{\s\vee}\simeq\bk\xx$ for some character $\xx$ of $\BJ$ 
trivial on $\J^1$ such that $\xx\circ\s=\xx$
(see Lemma \ref{chouxfarci}).
Since $\bk$ is distinguished,
$\xx$ is trivial on $\BJ\cap\G^\s$.
We will prove that $\bk$ is $\s$-selfdual, that is, that the character $\xx$ 
is trivial. 

Suppose first that $m=1$.
Thus $(\BJ,\bk)$ is a distinguished type in $\G$. 
Let $\pi$ denote the cuspidal ir\-reducible re\-pre\-sen\-ta\-tion of $\G$ 
compactly induced from $\bk$.
It is distinguished, thus $\s$-selfdual~by Theorem \ref{FP}. 
It follows that $\bk$ and $\bk^{\s\vee}\simeq\bk\xx$ are 
both contained in $\pi$, thus $\xx$ is trivial.

Suppose now that $m=2r$.
Since $\xx$ is trivial on $(\GL_{r}\times\GL_{r})(\ee)$, 
it must be trivial on $\GL_m(\ee)$.
Since $tw$ is $\s$-invariant, we have $\xx(tw)=1$.
Thus $\xx$ is trivial.
This proves (3).
\end{proof}

For part (2) of Proposition \ref{michelesolara}, 
it suffices to fix a distinguished representa\-tion $\bk$ of 
$\BJ$ extend\-ing $\n$ and to consider the canonical isomor\-phism:
\begin{equation*}
\Hom_{\BJ\cap\G^\s}(\bk,1)
\otimes\Hom_{\BJ\cap\G^\s}(\bt,1) \to \Hom_{\BJ\cap\G^\s}(\bl,1) 
\end{equation*}
(compare with Lemma \ref{passypassion}).

Proposition \ref{michelesolara} reduces the problem of the distinction of $\pi$ 
to that of $\bt$.
In the next section, we investigate the distinction of $\bt$ 
\textit{in the case where $\pi$ is supercuspi\-dal}.

\section{The supercuspidal ramified case}
\label{SECDISTTR0}

In this section, we inves\-ti\-gate 
the distinction of $\s$-selfdual \textit{supercuspidal} representations of
$\G$ in the case where $\T/\T_0$ ramified.

\subsection{The relative degree}

Let $\pi$ be a $\s$-selfdual cuspidal representation of 
$\G$ such that $\T/\T_0$ is ramified. 
Let $(\BJ,\bl)$~be~a $\s$-selfdual type con\-tained in $\pi$ and 
let $\bk$ be 
a $\s$-selfdual represen\-ta\-tion of $\BJ$ extending $\n$
given by Lemma \ref{GuyLodrant5}.
This defines a $\s$-selfdual irreduci\-ble representa\-tion $\bt$ 
of $\BJ$ trivial on $\J^1$. 
Let $\J$ denote the maximal com\-pact subgroup of $\BJ$ and $\rho$ denote the 
cuspidal repre\-sentation of $\J/\J^1\simeq\GL_m(\ee)$ induced by $\bt$.

Since $\bt$ is $\s$-selfdual, the representation $\rho$ is selfdual. 
Apply\-ing Fact \ref{cordoba} together with Lemma \ref{L71tr},
we get the following lemma mentioned in 
Remark \ref{herisson}. 

\begin{prop}
\label{flipflaplagirafe}
Let $\pi$ be a $\s$-selfdual supercuspidal representation of 
$\G$ such that $\T/\T_0$ is ramified.
Then its relative degree $m$ is either even or equal to $1$.
\end{prop}

\subsection{Distinction criterion in the ramified case}

Let $(\BJ,\bl)$ be a $\s$-selfdual type of index 
$\lfloor m/2\rfloor$ con\-tained in $\pi$.
We fix a distinguished represen\-tation $\bk$ of $\BJ$ extending $\n$
given by Proposition \ref{michelesolara}.
It is $\s$-selfdual, thus the representation $\bt$ of $\BJ$ trivial on $\J^1$ 
which correspond to this choice is $\s$-selfdual.
By Proposition \ref{michelesolara} again, 
$\pi$ is distinguished if and only if $\bt$ is distinguished. 
We now investigate the distinction of $\bt$.
For this, we will use the ad\-mis\-sible pairs of level zero introduced in 
Paragraphs \ref{pullfroid} and \ref{P45}.

Let us fix a $\s$-selfdual maximal simple stratum $[\aa,\b]$ such that 
$\BJ=\BJ(\aa,\b)$. 
Write $\E=\F[\b]$.
Let $(\K/\E,\xi)$ be an admissible pair of level zero attached to $\bt$
in the sense of Definition \ref{defiL0AP1}.
Since $\bt$ is $\s$-selfdual, Proposition \ref{charaghdinS4} tells us that 
there is a unique involutive $\E_0$-automorphism of~$\K$,
which we denote by $\s$, 
which is non-trivial on $\E$ and satisfies $\xi\circ\s=\xi^{-1}$.
Let $\K_0$ be the $\s$-fixed points~of $\K$ and $\E_0=\K_0\cap\E$. 

\begin{lemm}
\label{jaunepoussinr}
The representation $\bt$ is distinguished if and only if at least one of the 
following conditions is fulfilled: 
\begin{enumerate}
\item
$\ell=2$,
\item 
$m=1$ and $\bt$ is trivial on $\E_0^\times$,
\item
$m$ is even and $\xi$ is non-trivial on $\K_0^\times$. 
\end{enumerate}
\end{lemm}

\begin{rema}
Note that Case 3 cannot happen when $\ell=2$.
\end{rema}

\begin{proof}
The case $m=1$ is clear. 
Let us suppose that $m=2r$ for some $r\>1$.
The case~where the cha\-rac\-te\-ristic of $\FC$ is $0$ is given by 
\cite{HMIMRN} Proposition 6.3.
Suppose $\FC$ has characteristic $\ell>0$,
and fix an embedding $\iota:\flb\to\FC$.
Since $\xi\circ\s=\xi^{-1}$,
the image of $\xi$ is finite, thus contained~in the image~of $\flb$ in $\FC$. 
Indeed, the restriction of $\xi$ to the units of $\K^\times$ has finite image, 
and $\xi(t)$ has order at most $4$ 
since $\xi(\s(t))=\xi(t)^{-1}$ and $\s(t)\in\{-t,t\}$.
There is thus a $\flb$-character $\xi_0$ of $\K^\times$ such that 
$\xi_0^{}\circ\s=\xi_0^{-1}$ and $\xi=\iota\circ\xi_0$.
In particular, $(\K/\E,\xi_0)$ is an admissible pair of level zero.
Let $\bt_0$ be the $\s$-self\-dual $\flb$-repre\-sen\-ta\-tion 
attached to it. 
By Remark \ref{claudineremadm}, 
the representation $\bt$ is isomorphic to $\bt_0\otimes\FC$.
It thus suf\-fices to prove the lemma when $\FC$ is equal to $\flb$,
which we assume now.

We consider the canonical $\qlb$-lift $\widetilde{\xi}$ of $\xi$,
which has the same finite order as $\xi$.
It satisfies the identity $\widetilde{\xi}\circ\s=\widetilde{\xi}^{-1}$,
and the pair $(\K/\E,\widetilde{\xi})$~is admissible of level zero.
Attached to it, there is thus a $\s$-selfdual 
$\qlb$-representation~$\widetilde{\bt}$ 
of $\BJ$ trivial on $\J^1$.
Note that the kernel of $\widetilde{\bt}$ has finite index,
since it contains $\J^1$ and $t^4$,
thus $\widetilde{\bt}$ can be considered as a representation of 
a finite group.
From Proposition \ref{saintfargeau},
one checks easily that its reduction mod $\ell$ is $\bt$. 
Note that 
the restriction of $\xi$ to $\K_0^\times$ is either trivial~or 
(if $\ell\neq2$) equal to $\ep_{\K/\K_0}$.

Suppose $\xi$ is non-trivial on $\K_0^\times$. 
Then the same holds for $\widetilde{\xi}$, 
and the characteristic $0$ case~tells us that 
$\widetilde{\bt}$ is distinguished. 
As in the proof of Lem\-ma \ref{L72},
by applying Lemma \ref{aimee},
re\-du\-cing mod $\ell$ a non-zero in\-va\-riant form on 
$\widetilde{\bt}$ gives us a non-zero invariant form on $\bt$, 
which is thus~dis\-tinguished.

Suppose now that $\xi$ is trivial on $\K_0^\times$. 
Then the same holds for $\widetilde{\xi}$.
Let $\widetilde{\a}$ denote the unramified $\ell$-adic 
cha\-rac\-ter of $\K^\times$ 
of order $2$. 
Then $(\K/\E,\widetilde{\xi}\widetilde{\a})$ is an 
admissible pair of level zero.
It is at\-ta\-ched to $\widetilde{\bt}\widetilde{\h}$ 
where $\widetilde{\h}$ is the unramified $\ell$-adic character of $\BJ$ 
of order $2$.
Since $\widetilde{\xi}\widetilde{\a}$ is non-trivial on 
$\K_0^\times$, the represen\-tation $\widetilde{\bt}\widetilde{\h}$ is 
distinguished. 
Thus ${\bt}$ is $\h$-distinguished,
where $\h$ is the reduction mod $\ell$ of $\widetilde{\h}$.

If $\ell=2$, then $\bt$ is distinguished.
Suppose now that $\ell\neq2$.
If $\bt$ were both $\h$-distinguished and distinguished,
one would have two linearly inde\-pen\-dent linear forms in 
$\Hom_{\J\cap\G^\s}(\bt,1)$,
and this would contradict Lemma \ref{L72r}.
The result follows.
\end{proof}

The field extension $\E$ of $\F$ is not uniquely determined by $\pi$,
unlike its maximal tamely ramified extension $\T$.
To remedy this, let $\AP$ be the maximal tamely ramified 
sub-extension of $\K/\F$.
Write $\AP_0=\AP\cap\K_0$,
and let $\ap_0$ be the restriction of $\xi$ to $\D_0^\times$.

Since $\xi\circ\s=\xi^{-1}$
the character $\ap_0$ is quadratic, 
either trivial or (if $\ell\neq2)$ equal to $\ep_{\AP/\AP_0}$. We 
will see in Proposition \ref{rummo} that,
up to $\F_0$-equivalence, $\AP/\AP_0$
and $\ap_0$ are de\-ter\-mined by $\pi$.
 
\begin{theo}
\label{OttoDix}
Let $\pi$ be a $\s$-selfdual supercuspidal representation of $\G$. 
Suppose that $\T/\T_0$ is ramified.
Let $m$ be its relative degree and $\ap_0$ be the quadratic character 
of $\AP_0^\times$ associated to it.
\begin{enumerate}
\item
The representation $\pi$ is distinguished if and only if
at least one of the following conditions is fulfilled: 
\begin{enumerate}
\item
$\ell=2$,
\item 
$m=1$ and $\ap_0$ is trivial, 
\item
$m$ is even and $\ap_0$ is non-trivial.
\end{enumerate}
\item
Suppose that $\ell\neq2$.
Then $\pi$ is $\ep$-distinguished if and only if: 
\begin{enumerate}
\item
either $m=1$ and $\ap_0$ is non-trivial,
\item
or $m$ is even and $\ap_0$ is trivial. 
\end{enumerate}
\end{enumerate}
\end{theo}

\begin{rema}
\label{12TR}
If $\FC$ has characteristic $2$, then $\pi$ is always distinguished. 
If $\FC$ has characteristic not $2$, then $\pi$ is either distinguished 
or $\ep$-distinguished, but not both.
\end{rema}

\begin{proof}
By Proposition \ref{michelesolara} and Lemma \ref{jaunepoussinr},
it suffices to compare the restriction of $\xi$ to $\K_0^\times$ 
with $\ap_0$ when $\ell\neq2$.

Suppose first that $m=1$ and $\ap_0$ is trivial.
Since the restriction of $\bt$ to $\E_0^\times$ is equal to either $1$ or 
$\ep_{\E/\E_0}$, its restriction to $\T_0^\times$ is either $1$ or 
$\ep_{\T/\T_0}$, respectively. 
Since $\ap_0$ is trivial, we are in the first case, that is, 
the restriction of $\bt$ to $\E_0^\times$ is trivial.

Suppose now that $m\neq1$ and $\xi$ is non-trivial on $\K_0^\times$. 
We want to prove that $\ap_0$ is non-trivial. 
The restriction of $\xi$ to $\K_0^\times$ is equal to $\ep_{\K/\K_0}$. 
Thus $\ap_0$ is equal to $\ep_{\AP/\AP_0}$. 

Now suppose that $\FC$ has characteristic different from $2$,
and let $\chi$ be an unramified character of $\F^\times$ extending $\ep$. 
Note that the twisted representation $\pi'=\pi(\chi^{-1}\circ\det)$ is 
supercuspidal and $\s$-selfdual, 
and that the character associated with $\pi'$ is 
$\ap'_0=\ap^{}_0(\chi^{-1}\circ\N_{\K/\F})|_{\AP_0^\times}$,
where $\N_{\K/\F}$ is the norm map from $\K$ to $\F$.
Suppose first that $m=1$.
Then:
\begin{eqnarray*}
\text{$\pi$ is $\ep$-distinguished} 
&\Leftrightarrow& 
\text{$\pi'$ is distinguished} \\
&\Leftrightarrow& 
\text{the character $\ap'_0$ is trivial} \\
&\Leftrightarrow& 
\text{the character $\ap_0$ coincides with 
$\chi\circ\N_{\E/\F}$ on $\T_0^{\times}$}.
\end{eqnarray*}
Suppose now that $m\neq1$.
Then:
\begin{eqnarray*}
\text{$\pi$ is $\ep$-distinguished} 
&\Leftrightarrow& 
\text{$\pi'$ is distinguished} \\
&\Leftrightarrow& 
\text{the character $\ap'_0$ is non-trivial} \\
&\Leftrightarrow& 
\text{the character $\ap_0$ coincides with 
$(\chi\circ\N_{\K/\F})\ep_{\AP/\AP_0}$ on $\AP_0^{\times}$}.
\end{eqnarray*}
The restriction of $\chi\circ\N_{\K/\F}$ to $\AP_0^{\times}$ is 
$\ep\circ\N_{\AP_0/\F_0}=\ep_{\AP/\AP_0}$ to the power of $[\K_0:\AP_0]$, 
which is a $p$-power with $p$ odd. 
This gives us the expected result. 
\end{proof}

\begin{rema}
If $\pi$ is as in Theorem \ref{OttoDix} and $m>1$, 
its central character $\omega_\pi$ is always trivial on $\F_0^\times$. 
Indeed, since $\pi$ and $\bl$ have the same central character, 
we can express $\omega_\pi$ as the product $\omega_{\bk}\omega_{\bt}$,
where $\omega_{\bk}$ and $\omega_{\bt}$ are the central characters of $\bk$ 
and $\bt$ on $\F^\times$, respectively. 
Since $\bk$ is distinguished, its central character is trivial 
on $\F_0^\times$, 
thus $\omega_\pi$ and $\ap_0$ coincide on $\F_0^\times$. 
If $\ap_0$ is trivial then $\omega_\pi$ is trivial on $\F_0^\times$. 
Now assume that $\ap_0$ is equal to $\ep_{\AP/\AP_0}$.
Since $\AP/\AP_0$ is unramified by Lemma \ref{knr}, 
its restriction to $\F_0^\times$ is 
trivial if and only if $e(\AP_0/\F_0)$ is even, which is the case since 
$e(\AP_0/\T_0)=2$ when $m$ is even.
\end{rema}

\begin{rema}
\label{memoiresdunejeunefillerangee}
On the other hand, since $\T/\T_0$ is ramified,
the restriction of $\ep_{\T/\T_0}$ to $\F_0^\times$ is trivial if 
and only if $\kk_{0}^\times$ is contained in the subgroup 
$\ee^{\times 2}$ of squares in $\ee^\times$, that is, 
if and only if $f(\T_0/\F_0)$ is even.
It follows that, if $m=1$ and $f(\T_0/\F_0)$ is odd
(which is equivalent to $n$ being odd
by Lem\-ma \ref{MaraDesBois} and since $m$ is either $1$ or even)
then $\pi$ is distinguished if and only if $\omega_\pi$ is trivial on 
$\F_0^\times$.
\end{rema}

\section{The supercuspidal unramified case}
\label{SECDIST}

Let $\pi$ be a $\s$-selfdual cuspidal representation 
of $\G=\GL_n(\F)$ for $n\>1$. 
In this sec\-tion,~we in\-ves\-tigate the case where the quadratic extension 
$\T/\T_0$ is unramified. 
By Corollary \ref{zarathoustra},~the~re\-pre\-sentation
$\pi$ is dis\-tinguished if and only if any of the 
$\s$-selfdual types contained in $\pi$ is dis\-tin\-guished. 

\subsection{Existence of $\s$-selfdual extensions of the Heisenberg 
  representation}
\label{binzou1}

Let $[\aa,\b]$ be a maximal simple stratum as in Remark \ref{StandardStableType}. 
Write $\E=\F[\b]$
and suppose that it is unramified over 
$\E_0=\E^\s$.
Let us write $\BJ=\BJ(\aa,\b)$, $\J=\J(\aa,\b)$ and $\J^1=\J^1(\aa,\b)$. 

We may and will identify $\J/\J^1$ 
with the group $\GL_m(\ee)$, denoted $\GG$, equipped with the 
residual involution $\s\in\Gal(\ee/\ee_{0})$,
where $\ee$ and $\ee_0$ are the residue 
fields of $\E$ and $\E_0$, respectively. 

\begin{lemm}
\label{restonscalmes}
The group $\BJ\cap\G^\s$ is generated by $t$ and $\J\cap\G^\s$. 
\end{lemm}

\begin{proof}
Any $x\in\BJ$ can be written $x=t^my$ for unique $m\in\ZZ$ 
and $y\in\J$. 
Since $t$ is $\s$-invariant, we have $x\in\BJ\cap\G^\s$ if and only if 
$y\in\J\cap\G^\s$.
\end{proof}

\begin{lemm}
\label{PaulAlain}
Any character of $\BJ\cap\G^\s$ 
trivial on $\J^1\cap\G^\s$ extends to 
a character of $\BJ$ trivial on $\J^1$.
\end{lemm}

\begin{proof}
Let $\chi$ be a character of $\BJ\cap\G^\s$ trivial on $\J^1\cap\G^\s$.
Since $\J^1$ is a pro-$p$-group,
the first cohomology group of $\s$ in $\J^1$ is trivial.
The subgroup $\GG^\s$ thus identifies with 
$(\J\cap\G^\s)/(\J^1\cap\G^\s)$.
The restriction of $\chi$ to $\J\cap\G^\s$ thus induces a character of 
$\GG^\s$, which can be written $\a_0\circ\det$ for some 
character $\a_0$ of $\ee_{0}^\times$. 
Let $\a$ be a character of $\ee^\times$ extending $\a_0$,
and let $\phi$ be the character of $\J$ trivial on $\J^1$ 
inducing the character ${\a}\circ\det$ of $\GG$.
Since $\J=\bb^\times\J^1$, 
the element $t$ acts trivially on $\J/\J^1$ by 
conjugacy, thus normalizes $\phi$.
We thus may extend $\phi$ to $\BJ$ by setting 
$\phi(t)=\chi(t)$. Lem\-ma \ref{restonscalmes} implies 
that $\phi$ extends $\chi$ and it is trivial on $\J^1$. 
\end{proof}

Let $\t\in\Cc(\aa,\b)$ be a maximal simple character such that 
$\H^1(\aa,\b)$ is $\s$-stable and $\t\circ\s=\t^{-1}$.
Let $\n$ denote the Heisenberg representation of $\t$ on the 
group $\J^1$.

\begin{lemm}
\label{roots}
There is a $\s$-selfdual representation $\bk$ of $\BJ$ extending $\n$. 
\end{lemm}

\begin{proof}
Let $\bk$ be an irreducible representation of $\BJ$ extending $\n$.
By Lemma \ref{chouxfarci}, 
there is~a character $\xx$ of $\BJ$ trivial on 
$\J^1$ such that $\bk^{\s\vee}\simeq\bk\xx$
and $\xx\circ\s=\xx$.
We claim that there is~a character $\nu$ of $\BJ$ trivial on $\J^1$
such that $(\nu\circ\s)\nu=\xx$.
Indeed, if this is the case,
the representation $\bk\nu$ extends $\n$ and is $\s$-selfdual.

Consider first $\xx$ as a character of $\GG$ and 
write $\xx=\varphi\circ\det$ for some character $\varphi$ of $\ee^\times$.
Then we have $\varphi\circ\s=\varphi$,
thus there is a character $\a$ of $\ee^\times$ 
such that $(\a\circ\s)\a=\varphi$.
Choosing such a $\a$, there exists a unique character $\nu$ of $\J$ 
inducing $\a\circ\det$ on $\GG$.
Since $\BJ$ is gene\-ra\-ted by $t$ and $\J$,
it remains to extend $\nu$ to $\BJ$ by choosing a scalar 
$\nu(t)\in\FC^\times$ such 
that $\nu(t)^2=\xx(t)$.
\end{proof}

\subsection{Existence of distinguished extensions of the Heisenberg 
  representation}
\label{UR3}

Let $(\BJ,\bl)$ be a $\s$-selfdual type,
with associated simple character the character $\t$ 
of \S\ref{binzou1}.

In this paragraph, we suppose that $m$ is odd. 

\begin{prop}
\label{HoMcKay}
Suppose that $m$ is odd. 
There is a $\s$-selfdual distinguished 
representation $\bk$ of $\BJ$ exten\-ding $\n$. 
\end{prop}

\begin{proof}
We first assume that $\FC$ has characteristic $0$. 
By Lemma \ref{roots},
there is a $\s$-selfdual~repre\-sentation $\bk$ of $\BJ$ extending $\n$.
Let $\chi$ denote the character of $\BJ\cap\G^\s$ trivial on $\J^1\cap\G^\s$ 
associated to $\bk$ by Lemma \ref{passypassion}.
Since $m$ is odd, Lem\-ma \ref{L71} 
implies that $\GG$ possesses a 
$\s$-selfdual supercuspidal represen\-ta\-tion $\rho$.
Let $\bt$ be the unique re\-pre\-sentation of $\BJ$ trivial on $\J^1$ 
such that $t\in\Ker(\bt)$ and whose res\-triction to $\J$ is the inflation of 
$\rho$. 
This representation $\bt$ is $\s$-selfdual.
By Lemma \ref{L72}, it is also distinguished. 
Now let $\bl$ be the $\s$-selfdual type $\bk\otimes\bt$ on $\BJ$.
The natural~iso\-mor\-phism:
\begin{equation*}
\Hom_{\BJ\cap\G^\s}(\bk,\chi)\otimes\Hom_{\BJ\cap\G^\s}(\bt,1)
\to
\Hom_{\BJ\cap\G^\s}(\bl,\chi)
\end{equation*}
thus shows that $\bl$ is $\chi$-distinguished.

By Lemma \ref{PaulAlain}, there exists a character $\phi$ 
of $\BJ$ trivial on $\J^1$ extending $\chi$.
The representa\-tion $\bl'=\bl\phi^{-1}$ 
is thus a distinguished type.
Let $\pi'$ be the cuspidal re\-pre\-sen\-ta\-tion of $\G$ 
compactly induced from $(\BJ,\bl')$.
It is distinguished, thus $\s$-selfdual by Theorem \ref{FP}. 
Since $\bl'$ and $\bl'^{\s\vee}\simeq\bl'\phi(\phi\circ\s)$ are 
both contained in $\pi'$, it follows that $\phi(\phi\circ\s)$ is trivial.
This implies that $\bk'=\bk\phi^{-1}$ is both $\s$-selfdual and 
distinguished. 

Now assume that $\FC$ has characteristic $\ell>0$.
We then reduce to the characteristic $0$ case as in the proof of 
Lemma \ref{GuyLodrant5}.
\end{proof}

\begin{rema}
\label{chaussures2}
I don't know whether Proposition \ref{HoMcKay} holds when $m$ is even.
\end{rema}

\begin{coro}
\label{patinage}
\begin{enumerate}
\item
Any distinguished representation of $\BJ$ extending $\n$ 
is $\s$-selfdual.
\item
If $\ell=2$,
any $\s$-selfdual representation of $\BJ$ extending $\n$ is distinguished.
\end{enumerate}
\end{coro}

\begin{proof}
Let us fix a distinguished $\s$-selfdual representation $\bk$ of $\BJ$ 
extending $\n$ given by Propo\-si\-tion \ref{HoMcKay}.
Let $\bk'$ be a distinguished representation of $\BJ$ 
extending $\n$. 
Then $\bk'=\bk\phi$ for some character $\phi$ of $\BJ$ 
trivial on $(\BJ\cap\G^\s)\J^1$.
Thus $\phi(t)=1$ and $\phi$ induces the character $\a\circ\det$ 
on $\GG$, where $\a$ is a character of $\ee^\times$ trivial on 
$\ee_{0}^\times$, or equivalently $\a^{q_0+1}=1$.
Thus we have $\phi(\phi\circ\s)=1$.
This implies that $\bk'$ is $\s$-selfdual, which proves the first assertion. 

Now suppose that $\bk'$ is a $\s$-selfdual representation of $\BJ$ 
extending $\n$.
Then $\bk'=\bk\xi$ for some character $\xi$ of $\BJ$ such that 
$\xi(\xi\circ\s)$ is trivial.
Thus $\xi(t)\in\{-1,1\}$ and there is a character $\nu$ of $\ee^\times$ 
such that $\xi$ induces $\nu\circ\det$ 
on $\GG$ and $\nu^{q_0+1}=1$. 
It follows that $\xi$ is trivial on $(\J\cap\G^\s)\J^1$.
Thus, if $\ell=2$, 
the representation $\bk'$ is distinguished. 
\end{proof}

\begin{rema}
Let $\bk$ be a $\s$-selfdual representation of $\BJ$ extending $\n$.
Then the character $\chi$ of $\BJ\cap\G^\s$
associated to $\bk$ by Lemma \ref{passypassion} is quadratic and 
unramified.
\end{rema}

\subsection{Distinction criterion in the unramified case}
\label{guille}

Let $\pi$ be a $\s$-selfdual supercuspidal representation of 
$\G$.
Associated to it by Propo\-sition \ref{TT0canonique}, 
there is a quadratic extension $\T/\T_0$.
We assume that $\T$ is un\-ra\-mi\-fied over $\T_0$. 

Recall that,
by Theorem \ref{DSTT} and Proposition \ref{classesdetypesstables7}, 
the representation $\pi$ is distin\-guish\-ed if~and only any of its 
$\s$-selfdual types is dis\-tinguished. 
The following result is the analogue of Pro\-po\-si\-tion \ref{michelesolara}. 

\begin{prop}
\label{MAINTHM3nr}
Let $\pi$ be a $\s$-selfdual {supercuspidal} 
representation of $\G$, with unramified quadratic extension $\T/\T_0$
and relative degree $m$.
Let $(\BJ,\bl)$ be a $\s$-selfdual type in $\pi$.
Let $\J^1$ be the maxi\-mal nor\-mal pro-$p$-subgroup of $\BJ$
and $\n$ be an irre\-du\-cible com\-po\-nent of the res\-triction of $\bl$ 
to $\J^1$. 
\begin{enumerate}
\item
The integer $m$ is odd.
\item
\label{MAINTHM3nr.p2}
\label{Distkappa8}
There is a distinguished representation of $\BJ$ extending $\n$,
and any such extension of $\n$ is $\s$-self\-dual.
\item
\label{MAINTHM3nr.p3}
Let $\bk$ be a distinguished representation of $\BJ$ ex\-ten\-ding $\n$,
and let $\bt$ be the unique represen\-tation of $\BJ$ trivial on $\J^1$ 
such that $\bl\simeq\bk\otimes\bt$.
Then $\pi$ is distinguished if and only if $\bt$ is~dis\-tin\-guished.
\end{enumerate}
\end{prop}

\begin{proof}
We may and will assume that $\BJ=\BJ(\aa,\b)$ for some maximal simple 
stratum $[\aa,\b]$ as in Remark \ref{StandardStableType}. 
Following Remark \ref{Turron}, the extension $\E$ is unramified 
over $\E_0$.
We thus may apply the results of Paragraph \ref{binzou1}.

Let $\bk$ be a $\s$-selfdual representation of $\BJ$ extending $\n$, 
the existence of which is given by Lem\-ma \ref{roots}, 
and let $\bt$ be the ir\-re\-ducible representation of $\BJ$ trivial on $\J^1$ 
such that $\bl$ is isomorphic to $\bk\otimes\bt$.
Since $\bl$ and $\bk$ are $\s$-selfdual, $\bt$ is $\s$-selfdual.
Its restriction to $\J$ induces a cuspidal irre\-du\-cible 
representation of $\GL_{m}(\ee)$, denoted $\rho$. 
Since $\pi$ is supercuspidal, 
$\rho$ is also supercus\-pidal by Fact \ref{cordoba}.
Lemma \ref{L71} implies that $m$ is odd. 
We thus apply Proposition \ref{HoMcKay}, 
which gives us a $\s$-selfdual dis\-tin\-guished representation 
extending $\n$.

Part (2) of the proposition is given by Corollary \ref{patinage}.
For (3), 
it suffices to fix a distinguished representa\-tion $\bk$ of 
$\BJ$ extend\-ing $\n$ and to consider the canonical isomor\-phism:
\begin{equation*}
\Hom_{\BJ\cap\G^\s}(\bk,1)
\otimes\Hom_{\BJ\cap\G^\s}(\bt,1) \to \Hom_{\BJ\cap\G^\s}(\bl,1)
\end{equation*}
as in the ramified case.
\end{proof}

\begin{rema}
If one relaxes the supercuspidality assumption on $\pi$ 
(that is, we only assume $\pi$ to be $\s$-selfdual cuspidal with 
$\T/\T_0$ unramified), 
then its relative degree $m$ need not be odd, in 
which case our proof of Proposition \ref{MAINTHM3nr}\eqref{Distkappa8} 
doesn't apply 
(see Remarks \ref{chaussures} and \ref{chaussures2}).
Unlike the ramified case, 
I thus don't know whether there is a distinguished and $\s$-selfdual 
extension $\bk$ of $\n$ when $\pi$ is not supercuspidal.
\end{rema}

\begin{rema}
In both ramified and unramified cases, 
the distinguished representation $\bk$ of $\BJ$ extending $\n$ is not uni\-que 
in general, so neither is $\bt$.
If $\bk$ is a distinguished representation of $\BJ$ extending $\n$,
the other ones are exactly the 
$\bk\phi$ where $\phi$ ranges over the set of characters of $\BJ$ trivial on 
$(\BJ\cap\G^\s)\J^1$. 
\end{rema}

From now, we will thus assume that $\bk$ is a distinguished 
$\s$-selfdual re\-pre\-senta\-tion of $\BJ$ exten\-ding $\n$.
Proposition \ref{MAINTHM3nr} reduces the problem of the distinction of $\pi$ 
to that of $\bt$.
We now~inves\-ti\-gate the distinction of $\bt$. 

Let $\rho$ be the representation of $\GL_{m}(\ee)$ defined by restricting 
$\bt$ to $\J$. 
It is $\s$-selfdual. 
By Fact \ref{cordoba}, it is also super\-cuspidal.

Let $(\K/\E,\xi)$ be an admissible pair of level $0$ attached to $\bt$
in the sense of Definition \ref{defiL0AP1}.
Since $\bt$ is $\s$-selfdual, Proposition \ref{charaghdinS4} tells us that 
there is a unique involutive $\E_0$-automorphism of~$\K$,
which we denote by $\s$, 
which is non-trivial on $\E$ and satisfies $\xi\circ\s=\xi^{-1}$.
Let $\K_0$ be the $\s$-fixed points~of $\K$ and $\E_0=\K_0\cap\E$. 

\begin{lemm}
\label{jaunepoussin}
The representation $\bt$ is distinguished if and only if 
it is trivial on $\E_0^\times$.
\end{lemm}

\begin{proof}
Note that $\bt(x)=\xi(x)\cdot{\rm id}$ for all $x\in\E^\times$,
thus $\bt$ is trivial on $\E^\times_0$ if and only if $\xi$ is.
The represen\-tation $\rho$ is $\s$-selfdual, 
thus distinguished (see Lemma \ref{L72}).
We thus have:
\begin{equation*}
\Hom_{\BJ\cap\G^\s}(\bt,1) 
\subseteq \Hom_{\J\cap\G^\s}(\bt,1) 
\simeq \Hom_{\GL_m(\ee_0)}(\rho,1) 
\end{equation*}
where the space on the right hand side is non-zero (and has dimension $1$).
Since $\BJ\cap\G^\s$ is~ge\-ne\-ra\-ted by $\J\cap\G^\s$ and $t$, 
we deduce that $\bt$ is distinguished if and only if $t$ acts trivially on the 
space $\Hom_{\J\cap\G^\s}(\bt,1)$, 
that is, if and only if $\xi(t)$ is trivial.
The result follows from the fact that, since $\xi$ is $\s$-selfdual, 
it is trivial on the $\E/\E_0$-norms in $\E_0^\times$, 
thus on the units of $\E_0$.
\end{proof}

Let $\aaa_0$ denote the restriction of the character $\xi$ to $\T_0^\times$.

\begin{lemm}
\label{alfonso1}
The character $\aaa_0$ is quadratic and unramified.
\end{lemm}

\begin{proof}
As has been said in the proof of Lemma \ref{jaunepoussin},
the character $\xi$ is trivial on the subgroup of $\E/\E_0$-norms in $\E_0^\times$,
since $\bt$ is $\s$-selfdual.
Thus the restriction of $\xi$ to $\E_0^\times$ is either trivial or 
(if $\ell\neq2$)
equal to $\ep_{\E/\E_0}$. 
We get the expected result by restricting to $\T_0^\times$, 
since $\E$ is unramified over $\E_0$ and 
$e(\E_0/\T_0)$ is a $p$-power with $p$ odd. 
\end{proof}

We will see below (Remark \ref{effort})
that the character $\aaa_0$ is uniquely de\-ter\-mined by $\pi$.

\begin{theo}
\label{CNSDISTUNR}
Let $\pi$ be a $\s$-selfdual supercuspidal  
representation of $\G$. 
Suppose that $\T$ is unramified over $\T_0$.
\begin{enumerate}
\item
The representation $\pi$ is distinguished if and only if $\aaa_0$ is trivial.
\item
Suppose that the characteristic of $\FC$ is not $2$.
Then $\pi$ is $\ep$-distinguished if and only if $\aaa_0$ is non-trivial.
\end{enumerate}
\end{theo}

\begin{rema}
\label{12NR}
If $\FC$ has characteristic $2$, then $\pi$ is always distinguished. 
If $\FC$ has characteristic not $2$, then $\pi$ is either distinguished 
or $\ep$-distinguished, but not both.
\end{rema}

\begin{proof}
By Proposition \ref{MAINTHM3nr},
the representation $\pi$ is distinguished if and only if $\bt$ is distinguished. 
Lemma \ref{jaunepoussin} tells us that it is 
distinguished if and only if $\xi(t)=1$. 
The restriction of $\xi$ to $\E_0^\times$ is a qua\-dra\-tic unramified 
character. 
Since the ramification index of $\E_0$ over $\T_0$ is odd (for it is a
$p$-power), $\xi$ is trivial on $\E_0^\times$ if 
and only if it is trivial on $\T_0^\times$.
The first assertion is proven.

Now suppose that $\FC$ has characteristic different from $2$,
and let $\chi$ be an unramified character of $\F^\times$ extending $\ep$. 
Note that the twisted representation $\pi'=\pi(\chi^{-1}\circ\det)$ is 
supercuspidal and $\s$-selfdual, 
and that the character associated with $\pi'$ is 
$\aaa'_0=\aaa^{}_0(\chi^{-m}\circ\N_{\E/\F})|_{\T_0^\times}$
where $\N_{\E/\F}$ is the norm map from $\E$ to $\F$.
Thus:
\begin{eqnarray*}
\text{$\pi$ is $\ep$-distinguished} 
&\Leftrightarrow& 
\text{$\pi'$ is distinguished} \\
&\Leftrightarrow& 
\text{the character $\aaa'_0$ is trivial on $\T_0^\times$} \\
&\Leftrightarrow& 
\text{the character $\aaa_0$ coincides with 
$\chi^m\circ\N_{\E/\F}$ on $\T_0^{\times}$}.
\end{eqnarray*}
The restriction of $\chi\circ\N_{\E/\F}$ to $\T_0^\times$ is equal to 
$\ep\circ\N_{\T_0/\F_0}=\ep_{\T/\T_0}$ to the power of $[\E_0:\T_0]$,
which is a $p$-power. 
The second assertion then follows from the fact that $p$ and $m$ 
are odd.
\end{proof}

\begin{coro}
Let $\pi$ be a supercuspidal  representation of $\G$.
Suppose that $\T/\T_0$ is unra\-mified,
and that the ramification index of $\T/\F$ is odd.
Then $\pi$ is distinguished if and only if it is $\s$-selfdual and its central
character is trivial on $\F_0^\times$.
\end{coro}

\begin{proof}
Suppose that $\pi$ is $\s$-selfdual and that its central character
$\omega_\pi$ is trivial on $\F_0^\times$. 
By using a $\s$-selfdual type $(\BJ,\bl)$ contained in $\pi$ as above,
we can express $\omega_\pi$ as the product $\omega_{\bk}\omega_{\bt}$,
where $\omega_{\bk}$ and $\omega_{\bt}$ are the central characters of $\bk$ 
and $\bt$ on $\F^\times$, respectively. 
Since $\bk$ is distinguished, its central character is trivial 
on $\F_0^\times$, 
thus $\omega_\pi$ and $\aaa_0$ coincide on $\F_0^\times$. 
It remains to prove that $\aaa_0$ is trivial if and only if it is trivial on 
$\F_0^\times$. 

Suppose $\aaa_0$ is trivial on $\F_0^\times$. 
By Lemma \ref{alfonso1}, it is unramified, 
thus $\aaa_0^{e(\T_0/\F_0)}$ is trivial on $\T_0^\times$. 
Since $e(\T/\F)$ is odd, $e(\T_0/\F_0)$ is odd too,
and the expected result follows from the fact that $\aaa_0$ is quadratic. 
\end{proof}

\begin{rema}
\label{memoiresnr}
In particular, when $n$ is odd and $\F$ is unramified over $\F_0$, 
a supercuspidal re\-pre\-sentation of $\G$
is distinguished if and only if it is $\s$-selfdual and its 
central char\-ac\-ter is~tri\-vial on $\F_0^\times$. 
This has been proved by Prasad \cite{PrasadDMJ01} 
when $\FC$ has characteristic $0$.
Note that, since $m$ and $p$ are odd here, 
$n$ is odd if and only if $[\T:\F]$ is odd.
\end{rema}

\begin{rema}
\label{refrouge}
Note that, in the proof of Prasad \cite{PrasadDMJ01} Theorem 4, 
the isomorphism of $\pi$ with $\pi^{\s\vee}$ gives an element $g\in\G$
which has the property that $g\s(g)\in\J=\J(\aa,\b)$, 
but $g$ has \textit{a priori} no reason to normalize $\J$. 
Anyway, $g$ can be chosen in the maximal compact open subgroup 
$\aa^\times$ which contains $\J$
(which derives from \cite{BK} Theorem 3.5.11),
thus the group generated by $g$ and $\J$ will indeed be 
compact mod centre.
\end{rema}

\section{Statement of the final results}
\label{S7}

In this section we put together the main results of Sections 
\ref{SECDISTTR} to \ref{SECDIST}.
Let $\pi$ be a $\s$-selfdual~super\-cuspidal representation of $\G$.
Associated to it,
there are its relative degree $m$ and the quadratic extension $\T/\T_0$.
It is convenient to introduce the following definition,
which comes from \cite{AKMSS}.

\begin{defi}
\label{carioca}
A $\s$-selfdual type in $\pi$ is said to be 
\textit{generic} if either $\T/\T_0$ is unramified, 
or $\T/\T_0$ is ramified and this type has index $\lfloor m/2\rfloor$. 
\end{defi}

\begin{rema}
\label{cariocarema}
It is proved in \cite{AKMSS} {Proposition 5.5} 
that a $\s$-selfdual type is generic 
in the~sen\-se of Definition \ref{carioca} if and only if there are a 
$\s$-stable maximal unipotent subgroup $\N$ of $\G$ and a $\s$-self\-dual 
non-degenerate character $\psi_\N$ of $\N$ such that 
$\Hom_{\BJ\cap\N}(\bl,\psi_\N)$
is non-zero.
\end{rema}

Definition \ref{carioca} is convenient to us because of the following result, 
which subsumes Propositions \ref{MAINTHM3r} and \ref{flipflaplagirafe} 
(compare with Theo\-rem \ref{DSTT}).

\begin{theo}
\label{genericDSTT}
A $\s$-selfdual cuspidal representation of $\G$ is dis\-tin\-guished~if and 
only if~any of its generic $\s$-selfdual types is dis\-tin\-guished.
\end{theo}

Let $(\BJ,\bl)$ be a \textit{generic} $\s$-selfdual type contained in $\pi$.
Let $[\aa,\b]$ be a $\s$-selfdual simple~stra\-tum such that 
$\BJ=\BJ(\aa,\b)$. 
The restriction of $\bl$ to the maxi\-mal normal pro-$p$-subgroup
$\J^1$~is made of copies of a single irreducible representation $\n$. 
We~fix a distinguished $\s$-selfdual representation~$\bk$ of $\BJ$ 
extending $\n$,
the existence of which is given~by~Pro\-po\-sitions \ref{michelesolara} and 
\ref{MAINTHM3nr}. 
Let $\bt$ be the re\-pre\-sentation of $\BJ$ trivial on $\J^1$ such that 
$\bl$ is isomorphic to $\bk\otimes\bt$. 
Let $(\K/\E,\xi)$ be a admissible pair of level $0$ 
attached to $\bt$ and $\s$ be the involution
of $\K$ given by Pro\-position \ref{charaghdinS4}.
Let $\K_0$ be the field of $\s$-fixed points of $\K$.
We thus have $\K\simeq\K_0\otimes_{\F_0}\F$.

\begin{defi}
Let $\AP$ be the maximal tamely ramified sub-extension of $\K/\F$.
Write $\AP_0=\AP\cap\K_0$,
and let $\ap_0$ be the restriction of $\xi$ to $\D_0^\times$.
\end{defi}

It follows immediately from the definition that 
$\AP_0/\F_0$ is tamely ramified 
and the character $\ap_0$ is qua\-dratic, 
either trivial or (if $\ell\neq2$) 
equal to $\ep_{\AP/\AP_0}$.

\begin{prop}
\label{rummo}
The quadratic extension $\AP/\AP_0$ and the character $\ap_0$
are uniquely deter\-mined by $\pi$ up to $\F_0$-equivalence. 
That is, 
if $\AP'/\AP'_0$ and $\ap'_0$ are another quadratic~exten\-sion~and character 
associated to $\pi$,
then there is an $\F_0$-isomorphism $\h:\AP\to\AP'$ such that
$\h(\AP_0)=\AP'_0$ and $\ap_0^{}=\ap'_0\circ\h$.
\end{prop}

\begin{proof}
Start with a generic $\s$-selfdual type contained in $\pi$. 
Since it is unique up to $\G^\s$-con\-ju\-ga\-cy, 
we may assu\-me this is $(\BJ,\bl)$. 
Fix a $\s$-selfdual stra\-tum $[\aa',\b']$ 
such that $\BJ=\BJ(\aa',\b')$.
By~\cite{AKMSS} {Lemma 4.29}, 
we may assume that the maximal tamely ramified sub-extension of 
$\E'=\F[\b']$ over $\F$ is equal to $\T$. 
Fix a distinguished $\s$-selfdual 
represen\-ta\-tion $\bk'$ of $\BJ$ extending $\n$, 
let $\bt'$ be the represen\-tation of $\BJ$ trivial on $\J^1$ 
corresponding to this choice
and $(\K'/\E',\xi')$ be an admissible~ pair of level $0$ attached to $\bt'$.
This gives us a quadratic extension $\AP'/\AP'_0$ and a character $\ap'_0$ of 
$\AP_0'^\times$.

First, suppose that $[\aa',\b']=[\aa,\b]$ and $\K'=\K$.
We have $\bk'=\bk\phi$ for some character $\phi$ of $\BJ$ trivial on 
$(\BJ\cap\G^\s)\J^{1}$, thus $\bt'$ is isomorphic to $\bt\phi^{-1}$. 
Thus $\xi'$ is $\E$-isomorphic to $\xi\a^{-1}$ for some tamely ramified 
character $\a$ of $\K^\times$ trivial on $\K_0^\times$.
Restricting to $\AP_0$, we get $\ap'_0=\ap^{}_0$.

We now go back to the general case. 
By the previous argument, we may assume that $\bk'=\bk$,
thus $\bt'=\bt$.
Since $\AP$ and $\AP'$ are both unramified of same degree $m$ over $\T$, 
they are $\T$-isomorphic.
Let us fix a $\T$-isomorphism $f:\AP\to\AP'$.
Write $\s'$ for the involutive automorphism of $\K'$ given by Pro\-position 
\ref{charaghdinS4}. 

\begin{lemm}
\label{DDstrcuture}
We have $\s'\circ f=f\circ\s$.
\end{lemm}

\begin{proof}
Let us identify the residual fields of $\K$ and $\AP$, 
denoted~$\ll$, and those of $\E$ and $\T$, denoted $\ee$.
Note that, if $\h$ is any $\T$-automorphism of $\AP$, 
then it commutes with $\s$ since $\h$ and $\s\circ\h\circ\s^{-1}$ 
are both in $\Gal(\AP/\T)$ and have the same image in $\Gal(\ll/\ee)$.

We now consider the pair $(\AP/\T,\xi|_{\AP^\times})$. 
Since $\AP/\T$ is unramified of degree $m$,
it is admissible~of level~$0$.
Moreover, the $\ee$-regular character of $\ll^\times$ it induces 
is $\Gal(\ll/\ee)$-conjugate to the one~indu\-ced by $(\K/\E,\xi)$, 
which~doesn't depend on the identifications of residual fields 
we have made.~We 
have a similar result for $(\AP/\T,\xi'\circ f)$ and $(\K'/\E',\xi')$. 
Since $\bt'=\bt$ we deduce that $\xi'\circ f=\xi\circ\h$ for some 
$\h\in\Gal(\AP/\T)$.
Let $\a$ be the $\T$-automorphism $\s'\circ f\circ\s^{-1}\circ f^{-1}$ of 
$\AP'$. 
We have:
\begin{equation*}
\xi'\circ\a 
= \xi'^{-1}\circ f\circ\s^{-1}\circ f^{-1}
= \xi^{-1}\circ\h\circ\s^{-1}\circ f^{-1}
= \xi\circ\h\circ f^{-1}
= \xi'.
\end{equation*}
It follows from admissibility of $\xi'$ that $\a$ is trivial, 
as expected.
\end{proof}

Lemma \ref{DDstrcuture} implies that 
$\AP_0^{}$ and $\AP'_0$ are $\T_0$-isomorphic.
We thus~now may assume that $\AP=\AP'$ and $\AP^{}_0=\AP'_0$,
thus $\K$, $\K'$ have the same maximal unrami\-fied sub-extension 
$\AP$ over $\T$ and there is an automorphism $\h\in\Gal(\AP/\T)$ 
such that $\xi'(x)=\xi\circ\h(x)$ for all $x\in\AP^\times$.
Restricting to $\AP_0^\times$, we deduce that $\ap'_0=\ap^{}_0$.
\end{proof}

\begin{rema}
\label{effort}
In particular, the character $\aaa_0$ of Paragraph \ref{guille},
which is the restriction of~$\ap_0$ to $\T_0^\times$, 
is uniquely determined by $\pi$.
\end{rema}

We state the dichotomy and disjunction theorem.

\begin{theo}
\label{MAINTH1DD}
Let $\pi$ be a $\s$-selfdual supercuspidal  representation of $\G$. 
Let $\ell$ be the charac\-te\-ris\-tic of $\FC$.
\begin{enumerate}
\item
If $\ell\neq2$, then $\pi$ is either distinguished or $\ep$-distinguished, 
but not both.
\item
If $\ell=2$, then $\pi$ is always distinguished. 
\end{enumerate}
\end{theo}

\begin{proof}
See Remarks \ref{12TR} and \ref{12NR}.
\end{proof}

We now state the distinction criterion theorem.

\begin{theo}
\label{MAINTH1}
Let $\pi$ be a $\s$-selfdual supercuspidal representation of $\G$. 
Attached to it,~there are 
the quadratic extensions $\T/\T_0$ and $\AP/\AP_0$
and the character $\ap_0$.
\begin{enumerate}
\item
Suppose that $n$ is odd.
Then $\pi$ is distinguished if and only if its central character 
is trivial on $\F_0^\times$.
\item
If $\ell\neq2$, 
$\T/\T_0$ is ramified and $\AP/\AP_0$ is unramified,
then $\pi$ is distinguished if and only if the character $\ap_0$ is non-trivial. 
\item
Otherwise, $\pi$ is distinguished if and only if $\ap_0$ is trivial.
\end{enumerate}
\end{theo}

\begin{proof}
Item 1 is an immediate consequence of Theorem \ref{MAINTH1DD}
as explained in \S\ref{guderianintro}.
If $\ell\neq2$, a $\s$-selfdual supercuspidal representation $\pi$
is either distinguished or $\ep$-distinguished.
In the~lat\-ter case, the restriction of its central character
to $\F_0^\times$ is $\ep^n$, which is trivial if and only if $n$ is even.
See also Remarks \ref{memoiresdunejeunefillerangee} and \ref{memoiresnr}.

For the remaining items, 
see Theorems \ref{OttoDix} and \ref{CNSDISTUNR}:
it suf\-fices to check that, if $\T/\T_0$ is~un\-ramified,
then $\ap_0$ is tri\-vial if and only if its restriction $\aaa_0$ to $\T_0^\times$ 
is trivial, which follows from the fact that $m$ is odd in that case.
\end{proof}

\begin{rema}
\label{MAINTH1RMK}
The following conditions are equivalent:
\begin{enumerate}
\item
$\AP/\AP_0$ is ramified;
\item
$\T/\T_0$ is ramified and $m=1$.
\end{enumerate}
Indeed this follows from Remark \ref{Turron}, Lemma \ref{knr} and 
Proposition \ref{flipflaplagirafe}.
The following conditions are thus also equivalent:
\begin{enumerate}
\item
$\T/\T_0$ is ramified and $\AP/\AP_0$ is unramified;
\item
$\F/\F_0$ is ramified, $\T_0/\F_0$ has odd ramification order 
and $\AP/\AP_0$ is unramified;
\item
$\F/\F_0$ is ramified, $\T_0/\F_0$ has odd ramification order 
and $m\neq1$.
\end{enumerate}
\end{rema}

We now state the distinguished lift theorem. 
For the notion of the reduction mod $\ell$ of~an~in\-te\-gral irreducible 
$\qlb$-representation of $\G$, we refer to \cite{Vigb,Vigw}. 
If $\pi$ is an irreducible $\flb$-repre\-sen\-ta\-tion of $\G$, 
we say an integral irreducible $\qlb$-representation of $\G$ 
is a \textit{lift} of $\pi$ if its reduction mod $\ell$ is irreducible and 
isomorphic to $\pi$. 

\begin{theo}
\label{MAINTH2}
Let $\pi$ be a $\s$-selfdual supercuspidal $\flb$-representation of $\G$. 
\begin{enumerate}
\item
The representation $\pi$ admits a $\s$-selfdual supercuspidal lift to $\qlb$.
\item
Let $\widetilde{\pi}$ be a $\s$-selfdual lift of $\pi$,
and suppose that $\ell\neq2$.
Then $\widetilde{\pi}$ is distinguished if and only if 
$\pi$ is distinguished.
\end{enumerate}
\end{theo}

\begin{proof}
Let $(\BJ,\bl)$ be a $\s$-selfdual type in $\pi$.
Let $\n$ be the Heisenberg representation contained in the restriction of
$\bl$ to $\J^1$, with associated simple character $\t$. 
As in the proof of Lemma \ref{GuyLodrant5}, 
let $\widetilde{\t}$ be the lift of $\t$ with values in $\qlb$
and $\widetilde{\n}$ be the associated Heisenberg 
representation, whose reduction mod $\ell$ is isomorphic to $\n$.
Pro\-positions \ref{michelesolara} and \ref{HoMcKay} tell us that 
there is a distinguished $\s$-selfdual representation $\widetilde{\bk}$ of 
$\BJ$ which extends $\widetilde{\n}$.
Its reduction mod $\ell$, denoted $\bk$,
is a $\s$-selfdual representation of $\BJ$ extending~$\n$,
and it is distinguished thanks to Lemma \ref{aimee}.
(Note that, as in the proof of Lemma \ref{jaunepoussinr},
the fact that $\bk$ is $\s$-selfdual implies that it has finite image;
it thus can be considered as a representation of a finite group.)
Let $\bt$ be the irreducible representation of~$\BJ$ trivial on $\J^1$
such that $\bl$ is isomorphic to $\bk\otimes\bt$.
It is $\s$-selfdual.

The representation $\bt$ admits a $\s$-selfdual 
$\qlb$-lift $\widetilde{\bt}$ on $\BJ$ trivial on $\J^1$.
Indeed, let $(\K/\E,\xi)$~be an admissible pair of level $0$ attached to 
$\bt$.
Then, as in the proof of Lemma \ref{jaunepoussinr}, 
the canonical $\qlb$-lift $\widetilde{\xi}$ of $\xi$
defines a pair $(\K/\E,\widetilde{\xi})$ which is admissible of level $0$,
and the $\qlb$-representation~$\widetilde{\bt}$ of $\BJ$ trivial on $\J^1$ 
which is attached to it is both $\s$-selfdual and a lift of $\bt$.
The representation $\widetilde{\bk}\otimes\widetilde{\bt}$ is thus 
a $\s$-self\-dual $\ell$-adic type whose reduction mod $\ell$ is $\bl$.
Inducing $\widetilde{\bk}\otimes\widetilde{\bt}$ to $\G$, 
we get a $\s$-selfdual supercuspidal lift $\widetilde{\pi}$ of
$\pi$. 

Suppose that $\ell\neq2$ and let $\widetilde{\ep}$ be the canonical
$\ell$-adic lift of $\ep$,
that is, the $\qlb$-character of $\F_0^\times$ of kernel 
$\N_{\F/\F_0}(\F^\times)$. 
By Theorem \ref{MAINTH1DD}, 
since the representation $\widetilde{\pi}$ is $\s$-selfdual,
it is either distinguished or $\widetilde{\ep}$-distinguished,
but not both. 
Using the reduction argument of invariant linear forms as in Lemma ~\ref{aimee}, 
we see that if $\widetilde{\pi}$ is distinguished
(respecti\-ve\-ly, $\widetilde{\ep}$-distinguished),
then $\pi$ is distinguished
(respecti\-ve\-ly, ${\ep}$-distinguished).
By Theorem \ref{MAINTH1DD} applied to $\pi$, 
this is an~equi\-valence. 
\end{proof}

We end with the following result,
which is useful in \cite{AKMSS}.

\begin{prop}
Suppose that $\pi$ is a distinguished supercuspidal  representation 
of $\G$ and that $\ell\neq2$.
Then $\pi$ has no $\ep$-distinguished unramified twist if and only if $\AP/\AP_0$ 
is ramified, that is, if and only if $\T/\T_0$ is ramified and $m=1$. 
\end{prop}

\begin{proof}
Consider an unramified twist $\pi'=\pi(\chi\circ\det)$ of $\pi$, 
where $\chi$ is an unramified character of $\F^\times$.
We are looking for a $\chi$ such that $\pi'$ is $\ep$-distinguished. 
First, $\pi'$ is $\s$-selfdual if and only if $\pi\chi(\chi\circ\s)\simeq\pi$, 
that is, if and only if: 
\begin{equation}
\label{YUVAL0}
(\chi(\chi\circ\s))^{t(\pi)} = \chi^{2t(\pi)} = 1
\end{equation}
where $t(\pi)$ denotes the torsion number of $\pi$,
that is the number of unramified characters $\a$ of~$\G$ such that 
$\pi\a\simeq\pi$.
By \cite{MSt} \S3.4, we have $t(\pi)=f(\K/\F)=f(\AP/\F)$.
Now the quadratic character associated with $\pi'$ is 
$\ap_0'=\ap_0^{}(\chi\circ\N_{\K/\F})|_{\AP_0^\times}$ and we have:
\begin{equation}
\label{YUVAL1}
(\chi\circ\N_{\K/\F})|_{\AP_0^\times} 
= (\chi\circ\N_{\AP_0/\F_0})^{[\K:\AP]}.
\end{equation}
By Theorem \ref{MAINTH2},
the representation $\pi'$ is $\ep$-distinguished if and only if 
the character
\eqref{YUVAL1} is equal to $\ep_{\AP/\AP_0}$.
If $\AP$ is ramified over $\AP_0$, this is not possible since 
$\chi$ is unramified.
If $\AP/\AP_0$ is unramified,
choosing an unramified character $\chi$ of order $f(\AP_0/\F_0)$ 
gives us an $\ep$-distinguished twist $\pi'$. 
\end{proof}

\providecommand{\bysame}{\leavevmode ---\ }

\end{document}